\documentclass[11pt]{amsart}
\usepackage[leqno]{amsmath}
\usepackage{amsfonts, amssymb,srcltx, amsopn, color}

\usepackage[linktocpage=true]{hyperref}
\usepackage[inkscapelatex=false]{svg}

\usepackage[all]{xy}
\usepackage{pb-diagram, pb-xy}
\usepackage{overpic}

\usepackage{graphicx}

\usepackage{epsfig}
\usepackage{color}
\usepackage{amsmath}
\usepackage{amssymb}
\newtheorem{lemma}{Lemma}
\newtheorem{lem}{Lemma}[section]
\newtheorem{prop}{Proposition}[section]
\newtheorem{proposition}{Proposition}
\newtheorem{theorem}{Theorem}
\newtheorem{thmi}{Theorem}
\newtheorem{defn}{Definition}[section]

\newtheorem{corollary}{Corollary}

\theoremstyle{definition}
\newtheorem{definition}{Definition}

\newtheorem{construction}{Construction}
\newtheorem{remark}{Remark}

\usepackage{color}
\newcommand{\commentout}[1]{}
\usepackage{graphicx}

\DeclareMathOperator{\reals}{\mathbb R}
\newcommand{\coll}{\;\;\makebox[0pt]{$\bot$}\makebox[0pt]{$\smile$}\;\;}
\newcommand{\xing}[1]{{\Gamma_{\#}({\bf #1})}}
\DeclareMathOperator{\re}{\bf R}
\DeclareMathOperator{\pf}{\mathcal P\mathcal F}

\DeclareMathOperator{\dimension}{dim}
\DeclareMathOperator{\integers}{\mathbb Z}
\DeclareMathOperator{\naturals}{\mathbb N}

\newcommand{\recube}[1]{\mathfrak R(\bf{#1})}

\newcommand{\link}{\mathrm{Lk}}

\begin{document}
\title{On embeddings of CAT(0) cube complexes into products of trees via colouring their hyperplanes}
\maketitle
\thispagestyle{empty}

\vspace{5mm}

\centerline{{\sc Victor Chepoi$^{\small 1}$} and {\sc Mark F. Hagen}$^{\small 2}$ }

\vspace{3mm}

\medskip
\centerline{$^{1}$Laboratoire d'Informatique Fondamentale,}
\centerline{Universit\'e d'Aix-Marseille,}
\centerline{Facult\'e des Sciences de Luminy,} \centerline{F-13288
Marseille Cedex 9, France} \centerline{chepoi@lif.univ-mrs.fr}

\bigskip
\centerline{$^{2}$Department of Mathematics}
\centerline{University of Michigan, Ann Arbor,}
\centerline{Michigan, USA}
\centerline{markfhagen@gmail.com}

\bigskip\bigskip\noindent
{\footnotesize {\bf Abstract.}  We prove that the contact graph of a 2-dimensional CAT(0) cube complex ${\bf X}$ of maximum degree $\Delta$ can be coloured with at most $\epsilon(\Delta)=M\Delta^{26}$ colours, for a fixed constant $M$. This implies that ${\bf X}$ (and the associated median graph) isometrically embeds in the Cartesian product of at most $\epsilon(\Delta)$ trees, and that the event structure whose domain is ${\bf X}$ admits a nice labeling with $\epsilon(\Delta)$ labels. On the other hand, we present an example of a 5-dimensional CAT(0) cube complex with uniformly bounded degrees of 0-cubes which cannot be embedded into a Cartesian product of a finite number of trees. This answers in the negative a question raised independently by F. Haglund, G. Niblo, M. Sageev, and the first author of this paper.}


\section{Introduction}
In his seminal paper \cite{GromovHyperbolic}, among many other results, Gromov gave a nice combinatorial characterization of CAT(0) cube complexes as simply connected cube complexes in which the links of 0-cubes are simplicial flag complexes. Subsequently,  Sageev~\cite{Sageev95} introduced and investigated the concept of (combinatorial) hyperplanes of CAT(0) cube complexes, showing in particular that each hyperplane is itself a CAT(0) cube complex and divides the complex into two CAT(0) cube complexes.

These two results identify CAT(0) cube complexes as the basic objects in a ``high-dimensional Bass-Serre theory'', and CAT(0) and nonpositively-curved cube complexes have thus been studied extensively in geometric group theory.  For instance, many well-known classes of groups are known to act nicely on CAT(0) cube complexes (see, e.g.~\cite{CharneyDavis,CharneyDavis2,Farley,NibloReeves,WiseSmallCancellation}).  Groups acting essentially on CAT(0) cube complexes enjoy a wide variety of properties resulting from such actions -- they do not have Kazhdan's property (T)~\cite{NibloRoller} and many of them admit splittings or virtual splittings related to the hyperplanes (e.g.~\cite{Sageev97,NibloSplittingObstruction}), for example.  CAT(0) cube complexes whose hyperplanes are related to group splittings also lie at the heart of Wise's recent work on groups with a quasiconvex hierarchy~\cite{WiseIsraelHierarchy} and the closely-related recent resolution of the virtual Haken conjecture by Agol~\cite{AgolVirtualHaken}.

On the other hand, \cite{ChepoiMedian,Roller} established that the 1-skeleta of CAT(0) cube complexes are exactly the median graphs, i.e., the graphs in which any triplet of vertices admit a unique median vertex. Median graphs and related median structures have been investigated in several contexts by quite a number of authors for more than half a century. They  have many nice properties and admit numerous characterizations relating them to other discrete structures. Avann \cite{Avann} showed that median graphs and discrete median algebras (i.e., ternary algebras which are subdirect products of the two-element algebra $\{ 0,1\}$ \cite{BandeltHedlikova,Isbell}) constitute the same objects. Bandelt \cite{Bandelt_retracts} proved that median graphs are exactly the retracts of the hypercubes.  Barth\'elemy and Constantin \cite{Barthelemy_Constantin} showed that pointed median graphs are exactly the domains of event structures with binary conflict (investigated  in computer science in concurrency theory \cite{NiPlWi,WiNi, RoTh}), while Schaefer \cite{Scha} proved that median-stable subsets of Boolean algebras are exactly the solution sets of instances of the 2-SAT problem (a well-known problem in theoretical computer science). Mulder and Schrijver  \cite{MulderSchrijver} characterized the split systems (bipartitions) arising from halfspaces of median graphs as the Helly copair hypergraphs, thus extending the bijection of Buneman between trees and pairwise laminar (compatible) split systems. Due to this bijection for trees, Dress, Huber, and Moulton \cite{DrHuMo}
called {\it Buneman complexes} the median closures of arbitrary collections of splits of a finite set.  For other results and characterizations, see the books \cite{Feder,ImKl,Mulder_book,vandeVel_book} and  the
surveys \cite{BandeltChepoi_survey,BandeltHedlikova}.

All CAT(0) cube complexes ${\bf X}$ and median graphs -- the 1-skeleta $G({\bf X})$ of $\bf X$ -- are intimately related to hypercubes: they are constituted of cubes and themselves embed isometrically into hypercubes.  The minimum dimension of a hypercube into which $G({\bf X})$ (or $\bf X$) isometrically embeds equals the number of hyperplanes of ${\bf X}$, or, equivalently, the number of equivalence classes of the transitive closure of the ``opposite" relation of edges of $G({\bf X})$ on 2-cubes of ${\bf X}$, or, equivalently, the number of convex splits of $G({\bf X})$. While the dimension of the smallest hypercube into which the
median graph $G({\bf X})$ embeds is easy to determine, the problem of determining the least number $\tau({\bf X})=\tau(G({\bf X}))$ of tree factors necessary for an isometric embedding of the 1-skeleton of $\bf X$ into a Cartesian product of trees is hard.

There is a canonical construction of median graphs and CAT(0) cube complexes beginning from arbitrary graphs $G$: namely, for a graph $G$ the {\it simplex graph} $\kappa (G)$ has the simplices (the complete subgraphs) of $G$ as its vertices and pairs of (comparable) simplices differing in exactly one  vertex as its edges.  The graph  $\kappa (G)$ is median, moreover it was shown in \cite{BavdV_embedding} that $\kappa (G)$ can be isometrically embedded into the Cartesian product of at most $k$ trees if and only if the chromatic number $\chi(G)$ of $G$ is at most $k.$ In particular, it is NP-complete, even for $k=3$, to decide whether $\tau({\bf X})\le k$ for a 2-dimensional CAT(0) cube complex (i.e. if a 3-cube-free cube complex embeds into the Cartesian product of $k$ trees) \cite{BavdV_embedding}. Departing from triangle-free Mycielski graphs $G$ (i.e., graphs with arbitrarily high chromatic numbers), one gets, via the simplex-graph construction, 3-cube-free median graphs $\kappa (G)$ with arbitrarily large $\tau(\kappa (G)).$

For arbitrary CAT(0) cube complexes $\bf X$, the value $\tau({\bf X})$ is closely related to the chromatic number of the so-called {\it incompatibility} or {\it crossing graph} $\Gamma_{\#}({\bf X})$ of $\bf X$. $\Gamma_{\#}({\bf X})$ can be viewed as the intersection graph of the hyperplanes of $\bf X$: its vertices are the hyperplanes of $\bf X$ sensu \cite{Sageev95} and two hyperplanes are adjacent in $\Gamma_{\#}({\bf X})$ if and only if they cross (or, equivalently, they intersect). The crossing graph of the CAT(0) cube complex derived from the simplex graph $\kappa(G)$ of $G$ coincides with $G$ (see, e.g.~\cite{HagenQuasiArb,Roller}).

Extending the fact that $\tau(\kappa(G))=\chi(G),$ it was formally stated in \cite{BandeltChepoiEppstein} (and seems to be known to other people working in the field) that the equality  $\tau({\bf X})=\chi(\Gamma_{\#}({\bf X}))$ holds for all CAT(0) cube complexes $\bf X$. Since an arbitrary simplicial graph can be realized as the crossing graph of a CAT(0) cube complex $\bf X$, in order to better capture the structure of $\bf X$ and some graph-parameters of its 1-skeleton $G({\bf X}),$ the second author of this paper introduced in \cite{HagenQuasiArb} the concept of the {\it contact graph} $\Gamma({\bf X})$ of $\bf X$: the vertices of $\Gamma({\bf X})$
are the hyperplanes of $\bf X$ and two hyperplanes are adjacent in $\Gamma({\bf X})$ if and only if they cross or osculate (i.e., their carriers  touch each other). $\Gamma({\bf X})$ can be also viewed as  the intersection graph of the carriers of the hyperplanes of $\bf X$. The clique number $\omega(\Gamma({\bf X}))$ of the contact graph of $\bf X$ is exactly the maximum degree in $G({\bf X})$ of a 0-cube of $\bf X$, i.e., to the maximum number of 1-cubes incident to a 0-cube of $\bf X$. The contact graph $\Gamma({\bf X})$ always contains the crossing graph $\Gamma_{\#}({\bf X})$ as a spanning subgraph. $\Gamma({\bf X})$ also hosts the {\it pointed contact graph} $\Gamma_{\alpha}({\bf X})$ of the 1-skeleton $G_{\alpha}({\bf X})$ of $\bf X$ pointed at arbitrary vertex $\alpha$. The graph $\Gamma_{\alpha}({\bf X})$ has hyperplanes of $\bf X$ as vertices and two hyperplanes $H,H'$ are adjacent in $\Gamma_{\alpha}({\bf X})$ if and only if they are adjacent in $\Gamma({\bf X})$ and two incident 1-cubes, one crossed by $H$ and another crossed by $H'$, are directed away from the common origin.

Pairwise-independently, F. Haglund, G. Niblo, M. Sageev, and the first author of this paper asked the following question:

\medskip\noindent
{\bf Question 1.} {\it Is it true that all CAT(0) cube complexes $\bf X$ with uniformly bounded degrees can be isometrically embedded into a finite number of trees, or, equivalently, if there exists a function $\epsilon: \mathbb{N} \mapsto \mathbb{N}$ such that $\tau({\bf X})\le \epsilon(\Delta)$ for any CAT(0) cube complex $\bf X$ of degree $\Delta$?}

\medskip\noindent
This question is closely related with the conjecture of Rozoy and Thiagarajan \cite{RoTh} (also called the \emph{nice labeling problem}) asserting that:

\medskip\noindent
{\bf Question 2.} {\it Any event structure with finite (out)degree admits a labeling with a finite number of labels.}

\medskip\noindent
As noted above, pointed  median graphs are exactly the domains of event structures with binary conflict \cite{Barthelemy_Constantin}. Then, in view of the bijection between median graphs and 1-skeleta of CAT(0) cube complexes, the nice labeling problem for such event structures can be equivalently viewed as the colouring problem of the pointed contact graph $\Gamma_{\alpha}({\bf X})$ of the CAT(0) cube complex $\bf X$ associated to the domain of the event structure.  Since $\chi(\Gamma_{\alpha}({\bf X}))\le \chi(\Gamma({\bf X}))$ and $\chi(\Gamma_{\#}({\bf X}))\le \chi(\Gamma({\bf X})),$ in relation with Questions 1 and 2, the following question is natural:

\medskip\noindent
{\bf Question 3.} {\it Is it true that the chromatic number $\chi(\Gamma({\bf X}))$ of the contact graph of a CAT(0) cube complex $\bf X$ of degree $\Delta$ can be bounded by a function $\epsilon$ of $\Delta$?}

\medskip\noindent
Since $\omega(\Gamma({\bf X}))=\Delta$ and $\Gamma_{\#}({\bf X}),\Gamma_{\alpha}({\bf X})$ are subgraphs of $\Gamma({\bf X}),$ all three questions can be reformulated, namely: which of the classes of graphs $\Gamma_{\#}({\bf X}), \Gamma_{\alpha}({\bf X}),$ and $\Gamma({\bf X})$ are $\chi$-bounded? A class $\mathcal C$ of graphs is called $\chi$-{\it bounded} if there exists a function $f$ such that $\chi(G)\le f(\omega(G))$ for any graph $G$ of $\mathcal C$.  For example, Asplund and Gr\"{u}nbaum \cite{AsplundGrunbaum} proved that the intersection graphs of axis-parallel rectangles of ${\mathbb R}^2$ are $\chi$-bounded (we will review below some other classes of $\chi$-bounded intersection graphs).

Via a series of nontrivial examples, Burling \cite{Burling} showed that the class of intersection graphs of axis-parallel boxes of ${\mathbb R}^3$ is not $\chi$-bounded. Based on Burling's examples, it was recently shown in \cite{ChepoiNiceLabeling} that for CAT(0) cube complexes  the classes of graphs $\Gamma({\bf X})$ and $\Gamma_{\alpha}({\bf X})$ are not $\chi$-bounded, thus disproving the nice labeling conjecture of \cite{RoTh}. In this paper, we adapt this counterexample
by using the recubulation technique from~\cite{HagenQuasiArb} to show that the class of crossing graphs $\Gamma_{\#}({\bf X})$ of CAT(0) cube complexes is also not $\chi$-bounded, thus answering in the negative the first open question.

On the other hand, and this is the main contribution of our paper, we show that in the case of  2-dimensional CAT(0) cube complexes $\bf X$ the contact graphs ${\Gamma}({\bf X})$ (and therefore the crossing and the pointed contact graphs) are $\chi$-bounded by a polynomial function in $\omega({\Gamma}({\bf X}))=\Delta$, thus showing that in the 2-dimensional case the three questions have positive answers; this is the content of our main result:

\begin{theorem} \label{theorem1} Let ${\bf X}$ be a 2-dimensional CAT(0) cube complex such that the degrees of all its vertices are uniformly bounded by $\Delta.$ Then there exists $M<\infty$, independent of $\bf X$, such that $\chi(\Gamma({\bf X}))\le \epsilon(\Delta)=M\Delta^{26}.$ In particular, $\tau({\bf X})\le \epsilon(\Delta),$ i.e. the 1-skeleton of $\bf X$ isometrically embeds into the Cartesian product of at most $\epsilon(\Delta)$ trees. Finally, all event structures of (out)degree $\Delta_0$ whose domains are 2-dimensional cube complexes admit a nice labeling with at most $\epsilon(\Delta_0+2)$ labels.
\end{theorem}

We actually obtain the following bound: $\chi(\Gamma({\bf X}))\leq\epsilon(\Delta)=1165226\Delta^{26}$, or, simply $M=1165226$.

The second assertion  of Theorem \ref{theorem1} follows from the first assertion because ${\Gamma_{\#}}({\bf X})$ is a subgraph of
${\Gamma}({\bf X})$  and because of the equality $\tau({\bf X})=\chi({\Gamma_{\#}}({\bf X}))$. The third assertion is a consequence of the fact that the number of labels in a nice labeling is equal to $\chi(\Gamma_{\alpha}({\bf X})),$ and because $\Gamma_{\alpha}({\bf X})$ is a subgraph of $\Gamma({\bf X})$ and $\Delta\le \Delta_0+n$ holds for all $n$-dimensional CAT(0) cube complexes.

The main focus of our paper is thus to prove the first assertion of Theorem \ref{theorem1}. To show that the chromatic number $\chi(\Gamma({\bf X}))$ of the contact graph $\Gamma({\bf X})$ is polynomially bounded in $\Delta$, we show that the edges of $\Gamma({\bf X})$ can be distributed over six spanning subgraphs of $\Gamma({\bf X})$, such that the chromatic numbers of each of these subgraphs can be polynomially bounded. As a result, each vertex of $\Gamma({\bf X})$ (hyperplane of $\bf X$) receives a sextuple of colours, each colour corresponding to the colour received by this vertex in the colouring of the corresponding subgraph. Since each edge of $\Gamma({\bf X})$  is present in at least one spanning subgraph, the sextuple-colouring of the hyperplanes of ${\bf X}$ is a correct colouring of the contact graph  $\Gamma({\bf X})$. The number of colours is the product of the six numbers of colours used to colour the spanning subgraphs, whence it is polynomial in $\Delta$. In Sections~4-6, one after another, we will define and colour the six spanning subgraphs. For this, we will study the geometrical and the combinatorial properties of contact graphs of 2-dimensional CAT(0) cube complexes.

We conclude with the formulation of the second principal result:

\begin{theorem} \label{theorem2} For any $n>0$, there exists a finite CAT(0) cube complex ${\bf X}_n$ with constant  maximum degree  such that any colouring of the crossing graph of  ${\bf X}_n$ requires more than $n$ colours, i.e., any isometric embedding of ${\bf X}_n$ into a Cartesian product of trees requires $>n$ trees. There exists an infinite CAT(0) cube complex ${\bf X}$ with constant maximum degree  which cannot be isometrically embedded into a Cartesian product of a finite number of trees, i.e., the chromatic number of its crossing graph  is infinite.
\end{theorem}

\section{Related results}

Our counterexample in Theorem 2 and some steps of the proof of Theorem 1 are based on the fact that there exist classes of geometric intersection graphs that are not $\chi$-bounded, and also classes that are $\chi$-bounded. Therefore, we continue with a brief review of such classes. Given a family of sets $\mathcal F$ with the ground-set $S$, the {\it intersection graph} of $\mathcal F$ has the sets of $\mathcal F$ as the vertex-set and two sets $F,F'$ define an edge of the intersection graph if and only if $F\cap F'\ne\emptyset$.  With some abuse of notation, we will denote by $\chi({\mathcal F})$ and $\omega({\mathcal F})$, respectively, the {\it chromatic number} and the {\it clique number} of the intersection graph of $\mathcal F$. The {\it degree} $\delta({\mathcal F})$ of $\mathcal F$ is the maximum number of sets of $\mathcal F$  to which an element of $S$ belongs.
It is evident and well-known that the equality $\delta({\mathcal F})=\omega({\mathcal F})$ holds if $\mathcal F$ satisfies the {\it Helly property}, i.e., any subfamily ${\mathcal F}'$ of $\mathcal F$ has a nonempty intersection whenever any two sets of ${\mathcal F}'$ intersect.

Gallai established (in unpublished work; see \cite{Golumbic_book,Gyarfas}) that $\chi({\mathcal I})=\omega({\mathcal I})$ for families $\mathcal I$  of intervals of the real line (whose intersection graphs are called interval graphs). This founding result has numerous generalizations, among which we recall only a few of them. First, it is well known that the equality
\[\chi({\mathcal T})=\omega({\mathcal T})=\delta({\mathcal T})\]
holds for families $\mathcal T$ of subtrees of a tree \cite{Golumbic_book,Gyarfas} (the intersection graphs of subtrees are the so-called {\it chordal graphs} which are known to be a subclass of perfect graphs).

On the other hand, Asplund and Grunbaum \cite{AsplundGrunbaum} showed that $\chi({\mathcal R})\le 4\omega^2({\mathcal R})-4\omega({\mathcal R})$ for any family ${\mathcal R}$ of axis-parallel rectangles of ${\mathbb R}^2$. Burling \cite{Burling} presented a series ${\mathcal B}_n$ of axis-parallel boxes of ${\mathbb R}^3$ with $\omega({\mathcal B}_n)=2$ and $\chi({\mathcal B}_n)>n$ (see \cite{Gyarfas} for a description of Burling's construction). Gy\'arf\'as \cite{GyarfasChromatic} showed that $\chi({\mathcal I}_t)\le 2t\omega({\mathcal I}_t)$ for families of sets, each set consisting of $t$ intervals of the line. Gy\'arf\'as \cite{GyarfasChromatic} and Kostochka \cite{Kostochka} showed that the class of intersection graphs of chords of a circle is $\chi$-bounded by $2^{\omega}\omega(\omega+2)$; there are known examples  showing only that $\chi\ge \omega\log \omega$ (similar kinds of bounds were proved in \cite{KostochkaKratochvil} for polygon-circle graphs). On the other hand, Kostochka \cite{Kostochka}  proved that the chromatic number of any triangle-free intersection graph of chords is at most 5 (and this bound is known to be sharp). It was conjectured in \cite{GyarfasLehel} that the class of intersection graphs of ``L'' shapes in the plane is $\chi$-bounded and McGuinness \cite{McGuinness} established this conjecture in the case of L-shapes whose vertical stem is infinite.  Recently, this conjecture, as well as the famous conjecture of Erd\"os that triangle-free intersection graphs of segments in the plane can be coloured with a constant number of colours, were disproved in \cite{Pawlik_et_al,Pawlik_et_al_Erdos}.  Notably, these counterexamples involve a construction very similar to the above-mentioned construction of Burling, which plays a major role in our proof of Theorem~\ref{theorem2}.

We conclude with a few known results about the three questions in case of CAT(0) cube complexes. Using the result of Kostochka about the triangle-free graphs of chords and the ``stretchability'' of hyperplanes of plane 2-dimensional CAT(0) cube complexes (called squaregraphs), it was shown in \cite{BandeltChepoiEppstein_squaregraph} that 1-skeleta of such graphs can be embedded into Cartesian products of at most 5 trees. In \cite{BandeltChepoiEppstein} the CAT(0) cube complexes which can be embedded in Cartesian products of two trees were characterized in a local-to-global way as the 2-dimensional CAT(0) cube complexes in which the links of vertices are bipartite graphs.  In unpublished work, Sageev has shown that Gromov-hyperbolic CAT(0) cube complexes (and in particular their 1-skeleta) isometrically embed in the product of finitely many trees, and this was proved independently by Haglund in~\cite{HaglundHabilitation}.  Sageev and Dru\c{t}u have extended this to certain CAT(0) cube complexes that are universal covers of nonpositively-curved cube complexes with relatively hyperbolic fundamental group (personal communication from M. Sageev).

Likewise, the 1-skeleton of an acyclic CAT(0) cube complex of dimension $d$ admits an isometric embedding in the product of at most $d$ trees~\cite{BandeltChepoi_acyclic} (a cube complex $\bf X$ is \emph{acyclic} if its cubes define a hypergraph which is $\alpha$-acyclic in the sense of Berge~\cite{Berge_book}).  The same paper also introduces the notion of a {\it perfect} CAT(0) cube complex as a CAT(0) cube complex $\bf X$ whose crossing graph $\Gamma_{\#}({\bf X})$ is perfect (i.e., the chromatic number of $\Gamma_{\#}({\bf X})$ and of any of its induced subgraphs equals the clique number) and conjectures that a CAT(0) cube complex $\bf X$ is perfect if and only if the CAT(0) cube complexes which correspond to simplex graphs obtained via median homomorphisms from $G({\bf X})$ are perfect and shows that the {\it Strong Perfect Graph Conjecture} implies this conjecture. The Strong Perfect Graph Conjecture has since been proved~\cite{CRST}, and thus the conjectured characterization of perfect CAT(0) cube complexes is also true.

Relatedly, Ballman and \'{S}wi\c{a}tkowski, in~\cite{BallmanSwiatkowski}, showed that CAT(0) cube complexes with some additional structure -- \emph{foldable cubical chamber complexes} -- admit bi-Lipschitz embeddings into the product of $d$ trees, where $d$ is the dimension of the cube complex.  Moreover, Bowditch has, in~\cite{BowditchMedianTrees}, recently given conditions on metric median algebras ensuring that they admit a bilipschitz embedding in the direct product of finitely many $\reals$-trees, providing a non-discrete analogue of Theorem~\ref{theorem1}.

On the other hand, there are known to be several classes of event structures for which the nice labeling conjecture is true. Assous et al. \cite{AsBouChaRo} proved that the event structures of degree 2 admit   nice labelings with 2 labels and noticed that Dilworth's theorem implies that the conflict-free event structures of degree $n$ have nice labelings with $n$ labels. Recently, Santocanale \cite{Santocanale} proved that all event structures of degree  3 and with tree-like partial orders  have nice labelings with 3 labels.

\section{Preliminaries}
This section is devoted to definitions and basic properties of the objects used throughout the paper.  We begin with a brief review of graph colouring, and then discuss the basic properties of CAT(0) cube complexes (following the discussion in~\cite{HagenQuasiArb}) and median graphs (following the discussion in~\cite{BandeltChepoi_survey}).  We then define the crossing graph (see e.g.~\cite{Roller,NibloSplittingObstruction}) and the contact graph (see~\cite{HagenQuasiArb}) of a CAT(0) cube complex, and the pointed contact graph (see~\cite{ChepoiNiceLabeling}) of a pointed CAT(0) complex or a pointed median graph.  This is then related to the nice labeling problem for event structures.  We discuss disc diagrams in CAT(0) cube complexes, which are used throughout the paper, and then relate the crossing and contact graphs to isometric and convex embeddings of CAT(0) cube complexes.  Finally, we define and prove basic properties of hyperplane-distance, footprints, and imprints, all of which are used in our colouring of contact graphs.

\subsection{Graph colouring}
Let $G$ be a connected graph, with vertex set $\mathcal V(G)$.  An edge of $G$ joining $x,y\in\mathcal V(G)$ is denoted $xy$.  A \emph{graph homomorphism} $\phi:G\rightarrow H$ is a map from $\mathcal V(G)$ to $\mathcal V(H)$ such that, if $xy$ is an edge of $G$, then $\phi(x)\phi(y)$ is an edge of $H$.  A \emph{colouring} of $G$ by a set $\mathcal K$ of colours is a graph homomorphism $c:G\rightarrow K(\mathcal K)$, where $K(\mathcal K)=K_n$ is the complete graph with vertex set $\mathcal K$ of cardinality $n$.  Equivalently, $c$ is an assignment of an element of $\mathcal K$ -- a colour -- to each vertex of $G$ in such a way that, for each edge $xy$, we have $c(x)\neq c(y)$.  The \emph{chromatic number} $\chi(G)$ of $G$ is the cardinality of a smallest set $\mathcal K$ for which there exists a $\mathcal K$-colouring of $G$.  Note that if $H\subseteq G$ is a subgraph, then $\chi(H)\leq\chi(G)$.  Also, it was shown by de Bruijn and Erd\H{o}s~\cite{deBruijnErdos} that, for any graph $G$, we have $\chi(G)\leq n$ if and only if $\chi(H)\leq n$ for each finite subgraph $H$ of $G$.  Hence, to $\mathcal K$-colour $G$, it suffices to fix an arbitrary vertex $v$ and to $\mathcal K$-colour the ball $B_n(v)$ for each $n\geq 0$.

\subsection{CAT(0) cube complexes and hyperplanes}
For $d\in\naturals\cup\{0\}$, a \emph{$d$-cube} is a copy of $[-\frac{1}{2},\frac{1}{2}]^d$ endowed with the $\ell_1$ metric.  A \emph{cube complex} is a CW-complex $\bf X$ whose cells are cubes of various dimensions, attached in the expected way: any two cubes of $\bf X$ that have nonempty intersection intersect in a common face, i.e. the attaching map of each cube restricts to a combinatorial isometry on its faces.  For $i\in\naturals\cup\{0\}$, and any complex $Z$, we denote by $Z^{(i)}$ the $i$-skeleton of $Z$.

The \emph{link} of a 0-cube $x\in\bf X$ is the complex built of simplices, with a $(d-1)$-simplex for each $d$-cube containing $x$, with simplices attached according to the attachments of the corresponding cubes.  The simply-connected cube complex $\bf X$ is \emph{CAT(0)} if the link $Lk(x)$ of each 0-cube $x$ is a \emph{flag (simplicial) complex}, i.e. if each $(d+1)$-clique in $Lk(x)$ spans an $d$-simplex.  The \emph{dimension} of the CAT(0) cube complex $\bf X$ is the largest value of $d$ for which $\bf X$ contains a $d$-cube, and the \emph{degree} $\Delta$ of $\bf X$ is the degree of a highest-degree 0-cube.

A \emph{midcube} of the $d$-cube $c$, with $d\geq 1$, is the isometric subspace obtained by restricting exactly one of the coordinates of $d$ to 0.  Note that a midcube is a $(d-1)$-cube.  The midcubes $a$ and $b$ of $\bf X$ are \emph{adjacent} if they have a common face, and a \emph{hyperplane} $H$ of $\bf X$ is a subspace that is a maximal connected union of midcubes such that, if $a,b\subset H$ are midcubes, either $a$ and $b$ are disjoint or they are adjacent.  Equivalently, a hyperplane $H$ is a maximal connected union of midcubes such that, for each cube $c$, either $H\cap c=\emptyset$ or $H\cap c$ is a single midcube of $c$.  In~\cite{Sageev95}, Sageev showed that each hyperplane $H$ is a CAT(0) cube complex of dimension at most $\dimension{\bf X}-1$, and that ${\bf X}-H$ consists of exactly two components, called \emph{halfspaces}.  A 1-cube $c$ is \emph{dual} to the hyperplane $H$ if the 0-cubes of $c$ lie in distinct halfspaces of $H$, i.e. if the midpoint of $c$ is in a midcube contained in $H$.  The relation ``dual to the same hyperplane'' is an equivalence relation on the set of 1-cubes of $\bf X$; denote this relation by $\Theta$ and denote by $\Theta(H)$ the equivalence class consisting of 1-cubes dual to the hyperplane $H$.

In the remainder of this paper, $\bf X$ denotes a CAT(0) cube complex and $\mathcal H$ denotes the set of hyperplanes.  For each $H\in\mathcal H$, let $N(H)$ be the subcomplex of $\bf X$ consisting of all closed $d$-cubes $c$ such that $H\cap c\neq\emptyset$.  The subcomplex $N(H)$ is called the \emph{carrier} of $H$, and it was proved in~\cite{Sageev95} that $N(H)$ is a CAT(0) cube complex isomorphic to $[-\frac{1}{2},\frac{1}{2}]\times H$.  In particular, we shall often use the natural projection $N(H)\rightarrow H\cong\{0\}\times H$ arising from the isomorphism $N(H)\cong[-\frac{1}{2},\frac{1}{2}]\times H$.

Although $\bf X$ admits a CAT(0) metric arising from the $\ell_2$ metric on the constituent cubes~\cite{GromovHyperbolic}, it is more natural to use the \emph{wall-metric} arising from the hyperplanes, discussed, for example, in~\cite{HagenQuasiArb} and, in the context of median complexes, in~\cite{vandeVel_book}.  More precisely, $\bf X$ admits a metric $d$ such that the restriction of $d$ to any cube $c$ of $\bf X$ is the $\ell_1$ metric on $c$ and the restriction of $d$ to the 1-skeleton of $\bf X$ is the standard graph distance.  In particular, a \emph{combinatorial path} $P\rightarrow\bf X$ -- a path in the 1-skeleton of $\bf X$ -- is a geodesic segment in $\bf X$ if and only if $P$ is a geodesic segment of the 1-skeleton.  Equivalently, $P$ is a geodesic segment if and only if $P$ contains at most one 1-cube dual to each hyperplane of $\bf X$.  The length $|P|$ of the path $P$ is equal to the number of hyperplanes (with multiplicity) that cross $P$ in the sense defined below, and $P$ is a geodesic segment if and only if $|P|$ is equal to the number of hyperplanes that separate its endpoints (as defined below).  This $\ell_1$ metric was investigated by Haglund in~\cite{HaglundSemisimple}.

Let $Y\subseteq\bf X$ be a subcomplex.  Let $H$ be a hyperplane, and denote by ${\bf A}(H)$ and ${\bf B}(H)$ the halfspaces of $H$.  Then $H$ \emph{crosses} $Y$ if ${\bf A}(H)\cap Y$ and ${\bf B}(H)\cap Y$ are both nonempty.  In particular, if $H'$ is another hyperplane, then $H$ and $H'$ \emph{cross} if and only if each of the \emph{quarter-spaces} ${\bf A}(H)\cap{\bf A}(H'),\,{\bf A}(H)\cap{\bf B}(H'),\,{\bf B}(H)\cap{\bf A}(H'),\,{\bf B}(H)\cap{\bf B}(H')$ is nonempty.  Equivalently, $H$ and $H'$ cross if and only if there exists a 2-cube $s$ whose boundary path contains a concatenation $cc'$, where $c\in\Theta(H)$ and $c'\in\Theta(H')$.

If $H$ and $H'$ do not cross, but $\bf X$ contains a pair of 1-cubes $c\in\Theta(H),\,c'\in\Theta(H')$ such that $c,c'$ have a common 0-cube, then $H$ and $H'$ \emph{osculate}.  If $H$ and $H'$ either cross or osculate, then they \emph{contact}, denoted $H\coll H'$.  Equivalently, $H\coll H'$ if and only if $N(H)\cap N(H')\neq\emptyset$.

If $Y,Y'$ are convex subcomplexes of $\bf X$, and $H$ is a hyperplane such that $Y\subseteq{\bf A}(H)$ and $Y'\subseteq{\bf B}(H)$, then $H$ \emph{separates} $Y$ from $Y'$.  We see that $H\coll H'$ if and only if no third hyperplane separates $H$ from $H'$.  Relatedly, we say that a subset $\mathcal H'\subseteq\mathcal H$ is \emph{inseparable} if, given any two hyperplanes $H,H'\in\mathcal H'$, every hyperplane $H''\in\mathcal H$ that separates $H$ from $H'$ also belongs to $\mathcal H'$.  Also, the distance between the convex subcomplexes $Y$ and $Y'$ is equal to the number of hyperplanes that separate $Y$ from $Y'$, and this is also the length of a shortest geodesic segment having one endpoint in $Y$ and one endpoint in $Y'$.

Each hyperplane $H$, and its carrier $N(H)$, is convex with respect to the wall-metric, and we give a simple characterization of convexity below.  Sageev~\cite{Sageev95} also showed that $H$ is convex with respect to the CAT(0) metric.

The property of being a \emph{Cartesian product} is characterized as follows for CAT(0) cube complexes.  The Cartesian product ${\bf X}={\bf X}_1\times{\bf X}_2$ of two CAT(0) cube complexes ${\bf X}_1,{\bf X}_2$ is again a CAT(0) cube complex.  Let $\mathcal H_1$ and $\mathcal H_2$ denote the sets of hyperplanes of ${\bf X}_1$ and ${\bf X}_2$ respectively.  Then each hyperplane of $\bf X$ is of the form $H\times{\bf X}_2$, with $H\in\mathcal H_1$, or ${\bf X}_1\times H$, where $H\in\mathcal H_2$, and each hyperplane of the former form crosses each hyperplane of the latter form.  Conversely, if the set $\mathcal H$ of hyperplanes of $\bf X$ decomposes as a disjoint union $\mathcal H=\mathcal H_1\sqcup\mathcal H_2$ such that each element of $\mathcal H_1$ crosses each element of $\mathcal H_2$, then ${\bf X}\cong{\bf X}_1\times{\bf X}_2$, where for $i\in\{1,2\}$, the complex ${\bf X}_i$ is a convex, and thus CAT(0),  subcomplex of $\bf X$ that is crossed by each hyperplane in $\mathcal H_i$ and no others.  The wall-metric on $\bf X$ is identical to the metric defined by $d((x_1,x_2),(y_1,y_2))=d_1(x_1,y_1)+d_2(x_2,y_2)$, where $d_i$ is the wall-metric on ${\bf X}_i$ and $x_i,y_i\in{\bf X}_i$.

A major source of CAT(0) cube complexes is the notion of the ``cube complex dual to a wallspace''.  Here, we roughly follow the discussion in~\cite{HruskaWiseAxioms} of the procedure for producing a cube complex from a wallspace.  In the finite case, wallspaces derived from median graphs (and hence from cube complexes) were characterized by Mulder and Schrijver~\cite{MulderSchrijver} as  Helly  copair hypergraphs.  In the context of median graphs, the cubulation procedure described presently was formulated by Barth\'{e}lemy in~\cite{Barthelemy_cubulation}.

A \emph{wallspace} is a set $\mathcal S$, together with a collection $\mathcal H$ of bipartitions of $\mathcal S$, called \emph{walls}; the two sets in each bipartition are called \emph{halfspaces}.  We require that for all $s_1,s_2\in\mathcal S$, there are finitely many walls $H$ such that $s_1$ and $s_2$ lie in different halfspaces associated to $H$.

The dual cube complex $\bf X$ is constructed as follows: a 0-cube $v$ is a choice $v(H)$ of halfspace for each $H\in\mathcal H$, in such a way that $v(H)\cap v(H')\neq\emptyset$ for all $H,H'\in\mathcal H$.  Each $s\in\mathcal S$ determines a \emph{canonical} 0-cube $v_s$ defined by declaring, for each $H\in\mathcal H$, that $v_s(H)$ is the halfspace containing $s$.  The set of 0-cubes of the dual cube complex consists of the set of canonical 0-cubes, together with any 0-cube $v$ that differs from some, and hence any, canonical 0-cube on finitely many walls.

$\bf X$ has a set of hyperplanes corresponding bijectively to $\mathcal H$.  The construction of a dual cube complex from a wallspace appears in a group-theoretic context in the work of Sageev~\cite{Sageev95}; the formal notion of a wallspace is due to Haglund and Paulin~\cite{HaglundPaulin}, and the construction of the cube complex in this more general context appears in~\cite{ChatterjiNiblo} and~\cite{NicaCubulating}.

This construction shows that each 0-cube of $\bf X$ can be thought of as a consistent choice of halfspace for each hyperplane, which is to say that, if $v$ is a 0-cube of $\bf X$ and $H,H'$ are hyperplanes, then $v(H)\cap v(H')\neq\emptyset$, where $v(H)$ denotes the halfspace of $H$ containing $v$.  Moreover, since $\mathbf X$ is connected, any two 0-cubes are separated by finitely many hyperplanes, so that, for 0-cubes $v,v'$, there are finitely many hyperplanes $H$ such that $v(H)\neq v'(H)$.

Hence the \emph{pointed CAT(0) cube complex} ${\bf X}_v$ -- i.e., the cube complex ${\bf X}$ with basepoint $v$ -- is equipped with a natural orientation on the 1-skeleton.  Indeed, the \emph{initial} 0-cube of the 1-cube $c$ dual to the hyperplane $H$ is the 0-cube lying in $v(H)$ and the \emph{terminal} 0-cube of $c$ is the 0-cube lying in ${\bf X}-v(H)$.  The hyperplanes $H,H'$ of ${\bf X}_v$ \emph{contact with respect to $v$}, denoted $H\coll_vH'$, if $H$ and $H'$ are dual to 1-cubes $c,c'$, respectively, such that $c$ and $c'$ have the same initial 0-cube.  Note that $H\coll_vH'$ only if $H\coll H'$.  Also note that if $H$ crosses $H'$ in a 2-cube $s$, then at least one 0-cube of $s$ is initial in both of its incident 1-cubes in $s$, and hence, if $H$ and $H'$ cross, then $H\coll_vH'$.  On the other hand, if $H$ osculates with $H'$, then $H\coll H'$ if and only if neither of $H$ nor $H'$ separates the other from $v$.

\subsection{Median graphs and parallelism of edges}
Let $G$ be a connected graph and let $d$ be the standard path-metric on $G$ (i.e. each edge of $G$ has length 1 and $d(u,v)$ counts the number of edges in a shortest path joining the vertices $u$ and $v$).  The \emph{interval} $I(u,v)$ is the set of all vertices $x\in G$ such that $d(u,v)=d(x,u)+d(x,v)$.  The graph $G$ is a \emph{median graph} if for all triples of vertices $u,v,w\in G$, the set
\[m(u,v,w)=I(u,v)\cap I(u,w)\cap I(v,w)\]
contains exactly one point, also denoted $m(u,v,w)$, called the \emph{median} of $u,v,w$.

The induced subgraph $G(S)$ of $G$ generated by the set $S$ of vertices is \emph{convex} if for all $u,v\in S$, the interval $I(u,v)\subset G(S)$.  The subgraph $G(S)$ is \emph{gated} if for each vertex $v$ of $G$, there exists a unique vertex $v'=g(v)\in S$, called the \emph{gate of $v$ in $S$} such that $d(v,y)=d(v,v')+d(v',y)$ for all $y\in S$.  Each convex set of vertices of $G$ is gated.  A \emph{halfspace} is a convex subset $H$ such that $G-H$ is also convex, and the pair $(H,G-H)$ is a \emph{convex split}.

The relation $\Theta$ is defined on the set of edges of $G$ as follows: $\Theta$ is the \emph{Djokovi\'{c}-Winkler relation (``parallelism'')} defined as follows.  If $uv$ and $xy$ are edges of $G$, then $(uv,xy)\in\Theta$ if and only if
\[d(u,x)+d(v,y)\neq d(u,y)+d(v,x).\]
Equivalently, $\Theta$ is the transitive closure of the ``opposite'' relation: $uv$ and $xy$ are \emph{opposite edges of a 4-cycle} if $uvxy$ is a 4-cycle in $G$ (see~\cite{EppsteinFalmagneOvchinnikov,ImKl}).  The equivalence class $\Theta(uv)$ defines two complementary halfspaces $\mathbf{A}(uv)$ and ${\bf B}(uv)$ obtained by removing from $G$ the edges of $\Theta(uv)$ (but leaving their endpoints).  The class $\Theta(uv)$ therefore determines a convex split $({\bf A}(uv),{\bf B}(uv))$ of $G$~\cite{Mulder_book,vandeVel_book}.  Conversely, for each convex split $({\bf A},{\bf B})$, there exists at least one edge $uv$ such that ${\bf A}={\bf A}(uv)$ and ${\bf B}={\bf B}(uv)$.  The convex splits $({\bf A}_1,{\bf B}_1)$ and $({\bf A}_2,{\bf B}_2)$ are \emph{incompatible} if and only if each of the sets ${\bf A}_1\cap{\bf A}_2,\,{\bf A}_1\cap{\bf B}_2,\,{\bf B}_1\cap{\bf A}_2,\,{\bf B}_1\cap{\bf B}_2$ is nonempty.

The resemblance to the definition of crossing hyperplanes, and the use of the notation $\Theta$ for the set of 1-cubes dual to a hyperplane of a CAT(0) cube complex is not accidental.  For each median graph $G$, there exists a CAT(0) cube complex $\bf X$ whose 1-skeleton is $G$, and the hyperplanes of $\bf X$ correspond bijectively to convex splits of $G$: the equivalence class $\Theta(uv)$ of the edge $uv$ of $G$ is precisely the set of 1-cubes dual to the hyperplane $H$ of $\bf X$ that crosses the 1-cube $uv$.  Conversely, the 1-skeleton $G({\bf X})$ of the CAT(0) cube complex $\bf X$ is a median graph, and the hyperplanes of $\bf X$ correspond in the same way to the convex splits of $G({\bf X})$ (see~\cite{ChepoiMedian}).

There is thus a perfect analogy between the following notions about median graphs and the corresponding notions about CAT(0) cube complexes.  If $S$ is a set of vertices, and $H=({\bf A},{\bf B})$ is a convex split of $G$, then $H$ \emph{crosses} $S$ if there exist $u,v\in S$ with $u\in\bf A$ and $v\in\bf B$.  In particular, the crossing of two hyperplanes $H,H'$ of $\bf X$ corresponds to incompatibility of the corresponding convex splits of $G=G({\bf X})$.  Likewise, separation of subgraphs of $G$ by a convex split corresponds to separation of those subgraphs by the corresponding hyperplane.

Choosing a base vertex $v$ of $G$, we define an orientation of all edges.  Let $xy$ be an edge such that $d(x,v)\leq d(y,v)$.  Then $m=m(v,x,y)=x$ since $d(x,y)=d(m,x)+d(m,y)=1$ and hence $x$ is strictly closer than $y$ to $v$.  Let $x$ be the initial vertex of $xy$ and $y$ the terminal vertex.  Note that if $uv\in\Theta(xy)$, and the terminal vertex $x$ lies in the halfspace ${\bf A}(xy)$ of the corresponding convex split, then the terminal vertex $u$ of the edge $uv$ also lies in ${\bf A}(xy)$, and hence each convex split of $G$ is oriented.  If $\Theta_1$ and $\Theta_2$ are parallelism classes of edges, we write $\Theta_1{{\coll}_v}\Theta_2$ if either the corresponding convex splits are incompatible, or if there exist edges $xy_1\in\Theta_1$ and $xy_2\in\Theta_2$ such that $x$ is the initial vertex of $xy_1$ and $xy_2$.  Note that the hyperplanes $H,H'$ of $\bf X$ satisfy $H{\coll}_v H'$ if and only if $\Theta(H){\coll}_v\Theta(H')$.

\subsection{Contact and crossing graphs}
Let $\bf X$ be a CAT(0) cube complex and let $G({\bf X})={\bf X}^{(1)}$ be the corresponding median graph.  Let $\mathcal H$ be the set of hyperplanes of $\bf X$ or, equivalently, the set of parallelism classes of edges in $G({\bf X})$.  The contact graph of $\bf X$ was defined in~\cite{HagenQuasiArb} as a modification of the ``crossing graph'' -- the intersection graph of the set $\mathcal H$ of hyperplanes in $\bf X$ -- previously studied by Bandelt, Dress, Eppstein, Niblo, Roller, van de Vel and others (see~\cite{Roller} and~\cite{NibloSplittingObstruction}).  While any simplicial graph is the crossing graph of some CAT(0) cube complex, the class of graphs that arise as contact graphs is quite restricted: contact graphs are connected and quasi-isometric to trees~\cite{HagenQuasiArb}.

\begin{definition}[Contact graph, crossing graph]\label{defn:contact}
The \emph{contact graph} $\Gamma(\bf X)$ of $\bf X$ is the graph whose vertex set is the set $\mathcal H$, with an edge joining the vertices $H_1$ and $H_2$ if $H_1\coll H_2$.  We use the same notation for a vertex of $\Gamma(\bf X)$ as for its corresponding hyperplane, and we use the notation $H_1\coll H_2$ for the (closed) edge of $\Gamma(\bf X)$ joining the contacting hyperplanes $H_1$ and $H_2$.

The \emph{crossing graph} $\xing X$ is the subgraph of $\Gamma(\bf X)$ obtained by deleting each open edge corresponding to an osculating pair of hyperplanes, so that $H_1$ and $H_2$ are adjacent in $\xing X$ if and only if they cross.

Given hyperplanes $U,V\in\mathcal H$, we denote by $\rho(U,V)$ the distance in $\Gamma(\bf X)$ from $U$ to $V$.
\end{definition}

Likewise, for a pointed CAT(0) cube complex ${\bf X}_v$ or pointed median graph $G({\bf X})_v$, we define the pointed contact graph, introduced in~\cite{ChepoiNiceLabeling}, as follows.

\begin{definition}[Pointed contact graph]\label{defn:pointedcontact}
The \emph{pointed contact graph} $\Gamma_{v}({\bf X})$ is the subgraph of $\Gamma({\bf X})$ defined as follows: the vertex set of $\Gamma_{v}({\bf X})$ is the set $\mathcal H$ of hyperplanes of $\bf X$, and $H$ and $H'$ are joined by an edge of $\Gamma_{v}({\bf X})$ if and only if $H\coll_vH'$.
\end{definition}

Pointed contact graphs are used in our applications to the nice labeling problem.  Observe that if $H$ and $H'$ cross, then the intersection of their carriers contains a 2-cube $s$, one of whose four 0-cubes must be initial in the incident 1-cubes dual to $H$ and $H'$.  Hence $\xing X\subseteq\Gamma_{v}({\bf X})\subseteq\Gamma({\bf X})$.

\subsection{Event structures, nice labeling and the associated cube complexes}
The following is an informal summary of the relationship between (pointed) contact graphs of CAT(0) cube complexes and median graphs and nice labelings of event structures, following the treatment given in~\cite{ChepoiNiceLabeling}.

An \emph{event structure}\footnote{Also called a \emph{coherent event structure} or \emph{an event structure with binary conflict}.} is a triple $\mathcal E=(E,\leq, \smile)$, where $E$ is a set of \emph{events}, $\leq$ is a partial order on $E$, called \emph{causal dependency}, and $\smile$ is a symmetric, irreflexive binary relation on $E$ called \emph{conflict}.  For all $e,e',e''$, if $e\smile e'$ and $e'\leq e''$, then $e\smile e''$.  Additionally, $\mathcal E$ is \emph{finitary}, which is to say that for all $e,e'\in E$, there exist finitely many $e''\in E$ such that $e\leq e''\leq e'$.

The events $e$ and $e'$ are \emph{concurrent}, denoted $e\frown e'$, if they are incomparable in the partial ordering $\leq$ and $e\not\smile e'$.  A conflict $e\smile e'$ is \emph{minimal} if there does not exist $e''\not\in\{e,e'\}$ such that $e''$ precedes one of $e$ and $e'$ in $\leq$ and is in conflict with the other.  The events $e$ and $e'$ are \emph{independent} (or \emph{orthogonal}) if they are either concurrent or in minimal conflict.  An \emph{independent set} is a set of pairwise independent events in $E$.  The \emph{degree} of $E$ is the maximum cardinality of an independent set in $E$.

In~\cite{RoTh}, Rozoy and Thiagarajan formulated the \emph{nice labeling problem} for event structures.  A \emph{labeling} is a map $\lambda:E\rightarrow\Lambda$, where $\Lambda$ is some alphabet, and $\lambda$ is a \emph{nice labeling} if $\lambda(e)\neq\lambda(e')$ whenever $e$ and $e'$ are independent.  Solving the nice labeling problem for $\mathcal E$ entails constructing a nice labeling $\lambda$ such that $\Lambda$ is finite.  More quantitatively, one asks, given a class of event structures, whether there exists a function $f:\naturals\rightarrow\naturals$ such that, for an event structure $\mathcal E$ of degree $n$ in the given class, there exists a nice labeling of $\mathcal E$ with $|\Lambda|\leq f(n)$.  The first author answered this question in the negative in~\cite{ChepoiNiceLabeling} when the class in question is the class of event structures of degree at least five.  Theorem~\ref{theorem1}, however, asserts that the nice labeling conjecture is true for event structures of finite degree that have as their domain a CAT(0) cube complex of dimension at most 2.

The \emph{domain} $\mathcal D(\mathcal E)$ of the event structure $\mathcal E$ is defined as follows.  A \emph{configuration} $C$ is a subset $C\subseteq E$ of the set of events such that no two elements of $C$ are in conflict, and, if $e\leq e'\in C$ are not in conflict, then $e\in C$.  The domain $\mathcal D(\mathcal E)$ is the set of all such configurations $C$, ordered by inclusion.  This construction naturally gives rise to a median graph and an accompanying CAT(0) cube complex associated to $\mathcal E$.  Indeed, following~\cite{Santocanale}, let $G=G(\mathcal E)$ be the graph whose vertices are the elements of the domain $\mathcal D(\mathcal C)$, with $C$ and $C'$ joined by an edge if and only if $C=C'\cup\{e\}$ for some $e\in E-C$.  In this situation, the edge $C'C$ is directed from $C'$ to $C$.  In other words, an event $e\in E$ is viewed as a minimal change from one configuration to another~\cite{WiNi}.

It can be shown~\cite{Barthelemy_Constantin} that $G$ is a median graph, and thus $G=G({\bf X})$, where ${\bf X}$ is a CAT(0) cube complex; abusing language slightly, we refer to $G({\bf X})$ or to $\bf X$ as the \emph{domain} of $\mathcal E$, since these objects are uniquely determined by $\mathcal E$.  The hyperplanes of $\bf X$ correspond bijectively to the events in $E$.  The events $e$ and $e'$ are concurrent if and only if the corresponding hyperplanes cross.  In the language of CAT(0) cube complexes, this bijection was recently rediscovered by~\cite{AOS}.

Indeed, let $C''$ be a configuration that does not contain $e$ or $e'$, with $e,e'$ not in conflict, but such that $C=C''\cup\{e\}$ and $C'=C''\cup\{e'\}$ are downward-closed.  Then by the definition of concurrency, the configurations $C'',C',C,$ and $C'\cup C$ are the vertices of a 4-cycle in $G$ bounding a 2-cube in $\bf X$ whose crossing dual hyperplanes correspond to $e$ and $e'$.  On the other hand, if $e$ and $e'$ are in minimal conflict, then $C'$ and $C$ are both adjacent to $C''$, and thus $C''$ has two incident 1-cubes in $\bf X$, one dual to each of the hyperplanes $e$ and $e'$, and hence the corresponding hyperplanes osculate.  (The construction of $\bf X$ from the space of configurations of an event structure is highly reminiscent of the notion of a state complex of a metamorphic robot~\cite{AbramsGhrist,GhristPeterson} and of the construction of a cube complex from a wallspace~\cite{ChatterjiNiblo,NicaCubulating}.)

Conversely, each CAT(0) cube complex $\bf X$ (and thus each median graph $G({\bf X})$~\cite{ChepoiMedian}), and any fixed base 0-cube $v\in\bf X$, gives rise to an event structure ${\mathcal E}_v$ whose events are the hyperplanes of $\bf X$~\cite{Barthelemy_Constantin}.  Hyperplanes $H$ and $H'$ define concurrent events if and only if they cross, and $H\leq H'$ if and only if $H=H'$ or $H$ separates $H'$ from $v$.  The events defined by $H$ and $H'$ are in conflict if and only if $H$ and $H'$ do not cross and neither separates the other from $v$.  Thus the events corresponding to $H$ and $H'$ are in minimal conflict if $H$ and $H'$ osculate and neither of $H$ and $H'$ separates the other from $v$.

Already, from the definition of an event structure $\mathcal E$, one defines a graph $\mathcal G(\mathcal E)$ whose vertices are the events, with $e$ and $e'$ joined by an edge if and only if $e$ and $e'$ are independent, i.e. if and only if $e$ and $e'$ are concurrent (analogous to crossing hyperplanes) or in minimal conflict (analogous to osculating hyperplanes).  Hence $\mathcal G(\mathcal E)$ is a spanning subgraph of the contact graph $\Gamma({\bf X})$ that contains the crossing graph $\xing X$.  On the other hand, given a pointed CAT(0) cube complex ${\bf X}_v$, we see that the graph associated to the event structure $\mathcal E_v$ corresponding to $\mathbf X_v$ is precisely the pointed contact graph $\Gamma_v({\bf X})$.  It was noted already in~\cite{Santocanale} that a nice labeling of $\mathcal E$ corresponds to a colouring of the edges of the corresponding median graph in such a way that edges dual to the same hyperplane receive the same colour, and edges $c,c'$ dual to a pair of hyperplanes $H,H'$ that cross or osculate in a ``minimal conflict'' way receive different colours.  Thus we see that $\mathcal E$ admits a nice labeling if the corresponding pointed CAT(0) cube complex has finite chromatic number for its pointed contact graph, and in particular, $\mathcal E$ admits a nice labeling if the corresponding contact graph has finite chromatic number. Conversely, if $\Gamma_v({\bf X})$ has infinite chromatic number, then the corresponding event structure $\mathcal E_v$ does not admit a nice labeling.

\subsection{Disc diagrams}
We shall frequently use the technique of minimal-area disc diagrams in the CAT(0) cube complex $\bf X$.  For a discussion of disc diagrams over cube complexes, using the language and notation closest to that of the present paper, see~\cite{HagenQuasiArb} or~\cite{WiseIsraelHierarchy}.  The results which we shall use are summarized below.

A \emph{disc diagram} $D\rightarrow\bf X$ in the CAT(0) cube complex $\bf X$ is a combinatorial map, where $D$ is a contractible 2-dimensional cube complex, not necessarily CAT(0), such that $D$ is equipped with a specific embedding in $S^2$, so that $S^2=D\cup E$, where $E$ is a 2-cell.  The \emph{boundary path} of $D$ is the combinatorial path $P\rightarrow\bf X$ which is the restriction of $D\rightarrow\mathbf X$ to the attaching map of $E$.

\begin{proposition}[Existence of disc diagrams]\label{prop:discdiagramsexist}
Let $P\rightarrow G({\bf X})$ be a closed path.  Then there exists a disc diagram $D\rightarrow\bf X$ whose boundary path is $P$.
\end{proposition}

Let $D\rightarrow\bf X$ be a disc diagram.  The \emph{area} of $D$ is the number of 2-cubes in $D$.  The disc diagram has \emph{minimal area for its boundary path} $P$ if for all disc diagrams $D'$ with boundary path $P$, the area of $D'$ is at least the area of $D$.  Note that the equivalence relation $\Theta$ on the 1-cubes of $\bf X$ pulls back to an equivalence relation on the 1-cubes of $D$.  Each equivalence class $\Theta(H)$ of 1-cubes in $D$ determines a \emph{dual curve}, defined as follows.  If $s$ is a 2-cube of $D$, and $c,c'$ are opposite 1-cubes of $s$ with $c,c'\in\Theta(H)$, then the \emph{midcube} of $s$ corresponding to $H$ is the $\ell_1$ geodesic segment in $s$ joining the midpoint of $c$ to the midpoint of $c'$.  A \emph{dual curve} $K$ is a maximal concatenation of midcubes of 2-cubes in $D$.  The map $D\rightarrow\bf X$ restricts to a map $K\rightarrow H$, and moreover, the union $N(K)$ of all 2-cubes in $D$ containing constituent midcubes of $K$ -- the \emph{carrier of $K$} -- maps to the carrier $N(H)$.  A dual curve $K$ is an immersed interval or circle in $D$.

If $K$ is a dual curve, then an \emph{end} of $K$ is a midpoint $p$ of a 1-cube $c$ such that $p$ is contained in only one constituent midcube of $K$.  $K$ has at most two ends, and each end of $K$ lies on the boundary path $P$ of $D$.  The following proposition states, in the language of~\cite{WiseIsraelHierarchy}, that a minimal-area disc diagram does not contain a ``nongon'' or a ``monogon'' of dual curves, i.e. dual curves begin and end on the boundary path of $D$, and no dual curve crosses itself.

\begin{proposition}[No nongons or monogons]\label{prop:dualcurvesembed}
Let $D\rightarrow\bf X$ have minimal area for its boundary path $P$.  Then every dual curve $K$ in $D$ is an embedded interval (possibly of length 0), and in particular each 2-cube $s$ of $D$ contains 1-cubes of exactly two distinct equivalence classes in $\Theta$.  If $|K|>0$, then $K$ has exactly two ends on $P$.  If $K$ is a single point, then $K\in P$.
\end{proposition}

If $Q\subseteq P$ is a subpath of the boundary path of $D$, we say that $K$ \emph{emanates from} or \emph{ends on} $Q$ if $Q$ contains a 1-cube whose midpoint lies in $K$.

\begin{proposition}[No bigons]\label{prop:bigon}
If $D$ is a minimal-area disc diagram for its boundary path $P$, and $K_1$ and $K_2$ are distinct dual curves in $D$, then either $K_1\cap K_2=\emptyset$, or $K_1\cap K_2$ consists of a single point.  
In the latter case, $K_1$ and $K_2$ are said to \emph{cross}.  If $K_1$ and $K_2$ cross, then their corresponding hyperplanes also cross, and in particular $K_1$ and $K_2$ map to distinct hyperplanes.
\end{proposition}

Propositions~\ref{prop:dualcurvesembed} and~\ref{prop:bigon} are used implicitly in all of our disc diagram arguments.  The situation concerning trigons of dual curves is somewhat more subtle: a trigon of dual curves along the boundary path of $D$, for the diagrams we consider, in general contradicts minimality of the area of $D$ (except in certain special cases).  However, as in the proof of Theorem~\ref{theorem1}, if $\bf X$ is at most 2-dimensional, then all trigons of dual curves mapping to distinct hyperplanes are impossible, regardless of minimality of area, since the existence of pairwise-crossing triples of hyperplanes guarantees the presence of 3-cubes.

\begin{definition}[Trigon, trigon along the boundary]\label{def:trigon}
Let $D\rightarrow\bf X$ be a disc diagram with boundary path $P$.  If $K_1,K_2,K_3$ are distinct pairwise crossing dual curves in $D$, then they form a \emph{trigon of dual curves}, as at left in Figure~\ref{fig:trigons}.
Let $c_1$ and $c_2$ be distinct 1-cubes of $P$, and let $c_1Qc_2$ be one of the paths on $P$ subtended by $c_1$ and $c_2$.  For $i\in\{1,2\}$, let $K_i$ be the dual curve in $D$ emanating from $c_i$.  Suppose that $K_1$ and $K_2$ cross, and there exists a hyperplane $H$ such that the image of $c_1Qc_2$ lies in $N(H)$, with neither $c_1$ nor $c_2$ dual to $H$.  Then the pair $K_1,K_2$ forms a \emph{trigon along the boundary of $D$}.  See the middle of Figure~\ref{fig:trigons}.
\begin{figure}[h]
  \includegraphics[width=4in]{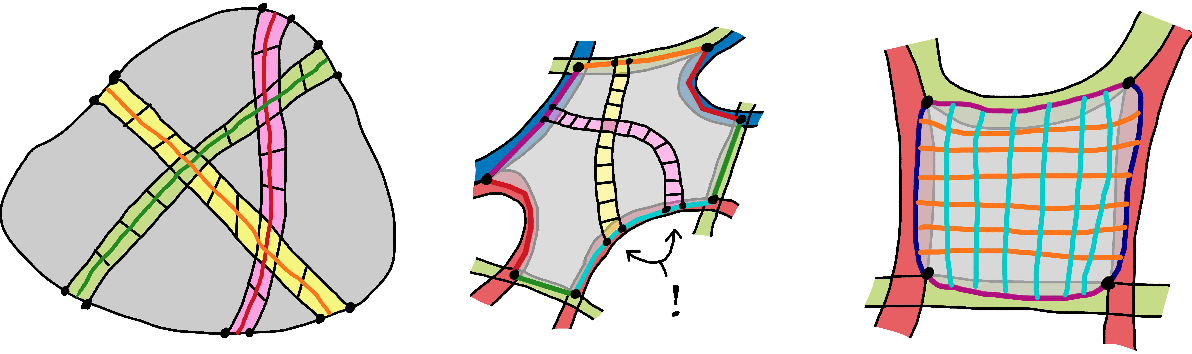}\\
  \caption{At left is a trigon of dual curves, which is in general possible in a minimal-area diagram, but which cannot occur in two dimensions.  In the center is a trigon along the boundary, which is always disallowed by minimality of area in our diagrams.  At right is a grid.}\label{fig:trigons}
\end{figure}
\end{definition}

Using \emph{hexagon moves}, one proves the first assertion in the following proposition.  The second follows from the fact that $\dimension{\bf X}$ is bounded below by the cardinality of any set of pairwise-crossing hyperplanes.

Every disc diagram $D$ in this paper has \emph{fixed carriers} in the sense of~\cite{HagenQuasiArb}.  This means that there is a fixed collection $\{H_i\}_{i=1}^k$ such that $\partial_pD=P_1P_2\ldots P_n$, where $P_i$ is a combinatorial geodesic segment in $N(H_i)$.  The diagram $D$ is minimal if it has minimal area among all diagrams with boundary path $\partial_pD$, and, fixing the collection $\{H_i\}$, the paths $P_i$ are chosen so as to minimize the area of $D$ among all diagrams thus constructed.  Finally, $D$ is chosen among all such minimal-area diagrams in such a way that $|\partial_pD|$ is minimal.

\begin{proposition}\label{prop:boundarytrigons}
Let $D$ be a diagram with fixed carriers that is minimal in the above sense.  Then $D$ contains no trigon $K_1,K_2,c_1Qc_2$ along the boundary.

If $\dimension{\bf X}\leq 2$, then any diagram $D\rightarrow\mathbf X$ that has minimal area for its boundary path contains no trigon of dual curves and no trigon along the boundary.
\end{proposition}



\begin{definition}[Grid]\label{def:trigon}

Denote by $[0,m]$ the tree with $m+1$ vertices and $m$ edges, such that each vertex has valence 1 or 2 (i.e. a subdivided line segment). A disc diagram $D$ is a \emph{grid} if $D\cong [0,m]\times[0,n]$ for some $m,n$.
\end{definition}

\begin{remark}\label{rem:discobservations}
Note that if $\bf X$ is 2-dimensional and $D$ is a minimal-area disc diagram in $\bf X$, then $D$ is itself a CAT(0) cube complex whose hyperplanes are the dual curves, since any triangle in the link of a 0-cube in $D$ would result in a trigon of dual curves, contradicting Proposition~\ref{prop:boundarytrigons}.

Moreover, if $H_1$ and $H_2$ are hyperplanes represented by dual curves $K_1$ and $K_2$ in $D$, and $D$ is minimal for a set of fixed carriers, and a subtended path $c_1Qc_2\subset P$ between $K_1$ and $K_2$, with $c_i$ dual to $K_i$, maps to the carrier of a single hyperplane $H$ whose carrier is one of the fixed carriers of $D$, then $H_1$ contacts $H_2$ if and only if $K_1$ contacts $K_2$.  Indeed, by the Helly property described below, $N(H_1)\cap N(H_2)\cap N(H)$ contains a 0-cube $p$.  Consider a path $R_1\rightarrow N(H_1)$ joining the initial 0-cube of $c_1$ to $p$, and a path $R_2\rightarrow N(H_2)$ joining $p$ to the terminal 1-cube of $c_2$. Then $R_1R_2(c_1Qc_2)^{-1}$ bounds a disc diagram $E$.  By choosing $E$ minimal relative to the fixed carriers $N(H_2),N(H_2),N(H)$, we find that we can replace $Q$ by a single 0-cube in $N(H_1)\cap N(H_2)\cap N(H)$, and thus replace $D$ by a lower-area diagram (with a shorter boundary path) and the same set of fixed carriers.
\end{remark}

\subsection{Isometric embeddings, convexity and the Helly property}
We briefly review some notions about isometric embeddings and convex subcomplexes of CAT(0) cube complexes.  As usual, the combinatorial map ${\bf Y}\rightarrow\bf X$ is an isometric embedding if the distance between any two points $x,y\in\bf Y$ (with respect to the wall-metric) is equal to the distance in $\bf X$ between the images of $x$ and $y$ in $\bf X$.  We have the following characterization of isometrically embedded subcomplexes of $\bf X$.  Recall that a set $\mathcal S$ of hyperplanes is \emph{inseparable} if for all $H_1,H_2\in\mathcal S$, any hyperplane separating $H_1,H_2$ belongs to $\mathcal S$.


\begin{proposition}\label{prop:isometricembedding}
Let ${\bf Y}\subseteq\bf X$ be an isometrically embedded subcomplex, and let $\mathcal H({\bf Y})$ be the set of hyperplanes crossing $\bf Y$.  Then $\mathcal H({\bf Y})$ is an inseparable set, and for all $H\in\mathcal H({\bf Y})$, the intersection $H\cap\bf Y$ is connected, and $N(H)\cap\mathbf Y$ is connected.

Conversely, let $\bf Y$ and $\bf X$ be locally finite CAT(0) cube complexes with hyperplane sets $\mathcal H({\bf Y})$ and $\mathcal H({\bf X})$ respectively.  Suppose there exists an injective graph homomorphism $\phi:\Gamma(\mathbf Y)\rightarrow\Gamma(\mathbf X)$ that is bijective on vertices and sends crossing edges to crossing edges. Suppose moreover that, if $U,V,W$ are hyperplanes of $\mathbf Y$ such that $V$ separates $U$ from $W$, then either $\phi(V)$ separates $\phi(U)$ from $\phi(W)$ or $\phi(V)$ crosses $\phi(U)$ or $\phi(W)$, and that if $V$ does not separate $U,W$, then $\phi(V)$ does not separate $\phi(U),\phi(W)$.  Then there is an isometric embedding $\mathbf Y\rightarrow\mathbf X$.
\end{proposition}

\begin{proof}
Suppose $H,H'\in\mathcal H({\bf Y})$ and that $H''$ separates $H$ from $H'$.  Then $H''$ must separate $H\cap \mathbf Y$ from $H'\cap \mathbf Y$, and hence each halfspace of $H''$ contains a nonempty subspace of $\mathbf Y$.  Let $y,y'\in H\cap \mathbf Y,H'\cap \mathbf Y$ be 0-cubes in distinct halfspaces of $H''$.  Since $\mathbf Y$ is isometrically embedded, there exists a geodesic segment $P\rightarrow \mathbf Y$ joining $y$ and $y'$, and $P$ must contain a 1-cube dual to $H''$.  Hence $H''$ crosses $\mathbf Y$ and $\mathcal H({\mathbf Y})$ is inseparable.

Suppose now that $H\in\mathcal H({\mathbf Y})$ is a hyperplane such that $N(H)\cap \mathbf Y$ is disconnected, and let $y,y'$ be 0-cubes in distinct components of $H\cap \mathbf Y$.  Let $P\rightarrow \mathbf Y$ be a geodesic segment joining $y$ to $y'$ and let $Q\rightarrow N(H)$ be a geodesic segment joining $y$ to $y'$.  Then $PQ^{-1}$ bounds a minimal-area disc diagram $D\rightarrow\bf X$, and since $|P|=|Q|$, each dual curve in $D$ travels from $P$ to $Q$.  No two dual curves emanating from $Q$ can cross, for otherwise there would be a trigon along the boundary lowering the area of $D$, and thus $P=Q$.  Hence $Q\subset \mathbf Y\cap N(H)$, and thus $y,y'$ actually belong to the same component of $N(H)\cap \mathbf Y$.  Hence $N(H)\cap \mathbf Y$ is connected, and thus $H\cap \mathbf Y$ is also.

We now prove the assertion about homomorphisms of contact graphs.  For a fixed $y_0\in\mathbf Y^{(0)}$, let $\mathcal H_0$ be the set of hyperplanes $H$ of $\mathbf Y$ with $y_0\in N(H)$.  Consider the set $\{\phi(H):H\in\mathcal H_0\}$.  This is a set of hyperplanes whose carriers pairwise-intersect, since $\phi$ is a homormophism of contact graphs.  Hence, since $\mathbf X$ is locally finite, $\cap_{H\in\mathcal H_0}N(\phi(H))$ contains a 0-cube $x_0$, and we let $f(y_0)=x_0$.

Next, if $y\in\mathbf Y$, then let $\gamma$ be a geodesic path joining $y_0$ to $y$.  Let $V_1,\ldots,V_n$ be the hyperplanes crossing $\gamma$, with $N(V_i)$ at distance $i-1$ from $y_0$ along $\gamma$.  Consider the hyperplanes $\phi(V_i)$ of $\mathbf X$.

First, $\{\phi(V_i)\}$ contains no \emph{facing triple}: for all distinct $i,j,k$, if $\phi(V_i),\phi(V_j),\phi(V_k)$ are pairwise non-crossing, then one of them separates the other two.  Indeed, since $\{V_i\}$ is the set of hyperplanes crossing a geodesic, it contains no facing triple.  If any two of $V_i,V_j,V_k$ cross, then their images under $\phi$ do as well, and there is nothing to prove.  Otherwise, we may assume that $V_j$ separates $V_i$ from $V_k$, whence our hypotheses about $\phi$ imply that either $\phi(V_j)$ separates $\phi(V_i)$ and $\phi(V_k)$, or $\phi(V_j)$ crosses $\phi(V_i)$ or $\phi(V_k)$.  Second, if for some $i,j$, the hyperplanes $\phi(V_i)$ and $\phi(V_j)$ are separated by a hyperplane $\phi(H)$, then $H=V_k$ for some $k$ between $i$ and $j$, i.e. $\{\phi(V_i):1\leq i\leq n\}$ is an inseparable set of hyperplanes in $\mathbf X$, by our hypothesis.  Finally, $\phi(V_i)$ and $\phi(V_k)$ are separated by $\phi(V_j)$ only if $V_j$ separates $V_i$ from $V_k$ in $\mathbf Y$.  Hence there is a unique geodesic path $\bar\gamma$ in $\mathbf X$ whose initial point is $y_0$, with the property that the set off hyperplanes crossing $\bar\gamma$ is exactly $\{\phi(V_i)\}$, with $\bar\gamma$ passing through $\phi(V_i)$ before $\phi(V_j)$ exactly when $i<j$.

Define a map $f:\mathbf Y\rightarrow\mathbf X$ by declaring $f(y)$ to be the terminal point of $\bar\gamma$ for each 0-cube $y\in\mathbf Y$.  The above construction shows that $d_{G(\mathbf X)}(f(y),f(y_0))=d_{G(\mathbf Y)}(y,y_0)$ for all $y$, since the left expression is equal to $|\bar\gamma|$ while the right is equal to $|\gamma|$, and the paths $\gamma,\bar\gamma$ are isometric.  Let the geodesic $\sigma$ of $\mathbf X$ join $f(y_1)$ and $f(y_2)$ for $y_1,y_2\in\mathbf Y$.  Then $|\sigma|=|\bar\gamma_1|+|\bar\gamma_2|-|\mathcal B|$, where $\gamma_i$ is a geodesic of $\mathbf Y$ joining $y_i$ to $y_0$ and $\mathcal B$ is the set of hyperplanes separating both $f(y_1)$ and $f(y_2)$ from $f(y_0)$.  But $\mathcal B$ is, by construction, the set of hyperplanes of the form $\phi(H)$, where $H$ separates $y_1,y_2$ from $y_0$.  Hence $|\sigma|=d_{G(\mathbf Y)}(y_1,y_2)$, so that $f$ is the desired isometric embedding.
\end{proof}

Note also that, since $G({\bf X})$ is an isometric subspace of $\bf X$ with the $\ell_1$ metric, an isometric embedding ${\bf X}\rightarrow\bf Y$ restricts to an isometric embedding $G({\bf X})\rightarrow\bf Y$.  Proposition~\ref{prop:isometricembedding} (see also~\cite{BandeltChepoiEppstein_squaregraph}) yields the equality $\tau({\bf X})=\chi(\xing X)$:

\begin{corollary}\label{cor:embeddingintrees}
The CAT(0) cube complex $\bf X$ (and hence $G({\bf X})$) isometrically embeds in a Cartesian product $\bf Y$ of at most $k$ trees ${\bf T}_1,\ldots,{\bf T}_k$ if and only if $\chi(\xing X)\leq k$.
\end{corollary}

\begin{proof}
Let ${\bf Y}=\prod_{i=1}^k{\bf T}_i$ be a product of trees $\mathbf T_i$.  Then $\xing Y$ is the join of $k$ totally disconnected graphs $\Gamma_i$, where $\Gamma_i$ is the crossing graph of ${\bf T}_i$.  Suppose there is an isometric embedding $\psi:{\bf X}\rightarrow\bf Y$.  Then there is an induced graph homomorphism $\xing X\rightarrow\xing Y$.  We colour $\xing X$ by assigning to each hyperplane $H$ the colour $i$ corresponding to the unique subgraph $\Gamma_i$ containing the image of $H$ in $\xing Y$.  If $H$ and $H'$ cross, then their images in $\bf Y$ also cross, and hence belong to distinct factors $\Gamma_i$.  Hence $H$ and $H'$ receive distinct colours, and this is thus a correct colouring in $k$ colours, i.e. $\chi(\xing X)\leq k$.

Conversely, let $c:\mathcal H\rightarrow\mathcal K$ be a correct colouring of $\xing X$ in the set $\mathcal K$ of $k$ colours.  For each $i\in\mathcal K$, let $\mathcal H_i=c^{-1}(i)$ be the set of hyperplanes with the colour $i$.  For each $i\in\mathcal K$, let ${\bf T}_i$ be the CAT(0) cube complex dual to the wallspace $({\bf X}^{(0)},\mathcal H_i)$.  Since $c$ is a correct colouring of the crossing graph, no two elements of $\mathcal H_i$ cross, and thus ${\bf T}_i$ is a tree.  Let ${\bf Y}=\prod_{i=1}^k{\bf T}_i$.  Each hyperplane of $\bf Y$ is of the form
\[{\bf H}(H,i)={\bf T}_1\times\ldots\times{\bf T}_{i-1}\times H\times{\bf T}_{i+1}\times\ldots\times{\bf T}_k\]
for some $H\in\mathcal H_i$ with $1\leq i\leq k$.  Moreover, ${\bf H}(H,i)$ and ${\bf H}(H',j)$ are distinct if $H\neq H'$ and cross if and only if $i\neq j$.  Furthermore, each $H\in\mathcal H$ gives rise to a hyperplane of this form, by construction.  Hence the identification $\mathcal H_i\ni H\mapsto {\bf H}(H,i)$ is a bijection yielding a graph homomorphism $\xing X\rightarrow\xing Y$ satisfying the separation hypotheses of Proposition~\ref{prop:isometricembedding}.  Thus, by Proposition~\ref{prop:isometricembedding}, there is an isometric embedding ${\bf X}\rightarrow\bf Y$.
\end{proof}

Convexity of a subcomplex ${\bf Y}\subseteq\bf X$ is characterized as follows: the subcomplex $\bf Y$ is convex if and only if, whenever $H$ and $H'$ are hyperplanes that cross $\bf Y$, either $H$ and $H'$ do not contact or $N(H)\cap N(H')\cap Y\neq\emptyset$ and, if $H$ and $H'$ cross, then $\mathbf Y$ contains a 2-cube representing this crossing.  In particular, the contact graph of $\bf Y$ is an induced subgraph of $\Gamma({\bf X})$ whose vertex set is inseparable.  From the point of view of median graphs, one verifies that, since $\bf Y$ is gated if it is convex, if $\Theta(H)$ and $\Theta(H')$ contain 1-cubes $c$ and $c'$ with a common 0-cube $v$, then the gate of $v$ in $\bf Y$ must lie in $N(H)\cap N(H')$.

Note also that $\bf X$ enjoys the \emph{Helly property}: if $Y_1,Y_2,\ldots,Y_n$ are convex subcomplexes of $\bf X$ such that $Y_i\cap Y_j\neq\emptyset$ for $i\neq j$, then $\bigcap_{i=1}^nY_i\neq\emptyset$.  This follows from the fact that convex subgraphs of $G({\bf X})$ are gated, or from a simple disc diagram argument (\cite{HagenQuasiArb}).

\subsection{Hyperplane-distance}\label{sec:hypdistance}
Let $U$ be a fixed hyperplane of $\mathbf X$.  In our applications, $U$ will be the central hyperplane of a specified sphere in $\Gamma({\bf X})$ of radius 2.

\begin{definition}[Hyperplane-distance]\label{defn:hyperplanedistance}
For each hyperplane $H$, the set nonempty $\mathfrak I=\{V_1,\ldots,V_n\}$ of hyperplanes is a \emph{separating chain for $H$} if
\begin{enumerate}
\item Each $V_i\in\mathfrak I$ separates $H$ from $U$.
\item If $V_i,V_j\in\mathfrak I$ are separated by a hyperplane $V$, then $V\in\mathfrak I$.
\item The hyperplanes in $\mathfrak I$ are pairwise non-crossing.
\item $\mathfrak J$ is not properly contained in a set of hyperplanes satisfying $(1)-(3)$.
\end{enumerate}
These properties ensure that the halfspaces of the $V_i\in\mathfrak I$ can be labeled $\mathbf A$ or $\mathbf B$ so that $\{\mathbf B(V_i)\}_{V_i\in\mathfrak I}$ is totally ordered by inclusion.  Let $\mathbb I(H)$ be the set of all separating chains for $H$.  The \emph{hyperplane-distance} of $H$ (with respect to $U$) is
\[d(H)=\min\{|\mathfrak I|:\mathfrak I\in\mathbb I(H)\}.\]
If $C$ is a collection of hyperplanes (or a subgraph of $\Gamma(\mathbf X))$, we let $D(C)=\sum_{H\in C}d(H)$.

Note that the combinatorial distance $d_{G(\mathbf X)}(N(H),N(U))$ in $G(\mathbf X)=\mathbf X^(1)$ between the carriers of $H$ and $U$ counts the hyperplanes that separate $U$ from $H$, since carriers are convex, and that $d(H)$ is bounded above by this quantity.  Note also that $d(H)=0$ if and only if $H\coll U$.
\end{definition}

\begin{lemma}\label{lem:cycledistance}
Let $H$ be a hyperplane with $\rho(U,H)=2$ and let $F$ be a hyperplane such that $U\coll F\coll H$.  Then there exists a path $P\rightarrow N(F)$ such that $d_{G(\mathbf X)}(N(H),N(U))=|P|$.  
\end{lemma}

\begin{proof}
Let $P\rightarrow N(F)$ be a shortest combinatorial path joining a 0-cube $a\in N(F)\cap N(H)$ to a 0-cube $b\in N(F)\cap N(U)$.  Let $Q\rightarrow G(\bf X)$ be a path realizing the distance from $N(H)$ to $N(U)$, with endpoints $c\in N(H)$ and $d\in N(U)$.  Let $A\rightarrow N(H)$ be a geodesic segment joining $a$ to $c$ and let $B\rightarrow N(U)$ be a geodesic segment joining $b$ to $d$.  Then there exists a minimal-area disc diagram $D\rightarrow\bf X$ with boundary path $PBQ^{-1}A^{-1}$.  Every dual curve in $D$ emanating from $P$ ends on $Q$, since otherwise there is a trigon of dual curves along the boundary path of $D$, since each of the paths $P,A,B$ maps to the carrier of a hyperplane.  Thus $|P|\leq |Q|$ and hence $|P|=|Q|=d_{G(\mathbf X)}(N(H),N(U))$ by minimality of $|Q|$.
\end{proof}

\begin{remark}
Note $d_{G(\mathbf X)}(N(H),N(U))>0$ if and only if $\rho(U,H)\geq 2$.  In our applications, Lemma~\ref{lem:cycledistance} is applied in such a way that $P$ lies in the \emph{father} of $H$, defined below.
\end{remark}

\subsection{Footprints and imprints}
In this section, we suppose that the CAT(0) cube complex ${\bf X}$ is 2-dimensional and in particular that each hyperplane of  ${\bf X}$  is a tree.  The carrier $N(H)$ of any hyperplane $H$ is bounded by two disjoint subcomplexes $H^+$ and $H^-$ which are both  isomorphic to $H$ and constitute convex and therefore gated subcomplexes of ${\bf X}.$ If $V$ is a hyperplane contacting $H,$  then the \emph{footprint} of $H$ in $V$ is $F(H,V)=N(V)\cap N(H)$. Any footprint is gated as the intersection of two gated subcomplexes. If $V$ and $H$ are osculating, then $F(H,V)$ is completely contained in $H^+$ or in $H^-.$ On the other hand, if $V$ and $H$ are crossing, then $F(V,H)$ contains the union of two isomorphic subcomplexes $F^+(H,V)=N(V)\cap H^+$ and $F^-(H,V)=N(V)\cap H^-$, each of which is a 1-cube, and $F(H,V)$ is a single 2-cube, since ${\bf X}$ is 2-dimensional. We call the projection of the footprint $F(H,V)$ on $V$ the {\it imprint} of $H$ on $V$ and denote it by $J(H,V).$ Denote by ${\mathcal F}(V)$ and ${\mathcal J}(V)$ the set systems consisting of all footprints $F(H,V)$ and  of all imprints of all hyperplanes $H$ such that  $H\coll V$.
We emphasize that, if $H$ and $H'$ are distinct hyperplanes that contact $V$, it may happen that $F(H,V)$ and $F(H',V)$ denote the same subcomplex of $N(V)$, but we regard them as distinct elements of $\mathcal F(V)$.

We begin our discussion of footprints and imprints with a consequence of the Helly property for hyperplanes.

\begin{lemma}\label{lem:contactingfootprints}
Let $H',H'',V$ be hyperplanes such that  $H'\coll V$ and $H''\coll V$.  Then $H'\coll H''$ if and only if $F(H',V)\cap F(H'',V)\neq\emptyset$. If $H'\coll H''$,
then $J(H',V)\cap J(H'',V)\ne\emptyset.$
\end{lemma}

\begin{proof} If $v\in F(H',V)\cap F(H'',V),$ then by definition of footprints we conclude that $v\in N(H')\cap N(H''),$ yielding $H'\coll H''.$ Conversely, if $H'\coll H'',$ then there exists $v\in N(H')\cap N(H'').$ Let $v_0$ be the gate of $v$ in $N(V).$ Pick $x'\in N(V)\cap N(H')$ and $x''\in N(V)\cap N(H'').$ Since the carriers $N(H')$ and $N(H'')$ are convex and $x',v\in N(H'),$ $x'',v\in N(H''),$ and $v_0\in I(v,x')\cap I(v,x''),$ we conclude that $v_0\in N(H')\cap N(H'').$ Since $v_0$ also belongs to $N(V),$ this implies that $v_0\in F(H',V)\cap F(H'',V).$ Finally, the projection $v'_0$ of $v_0$ in $V$ belongs to the imprints $J(H',V)$ and $J(H'',V).$
\end{proof}

\begin{lemma}\label{lem:footprintsdegree} For a hyperplane $V,$ any vertex $v$ of the 1-skeleton of $N(V)$ belongs to at most $\Delta$ footprints from the family ${\mathcal F}(V)$. In particular, $\delta({\mathcal F}(V))\le \Delta$ and $\delta({\mathcal J}(V))\le 2\Delta.$
\end{lemma}

\begin{proof} The degree of $v$ in $G({\bf X})$ and therefore in the 1-skeleton of $N(V)$ is at most $\Delta$. Consider the set ${\mathcal H}_v$ of all
hyperplanes $H$ such that $v\in F(H,V).$ If $H\in {\mathcal H}_v$ crosses the hyperplane $V$, then the equivalence class $\Theta(H)$ of $H$ contains
an edge $e_H$ incident to the vertex $v$ and belonging to $N(V).$   Analogously, if the hyperplanes $H\in {\mathcal H}_v$ and $V$ osculate, then  any vertex of
$F(H,V),$ in particular $v,$ is incident to an edge $e_H$ of $\Theta(H)$ (in this case $e_H$ does not belong  to $N(V)$). Two edges $e_H,e_{H'}$ defined by two different
hyperplanes $H,H'\in {\mathcal H}_v$ are different because they belong to two different equivalence classes of the relation $\Theta$. Thus $|{\mathcal H}_v|\le \Delta,$ establishing that $\delta({\mathcal F}(V))\le \Delta.$ Since each vertex $v_0$ of the tree $V$ is the image of two vertices of the 1-skeleton of $N(V)$, which belong to at most $\Delta$ footprints each,
$v_0$ belongs to at most $2\Delta$ imprints.
\end{proof}

\begin{proposition}\label{lem:footprintscolours}
$\chi({\mathcal F}(V))\le \chi({\mathcal J}(V))\le 2\Delta,$ where $\chi(\mathcal F(V))$ is the chromatic number of the intersection graph of $\mathcal F(V)$.
\end{proposition}

\begin{proof} The inequality $\chi({\mathcal F}(V))\le \chi({\mathcal J}(V))$ is obvious because two intersecting footprints give rise to two intersecting imprints. It is well-known (see, for example, \cite{Golumbic_book},~\cite{Gyarfas}) that if ${\mathcal F}$ is a family of subtrees of a tree $T,$ then $\chi({\mathcal F})=\omega({\mathcal F})$ and, since ${\mathcal F}$ satisfies the Helly property, $\omega({\mathcal F})=\delta({\mathcal F}).$ Since ${\mathcal J}(V)$ is a family of subtrees of the tree $V$ and $\delta({\mathcal J}(V))\le 2\Delta$ by Lemma \ref{lem:footprintsdegree}, we conclude that $\chi({\mathcal J}(V))\le 2\Delta.$
\end{proof}

\section{Canonical paths, grandfathers, and the weak combing property}
Choose, once and for all, an arbitrary but fixed \emph{base hyperplane} $H_0\in\mathcal H$.  For any $H\in\mathcal H$, the \emph{grade} of $H$ is $g(H)=\rho(H,H_0)$.

\begin{definition}[Ball, sphere, cluster]\label{defn:fullsphere}
For each $r\geq 0$, the \emph{(full) ball} $B_r:=B_r(H_0)$ is the full (i.e., induced) subgraph of $\Gamma(\bf X)$ generated by the set of hyperplanes $H$ with $g(H)\leq r$.  The \emph{(full) sphere} $S_r:=S_r(H_0)$ is the full subgraph of $\Gamma(\bf X)$ generated by the set of hyperplanes $H$ with $g(H)=r$.

Let $H,H'\in S_r$ be hyperplanes.  Then $H\sim H'$ if and only if there exists a path $P$ in $\Gamma(\bf X)$ joining $H$ to $H'$ such that every vertex of $P$ corresponds to a hyperplane of grade at least $r$.  This defines an equivalence relation on the grade-$r$ hyperplanes.  An equivalence class of hyperplanes of grade $r$ is called a \emph{grade-$r$ cluster}.
\end{definition}

The notion of a \emph{realization} allows us to translate statements about paths in $\Gamma(\bf X)$ into statements about paths in $\bf X$.  More specifically, note that if $H_0\coll H_1\coll H_2$ is a path in $\Gamma(\bf X)$, then we have a path $c_0Pc_2$ in $\bf X$, where $c_0,c_2$ are 1-cubes dual to $H_0$ and $H_2$, respectively, and $P\rightarrow N(H_1)$ is a combinatorial path joining a 0-cube of $N(H_0)\cap N(H_1)$ to a 0-cube of $N(H_1)\cap N(H_2)$.

\begin{definition}[Realization, canonical path]\label{defn:canonicalpath}
Let $H_0\coll H_1\coll H_2$ be a path in $\Gamma$.  An edge-\emph{realization} of $H_0\coll H_1\coll H_2$ is a combinatorial geodesic $P\rightarrow N(H_1)$ that joins $N(H_0)$ to $N(H_2)$.  If $\gamma=H_0\coll H_1\coll H_2\coll\ldots\coll H_r$ is an embedded path in $\Gamma$, then a \emph{realization} of $\gamma$ is a path $\re(\gamma)=R_1R_2\ldots R_{r-1}$, where each $R_i$ is a realization of the path $H_{i-1}H_iH_{i+1}$.

Let $\gamma=H_0\coll H_1\coll\ldots\coll H_r=H$ be a geodesic path in $\Gamma(\bf X)$.  The \emph{weight} $||\gamma||$ of $\gamma$ is the ordered $r$-tuple $(|R_{r-1}|,|R_{r-2}|,\ldots,|R_1|)$, where $\re(\gamma)=R_1R_2\ldots R_{r-1}$ is a realization such that the previous $r$-tuple is minimal in the lexicographic order as $\re(\gamma)$ varies among realizations of $\gamma$.

A path $\gamma(H)=H_0\coll H_1\coll\ldots\coll H_r=H$ is a \emph{canonical path} for $H$ if $||\gamma(H)||$ is minimal (in the lexicographic order) among all paths $\gamma$ of $\Gamma(\bf X)$ joining $H_0$ to $H$.
The hyperplane
$f^2(H)=H_{r-2}$ is called  the {\it grandfather} of $H$ (with respect to $\gamma(H)$).
\end{definition}

Figure~\ref{fig:generalcombing} contains heuristic pictures of realizations.  Given a grade-$r$ hyperplane $H=H_r$, there are in general many canonical paths joining $H_0$ to $H_r$.  Figure~\ref{fig:realization} shows a grade-3 hyperplane in a CAT(0) cube complex, and two distinct canonical paths, along with their realizations.  Figure~\ref{fig:realization} also shows that, in general, a given path in $\bf X$ may realize many paths in $\Gamma(\bf X)$.
\begin{figure}[h]
  \includegraphics[width=4in]{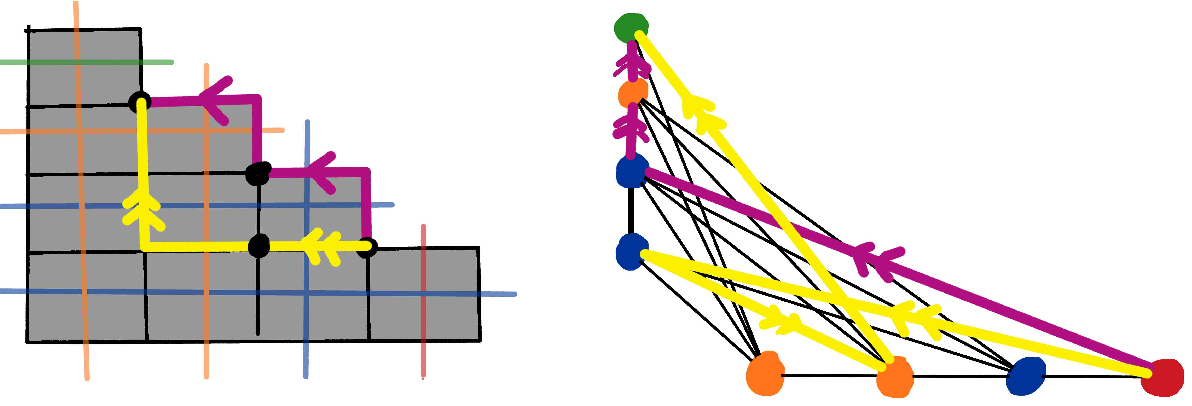}\\
  \caption{$\bf X$ is shown at left and the contact graph $\Gamma(\bf X)$ at right.  Arrowed paths in $\bf X$ are least-weight realizations of the correspondingly-arrowed paths in $\Gamma(\bf X)$.}\label{fig:realization}
\end{figure}

\begin{proposition} 
\label{prop:weakcombing}
Let $H,H'$ be two hyperplanes belonging to a common grade-$r$ cluster of $\Gamma({\bf X}),$ with $r\geq 2$, and let  $\gamma(H),\gamma(H')$ be two canonical paths, respectively joining $H_0$ to $H=H_r$ and to $H'=H'_r.$   Then the grandfathers  $f^2(H),f^2(H')$ of $H$ and $H'$ in $\gamma(H)$ and $\gamma(H')$ either coincide or contact, i.e.,  either $H_{r-2}= H'_{r-2}$  or $H_{r-2}\coll H'_{r-2}$.
\end{proposition}

\begin{proof}
First, note that the claim is obviously true for $r=2$, since in that case $H_{r-2}=H'_{r-2}=H_0$, so assume $r\geq 3$ and assume that $H_{r-2}\neq H'_{r-2}$.

\textbf{The disc diagram $D$:}  Let $\re(\gamma(H))=R_1R_2\ldots R_{r-1}$ and $\re(\gamma(H'))=R'_1R'_2\ldots R'_{r-1}$ be least-weight realizations of $\gamma(H)$ and $\gamma(H')$ respectively, so that $R_i\rightarrow N(H_i)$ and $R'_i\rightarrow N(H'_i)$ are combinatorial geodesics for each $i$.  Let $P_0\rightarrow N(H_0)$ be a combinatorial geodesic joining the initial 0-cubes of $R_1$ and $R'_1$.

Since $H$ and $H'$ belong to the same grade-$r$ cluster, then by definition there exists a shortest path $H=V^0\coll V^1\coll V^2\ldots\coll V^k=H'$ joining $H$ to $H'$, and $g(V_i)\geq r$ for $0\leq i\leq k$.  Hence there is a concatenation $Q=Q_0Q_1\ldots Q_k$, where $Q_i\mapsto N(V_i)$ is a combinatorial geodesic, joining the terminal 0-cube of $R_{r-1}$ to that of $R'_{r-1}$.  Hence we have a closed, piecewise-geodesic path
\[A=\left(\prod_{i=1}^{r-1}R_i\right)Q\left(\prod_{i=1}^{r-1}R'_i\right)^{-1}P_0^{-1}\rightarrow G(\bf X).\]
Let $D\rightarrow\bf X$ be a minimal-area disc diagram with boundary path $A$.  This notation is illustrated in Figure~\ref{fig:generalcombing}.
\begin{figure}[h]
  \includegraphics[width=4in]{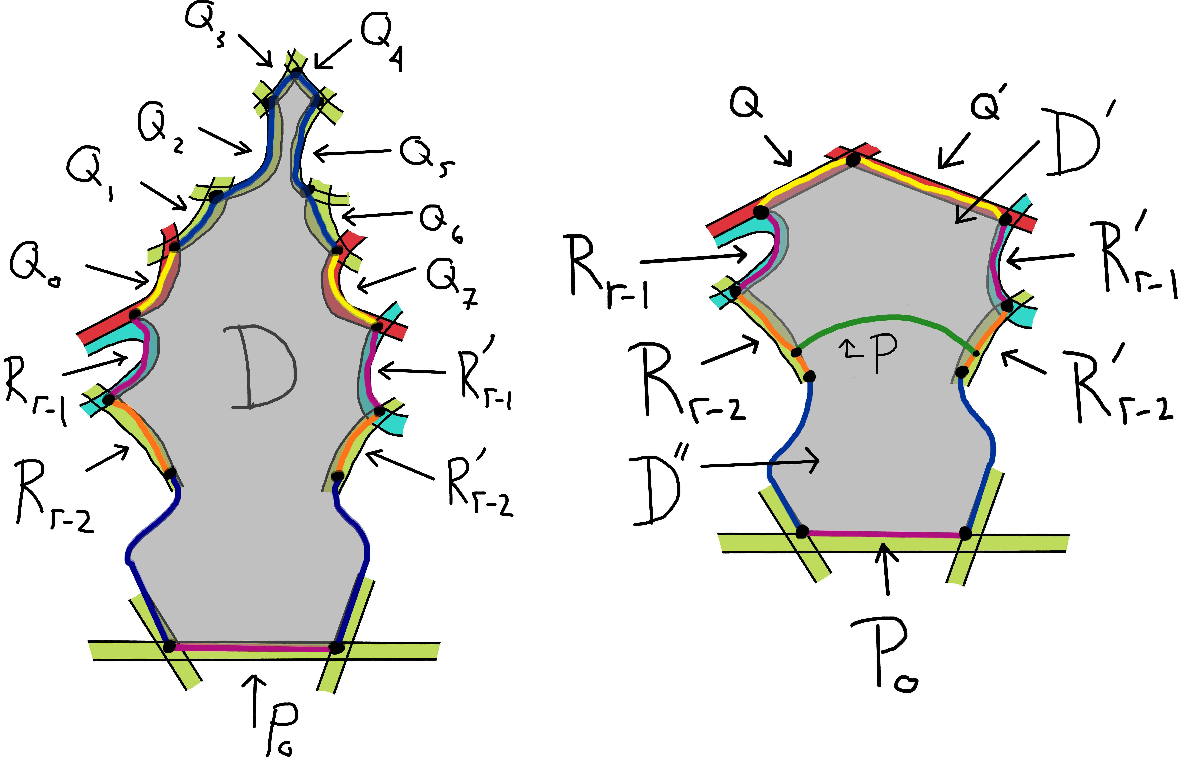}
  \caption{The disc diagram $D$ in the proof of Proposition \ref{prop:weakcombing} is shown at left.  The hyperplane carriers containing the various named subpaths of the boundary path of $D$ are shown.  At right is the same diagram, drawn for simplicity in the case where $H$ and $H'$ contact, showing the path $P$ and the resulting subdiagrams $D'$ and $D''$.}\label{fig:generalcombing}
\end{figure}

\textbf{The path $P$ of $G(\bf X)$ and the subdiagrams $D'$ and $D''$:}  By Lemma~\ref{lem:supportingE} below, there exists a combinatorial path $P\hookrightarrow D$ whose endpoints lie on $R_{r-2}$ and $R'_{r-2}$, with the property that every dual curve in $D$ crosses $P$ at most once, and no dual curve that crosses $P$ emanates from $R_{r-2}$ or $R'_{r-2}$.  Note that $P$ separates $D$ into two disc diagrams, i.e. $D=D'\cup_PD''$, where $D'$ is the subdiagram containing $Q$ and $D''$ is the subdiagram containing $P_0$, as shown at right in Figure~\ref{fig:generalcombing}.

\textbf{Analysis of $D'$:}  Let $K$ be a dual curve in $D'$ emanating from $P$ and mapping to a hyperplane $W$.  Then there is a dual curve $L$ in $D$ such that $L\cap D'=K$.  Since no dual curve crosses more than one 1-cube of $P$, and no dual curve crossing $P$ ends on $R_{r-2}$ or $R'_{r-2}$, the dual curve $L$ has exactly one end on the boundary path of $D'$, i.e. on $R_{r-1},R'_{r-1}$ or $Q$, and one end on the boundary path of $D''$, on $P_0$ or $R_i$ or $R'_i$, with $i\leq r-3$.

Since $r\geq 3$, the end of $L$ on the boundary path of $D''$ cannot be on $P_0$ or on $R_i$ or $R'_i$ with $i<r-3$, for otherwise $W$ would contact $H_i$ or $H'_i$, with $i<r-3$, and also contact $H_{r-1}$ or $H'_{r-1}$ or $V^i$, contradicting the fact that canonical paths are geodesics in $\Gamma(\bf X)$.  Similarly, $L$ cannot end on $Q$, and hence $L$ travels from $R_{r-3}$ to $R'_{r-1}$ or to $R_{r-1}$, or,  when $r=3$, from $P_0$ to $R_{r-1}$ or $R'_{r-1}$, as shown in Figure~\ref{fig:weakcombing2}.
\begin{figure}[h]
  \includegraphics[width=1.25in]{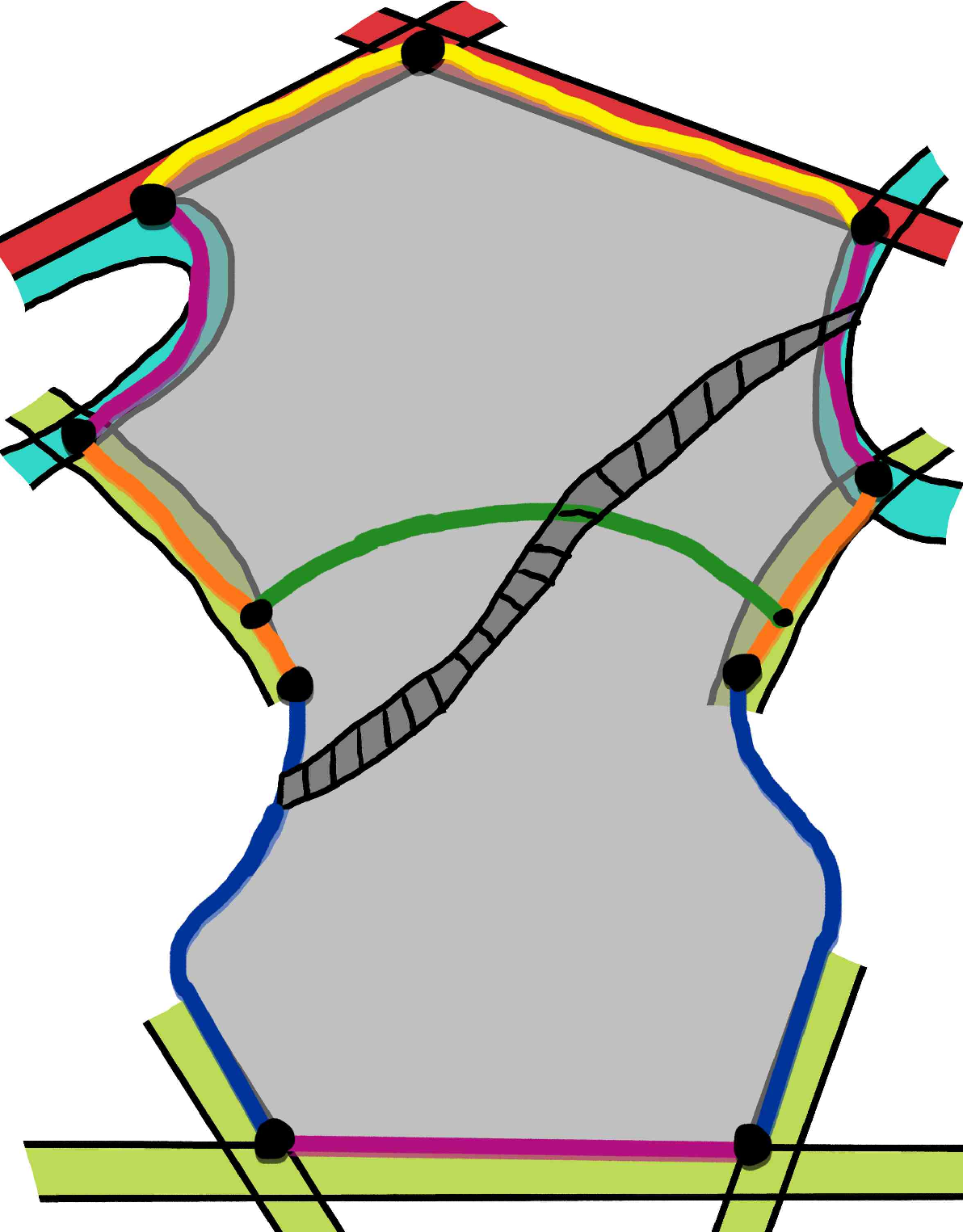}\\
  \caption{The path $P$ and the carrier of a dual curve $L$ that crosses $P$.}\label{fig:weakcombing2}
\end{figure}

Note that, were $L$ to travel from $R_{r-3}$ to $R_{r-1}$, then, as in Figure~\ref{fig:weakcombing3}, we would have $|S|=|R_{r-2}|$ and thus $W$ could replace $H_{r-2}$ in $\phi(H)$, leading to a lower-weight path, contradicting the fact that $\gamma(H)$ is canonical. Indeed, the subdiagram between $S,R_{r-2}$ and the subtended parts of $R_{r-3}$ and $R_{r-1}$ is a grid, since $D$ is of minimal area, and thus $|S|=|R_{r-2}|$.  Hence we may assume that $L$ travels from $R_{r-3}$ to $R'_{r-1}$.

Let $S$ be the path on the carrier $N(L)$ of $L$ that is isomorphic to $L$ and is separated from $R'_{r-3}$ by $L$.  Note that the 1-cube of $R'_{r-1}$ dual to $L$ cannot be the terminal 1-cube of $R_{r-1}$.  Indeed, the hyperplane $W$ has grade at most $r-2$ since $L$ emanates from $R_{r-3}$, and hence $W$ cannot contact the grade-$r$ hyperplane $V'$.  On the other hand, the 0-cube of $S$ on $R'_{r-1}$ is the terminal 0-cube of a 1-cube contained in $R'_{r-1}$, and hence the subpath $S'_{r-1}\subset R'_{r-1}$ subtended by $S$ and $Q_k$ satisfies $|S'_{r-1}|<|R'_{r-1}|$.  We thus have a path $H_0\coll H_1\coll\ldots\coll H_{r-3}\coll W\coll H'_{r-1}\coll H'$.  This path has weight at most
\begin{eqnarray*}
(|S'_{r-1}|,|S|,\ldots)&<&(|R'_{r-1}|,|R'_{r-2}|,|R'_{r-3}|,\ldots)\\
&=&||\re(\gamma(H'))||
\end{eqnarray*}
since $|S'_{r-1}|<|R'_{r-1}|$.  This contradicts that $\phi(H')$ is a canonical path.

\begin{figure}[h]
  \includegraphics[width=1.5in]{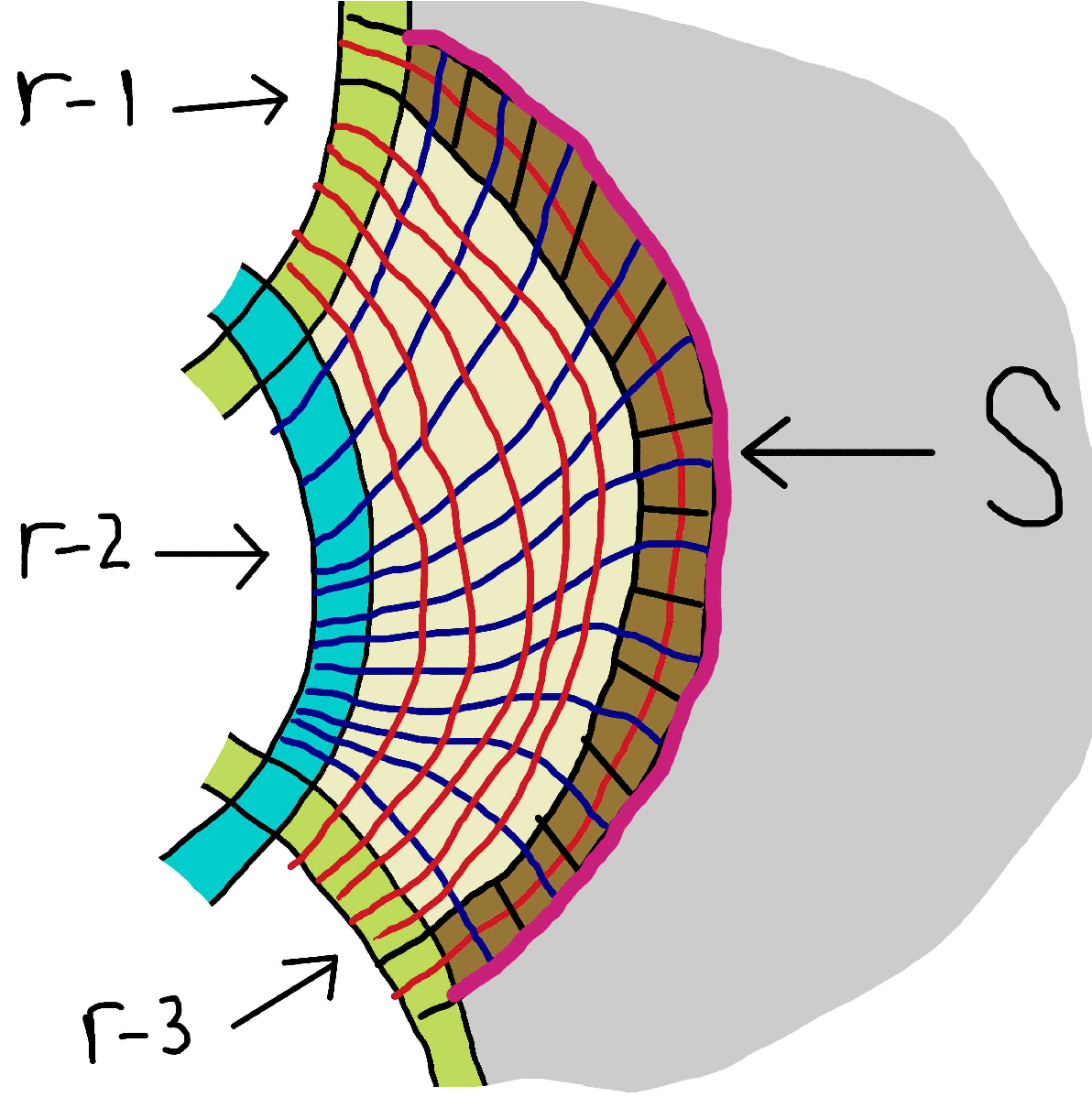}\\
  \caption{When $L$ travels from $R_{r-3}$ to $R_{r-1}$, we obtain a lower-weight path.}\label{fig:weakcombing3}
\end{figure}

\textbf{Conclusion:}  Since any dual curve in $D'$ emanating from $P$ leads to a contradiction either of minimality of the area of $D$ or of the fact that $\gamma(H)$ or $\gamma(H')$ is canonical, we conclude that $|P|=0$, and hence that $H_{r-2}\coll H'_{r-2}$.  This contact is in fact visible in the diagram $D$ -- see Figure~\ref{fig:weakcombing2}.
\end{proof}

\begin{lemma}\label{lem:supportingE}
Using the notation of Proposition~\ref{prop:weakcombing}, there exists a path $P\rightarrow D\rightarrow \bf X$ such that
\begin{enumerate}
\item $P$ joins a 0-cube of $R_{r-2}$ to a 0-cube of $R'_{r-2}$.
\item Each dual curve in $D$ is dual to at most a single 1-cube of $P$.
\item No dual curve in $D$ that crosses $P$ has an end on $R_{r-2}$ or $R'_{r-2}$.
\end{enumerate}
\end{lemma}

\begin{proof}
Choose a shortest path $P\rightarrow D^{(1)}$ in the 1-skeleton $D^{(1)}$ of $D$ that joins a 0-cube of $R_{r-2}$ to a 0-cube of $R'_{r-2}$.  We first modify $P$, without affecting its endpoints, so that (2) is satisfied.  We then show that $P$ satisfies (3).

\textbf{Modifying $P$ to satisfy (2):}  Let $K$ be a dual curve in $D$ that is dual to two distinct 1-cubes $c,c'$ of $P$.  Moreover, suppose that $K$ is an innermost such dual curve, in the sense that no two 1-cubes between $c$ and $c'$ on $P$ are dual to the same dual curve.  Consider the path $T$ on $N(K)$ traveling from the initial 0-cube of $c$ to the terminal 0-cube of $c'$.  Then $T$ and $cP'c'$ bound a subdiagram $E$, where $P'$ is the subtended part of $P$; see the left picture in Figure~\ref{fig:supportingE}.  Since $K$ is innermost, every dual curve in $E$ travels from $P'$ to $T$.  Indeed, the only other possibility is a dual curve $L$ dual to at least two distinct 1-cubes of $T$, but that would lead to a bigon between $K$ and $L$, contradicting minimality of the area of $D$.  Hence $|T|=|cP'c'|$, and we replace $P$ by a new path, with the same endpoints, in which $cP'c'$ is replaced by $T$.  This lowers the number of dual curves that cross $P$ in more than one way, and thus in finitely many such steps we arrive at a choice of $P$ satisfying (2).

\textbf{Verifying (3):}  Let $C$ be a dual curve in $D$ that emanates from $R_{r-2}$ and crosses $P$, as at right in Figure~\ref{fig:supportingE}.  Let $P=P'cP''$, where $c$ is the 1-cube of $P$ dual to $C$ and $P'$ is the subpath of $P$ joining the initial 0-cube of $P$ to the initial 0-cube of $c$.  Let $Tc'$ be the subpath of $R_{r-2}$ between the initial 0-cube of $P$ and the 1-cube $c'$ of $R_{r-2}$ dual to $C$.  Let $F$ be the subdiagram of $D$ bounded by $Tc',P'c$, and $S$, where $S$ is the shortest path on the carrier of $C$ that joins the endpoints of $c$ and $c'$.

No dual curve in $F$ emanating from $T$ can end on $S$, since that would lead to a trigon of dual curves along the boundary path of $D$ and a consequent reduction in area.  Hence, as shown in Figure~\ref{fig:supportingE}, dual curves in $F$ travel from $S$ to $P'$ or from $T$ to $P'$, or from $c$ to $c'$.  The former type shows that $|S|\leq |P'|$, with equality if and only if $|T|=0$.  We thus have that $|SP''|\leq|P'|+|P''|<|P'cP''|=|P|$, contradicting the assumption that $P$ was a shortest path joining $R_{r-2}$ to $R'_{r-2}$.  Indeed, $SP''$ has its endpoints on $R_{r-2}$ and $R'_{r-2}$ since $c'$ is a 1-cube of $R_{r-2}$ and $P''$ is the terminal subpath of $P$.
\begin{figure}[h]
  \includegraphics[width=3in]{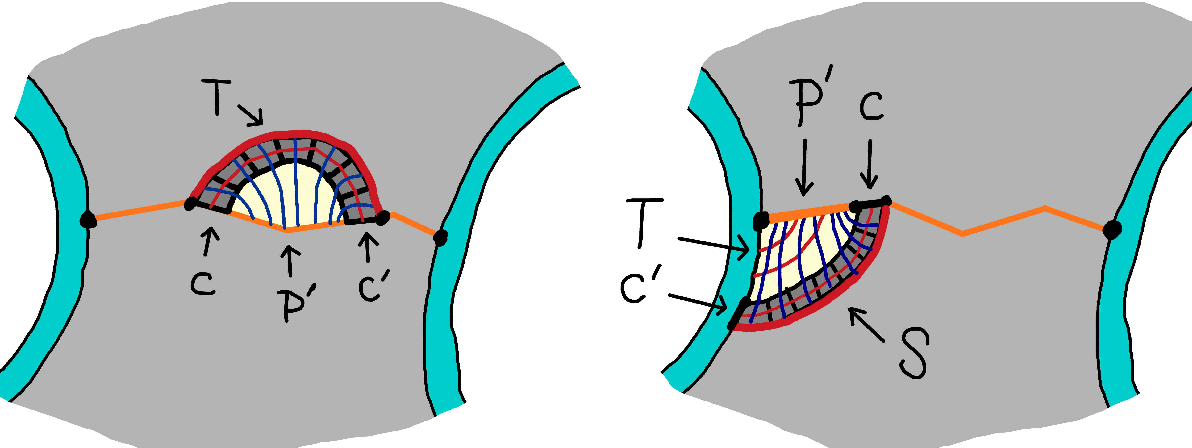}\\
  \caption{Left to right: the subdiagrams $E$ and $F$ of $D$.}\label{fig:supportingE}
\end{figure}
\end{proof}

\begin{remark}\label{rem:comment13}
In the case where $\bf X$ is 2-dimensional, the second part of the proof of Lemma~\ref{lem:supportingE} can be simplified slightly, using the fact that a minimal-area disc diagram in a 2-dimensional CAT(0) cube complex is itself a CAT(0) cube complex.  Although this ceases to be true in higher dimensions, the proof given above works for arbitrary CAT(0) cube complexes, and indeed the weak combing property established by Proposition~\ref{prop:weakcombing} and Corollary~\ref{cor:weakcombing}, as well as the bound on the diameters of clusters in $\Gamma(\mathbf X)$ established in Corollary~\ref{cor:diameter}, holds for any CAT(0) cube complex.
\end{remark}

Applying  Proposition \ref{prop:weakcombing} to a pair $H,H'$ of contacting hyperplanes of the same grade, we obtain the following property of grandfathers, which is used in colouring $\Gamma$:

\begin{corollary} [Weak combing]\label{cor:weakcombing} The grandfathers in canonical paths of two contacting hyperplanes of the same grade either coincide or contact.
\end{corollary}

Since the distance in $\Gamma({\bf X})$ from a hyperplane to its grandfather is 2, from Proposition \ref{prop:weakcombing} we also immediately obtain:

\begin{corollary} [Diameter of clusters] \label{cor:diameter} The diameter of each cluster in the contact graph  $\Gamma(\bf X)$ is at most 5.
\end{corollary}

\begin{remark} There are finite, 2-dimensional CAT(0) cube complexes whose contact graphs contain clusters of diameter 5; thus Corollary \ref{cor:diameter} is sharp.
\end{remark}


\section{Potential fathers, iterated footprints and imprints}
For a hyperplane $U$  of grade $r-2$, and a fixed cluster $\mathcal C$ of grade $r$, denote by $\mathcal R(U)={\mathcal R}(U,\mathcal C)$ the set of all hyperplanes $H$ in $\mathcal C$ such that $U$ is the grandfather of $H$ in a fixed canonical path $\gamma(H),$ i.e., $f^2(H)=U.$ As before, $U^+$ and $U^-$ denote the two copies of $U$ bounding the carrier $N(U).$ For a hyperplane $H\in {\mathcal R}(U),$ denote by $PF(H)$ the set of all hyperplanes $V$ which contact at the same time $H$ and its grandfather $U$ and call any such hyperplane $V$ a {\it potential father} of $H.$ Let $\pf(H)$ denote the union of carriers $N(V)$, where $V$ varies over the set of potential fathers of $H$. The \emph{iterated footprint} of $H$ on its grandfather $U=f^2(H)$ is the subcomplex $IF(H,U)=\pf(H)\cap N(U)$. Analogously, the {\it iterated imprint} $IJ(H,U)$ of $H$ on $U$ is the projection of $IF(H,U)$. Denote by ${\mathcal F}(U)$ and ${\mathcal J}(U)$ the set families consisting of all iterated footprints $IF(H,U)$ and imprints $IJ(H,U)$ taken over all hyperplanes $H$ having $U$ as the grandfather.  
When the grandfather $U=f^2(H)$ is a fixed hyperplane, we use the notation $IF(H,U)=IF(H)$ for the iterated footprint of $H$ on $U$.

\begin{lemma}\label{lem:potentialpatherconnected}
Let $H\in {\mathcal R}(U).$
Then the subcomplex $\pf(H)$ is connected and hence the iterated footprint $IF(H)$ is a connected subcomplex of $N(f^2(H))$. In particular, the iterated imprint $IJ(H)$ is a
subtree of $U.$

Finally, if $H,H'\in\mathcal R(U)$ contact, then $IF(H)\cap IF(H')\neq\emptyset$, and hence the iterated imprints of $H$ and $H'$ in $U$ intersect in a subtree.
\end{lemma}

\begin{proof}
We first show that the set $PF(H)$ is inseparable, i.e. that if $V,V'\in PF(H)$ and $V''$ separates $V$ from $V'$, then $V''\in PF(H)$.  Indeed, let $P\rightarrow N(U)$ be a geodesic joining a closest pair of 0-cubes of $F(V,U)$ and $F(V',U)$, let $Q\rightarrow N(H)$ be a shortest geodesic of $N(H)$ joining $F(V,H)$ to $F(V',H)$.  Note that $P$ and $Q$ are necessarily disjoint since $U$ and $H$ do not contact.  Hence the shortest paths $R,R'\rightarrow N(V),N(V')$ joining the initial and terminal 0-cubes of $P,Q$, respectively, have length at least 1.  Likewise, since $V''$ separates $V$ and $V'$, it must separate $F(V,U)$ and $F(V',U)$ and also $F(V,H)$ and $F(V',H)$, and hence $V''$ crosses $P$ and $Q$ and hence crosses $H$ and $U$, and thus $V''\in PF(H)$.

Let $D$ be a minimal-area disc diagram with boundary path $RQ(R')^{-1}P^{-1}$.  By minimality of area, dual curves in $D$ travel from $R$ to $R'$ or from $P$ to $Q$.  If $C$ is a dual curve traveling from $P$ to $Q$, then $C$ maps to a hyperplane $V''$ that crosses $U$ and $H$, and hence $V''\in PF(H)$.  Thus the 1-cube $c\subset P$ dual to $C$ lies in $N(V'')\subset IF(H)$, and hence $P\subset IF(H)$.  If there is no such dual curve $C$, then $|P|=0$ and $N(V)\cap N(V')\neq\emptyset$.  Thus $IF(H)$ is connected.  The projection $N(U)\rightarrow U$ preserves connectedness, and hence $IJ(H)$ is a connected subtree of $U$.

Finally, if $H,H'$ contact, then by Lemma~3.5 of~\cite{HagenQuasiArb}, either $H$ and $H'$ have a common potential father $V$, or there exist potential fathers $V,V'$ of $H$ and $H'$ respectively such that $V\coll V'$.  In the first case, $F(V,U)$ belongs to the iterated footprint of both $H$ and $H'$, and in the second case, $F(V,U)\cap F(V',U)\neq\emptyset$ since $U$ is convex.
\end{proof}

For a hyperplane $U,$  fix once and for all a vertex $b^*$ of $U$ as a root of the tree $U.$ Among the potential fathers of a hyperplane $H\in {\mathcal R}(U),$ pick a hyperplane $V$ whose imprint $J(V,U)$ is closest to $b^*$, i.e., $d(b^*,IJ(H,U))=d(b^*,J(V,U))=\min \{ d(b^*,J(V',U)): V'\in PF(H)\}$  (the distance $d(b^*,J(V,U))$ is measured according to the usual combinatorial distance in a tree $U$ between a vertex and a subtree of $U$). Additionally, if there exist several potential fathers of $H$ whose imprints have the same minimal distance to $U,$ then let $V$ be that potential father for which the
imprint $J(H,V)$ is closest to $J(V,U).$  If there are several such hyperplanes $V$, choose one arbitrarily. Set $f(H)=V$ and call it the {\it father} of $H.$ The vertex $b_H$ of $IJ(H,U)$ realizing the distance $d(b^*,IJ(H,U))$ is called the {\it root} of the iterated imprint $IJ(H,U).$
(Note that the path  $\gamma^*(H)=H_0\coll H_1\coll\ldots\coll U=H_{r-2}\coll f(H)\coll H$  obtained from $\gamma(H)$ by replacing the hyperplane $H_{r-1}$  by the father $f(H)$ is a geodesic between $H_0$ and $H$ in $\Gamma({\bf X})$ but is not necessarily a canonical path.) On ${\mathcal R}(U)$ we define a partial order $\prec$  by setting $H\prec H'$ if and only if $b_{H}\ne b_{H'}$ and $b_{H}$ belongs to the unique path of $U$ between $b^*$ and $b_{H'}$ (in this case we will also write $b_{H'}\prec b_{H}$) and breaking ties arbitrarily when  $b_{H}=b_{H'}$.

\begin{remark}\label{rem:welldefinedfathers}
We briefly review the logic of the choice of fathers.  Recall that we have fixed a base hyperplane $H_0$ and graded $\Gamma(\mathbf X)$ with respect to $H_0$.  We then chose, for each hyperplane $H$, a canonical path $\gamma(H)$ joining $H_0$ to $H$.  This choice uniquely determines a grandfather $f^2(H)$ for each hyperplane $H$ of grade at least 2.  For any hyperplane $U$, there is therefore a well-defined set $\mathcal R(U)$ containing those hyperplanes $H$ for which, with respect to our fixed choice of canonical paths, $U=f^2(H)$.

We then focus on a single hyperplane $U$, and fix a base vertex $b^*$ in the tree $U$.  The \emph{father} $f(H)$ of $H\in\mathcal R(U)$ is a hyperplane $V=f(H)$ such that $V\coll U,V\coll H$, and no hyperplane $V'$ satisfying these criteria has imprint on $U$ closer to $b^*$ than does $V$.  Note that there could be more than one hyperplane $V$ satisfying these criteria.  In this case, we choose \emph{the} father of $H$ arbitrarily among all hyperplanes $V$ with the desired properties.  In practice, this arbitrary choice is justified since we shall only use the three given properties of $f(H)$.  Having chosen the father of each $H\in\mathcal R(U)$, we see that $\prec$ partially orders $\mathcal R(U)$.
\end{remark}

\begin{lemma}\label{lem:contactingfathers}
If $H,H'\in {\mathcal R}(U)$  and $H\coll H'$, then one of the following holds:
\begin{enumerate}
\item $f(H)=f(H').$
\item $f(H)\coll f(H').$
\item $H\prec H'$ and $f(H')$ contacts a potential father $W$ of $H$ such that $W$ crosses $U$.
\item $H'\prec H$ and $f(H)$ contacts a potential father $W'$ of $H'$ such that $W'$ crosses $U$.
\end{enumerate}
\end{lemma}

\begin{proof}
\textbf{The disc diagram $D$:}  Let $V=f(H)$ and $V'=f(H')$ and suppose that $V\neq V'$ and that $V$ and $V'$ do not contact.  Following the proof of Proposition~\ref{prop:weakcombing}, construct a disc diagram $D\rightarrow\bf X$ as follows.  Let $P\rightarrow N(U)$ be a combinatorial path joining $a\in F(V,U)$ to $c\in F(V',U)$, where $a$ and $c$ are chosen to be the preimages in $N(U)$ of the roots $b_H$ and $b_{H'}$, respectively, of the iterated imprints $IJ(H,U)$ and $IJ(H',U)$.  Since $V$ does not contact $V'$, we have $a\neq c$ and hence $|P|\geq 1$.

Let $R,R'\rightarrow N(V),N(V')$ join $a$ (respectively, $c$) to a closest 0-cube of $N(H)$ (respectively, $N(H')$).  Let $Q,Q'\rightarrow N(H),N(H')$ be shortest geodesics such that $PR'Q'QR^{-1}$ is a closed path, and let $D$ be a minimal-area disc diagram for that path; see the left side of Figure~\ref{fig:lemma6}.  There is at least one dual curve $C$ emanating from $P$, and $C$ cannot end on $R$ or on $R'$ since that would lead to a trigon of dual curves along the boundary path of $D$, contradicting minimality of area.  Thus $C$ ends on $H$ or on $H'$, and hence maps to a hyperplane $W(C)$ that crosses $U$ and crosses either $H$ or $H'$.

\textbf{Interpretation in $U$:}  Let $P_0$ be the image of $P$ in $U$ under the projection $N(U)\rightarrow U$, so that $P_0$ is the shortest path joining $b_H$ to $b_{H'}$ in the tree $U$.  Let $g\in U$ be the gate of the root $b^*$ in $P_0$, i.e. $g\in P_0$ is the unique point such that for all $p\in P_0$, any geodesic from $p$ to $b^*$ passes through $g$.  Either $g$ is contained in the interior of $P_0$, or $g$ is equal to one of the endpoints, so without loss of generality, suppose that $g\neq b_H$.  This situation is shown in the center of Figure~\ref{fig:lemma6}.  We shall show that in fact every path in $U$ from $b^*$ to $b_H$ must pass through $b_{H'}$ (that is to say, that $g=b_{H'}$) and thus that $H'\prec H$.

\textbf{A potential father of $H'$ contacts $f(H)$:}  Let $a'_0$ be the 0-cube of $P_0$ adjacent to $b_H$, and let $a'\in P$ be the 0-cube mapping to $a'_0$.  Then there is a dual curve $C$ in $D$ emanating from the 1-cube $aa'$ and ending on $Q$ or on $Q'$.  Let $W$ be the hyperplane to which $C$ maps.  If $C$ ends on $Q$, then $W$ is a potential father of $H$.  But $d(a'_0,b^*)<d(b_H,b^*)$ since $a'_0$ is closer than $b_H$ to the gate $g$.  This implies that $W=f(H)$.  But all dual curves emanating from $P$ map to distinct hyperplanes, a contradiction.  Thus $W$ is not a potential father of $H$, and hence $W\in PF(H')$.  On the other hand, $W\coll f(H)$.  It therefore remains to show that $H'\prec H$.

\textbf{$b_{H'}$ is the gate:}  Every other dual curve $C'$ emanating from $P$ must end on $Q'$ and thus map to a potential father of $H'$.  Indeed, no such dual curve can cross $C$ by minimality of the area of $D$.  Hence every path in $U$ from $b^*$ to an interior vertex of $P_0$ must pass through $b_{H'}$, since $V'$ is the father of $H'$, and thus $g=b_{H'}$.  Hence each path from $b^*$ to $b_H$ passes through $b_{H'}$, and thus $H'\prec H$.
\begin{figure}[h]
  \includegraphics[width=4in]{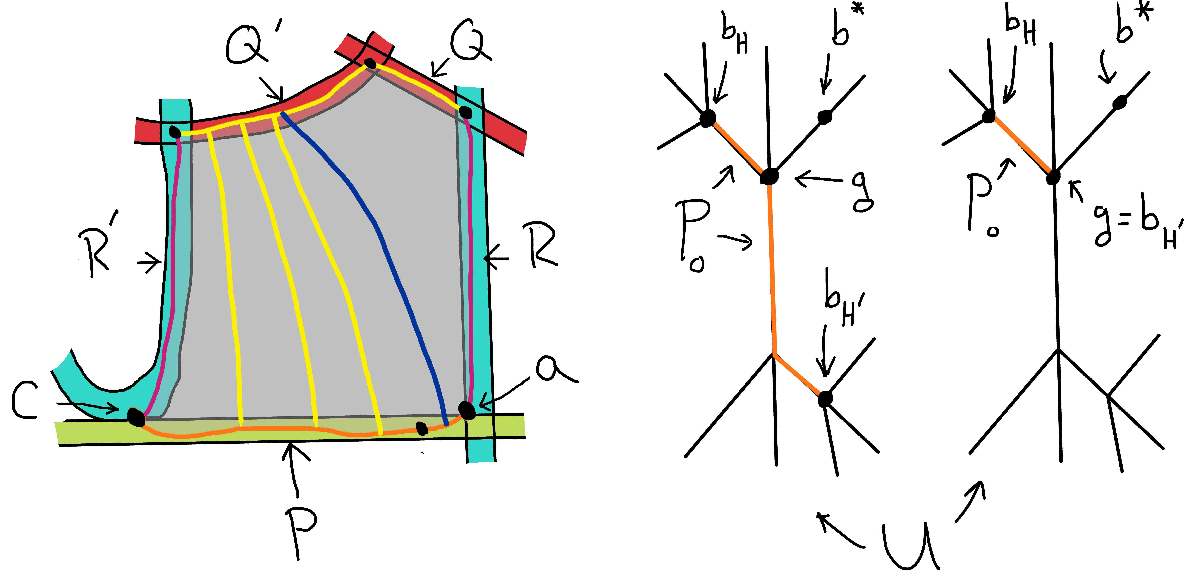}\\
  \caption{At left is the diagram $D$ in the proof of Lemma~\ref{lem:contactingfathers}.  In the center is an a priori picture of the projection of $P$ to $U$; at right is the actual picture.}\label{fig:lemma6}
\end{figure}
\end{proof}

\section{The graph $\Upsilon(U)$}
Now, we define the following subgraph $\Upsilon(U)$ of $\Gamma({\bf X})$: the vertices of $\Upsilon(U)$ are the hyperplanes of ${\mathcal R}(U)$ and two hyperplanes $H$ and $H'$ are adjacent in $\Upsilon(U)$ if and only if $H\coll H',$ the fathers $f(H)$ and $f(H')$ are different, and $f(H)$ and $f(H')$ do not contact. By Lemma \ref{lem:contactingfathers}, if $H$ and $H'$ are adjacent in $\Upsilon(U)$, then either $H\prec H'$ and the father $f(H')$ of $H'$ contacts a potential father of $H$, or $H'\prec H$ and the father $f(H)$ of $H$ contacts a potential father of $H'.$  Note that $\Upsilon(U)$ is a subgraph of the grade-2 cluster $\mathcal C$ centered at $U$. The graph $\Upsilon(U)$ can also be viewed as a subgraph of the intersection graph of iterated imprints of hyperplanes in $\mathcal R(U)$, by Lemma~\ref{lem:potentialpatherconnected}.

Our goal is to colour $\Upsilon(U)$.  Since to colour the whole graph  $\Upsilon(U)$ it is enough to colour each of its connected components, we will assume without loss of generality that $\Upsilon(U)$ is connected.  To colour $\Upsilon(U)$, we will group the edges of  $\Upsilon(U)$ into three spanning subgraphs $\Upsilon_0(U),\Upsilon_1(U),\Upsilon_2(U)$
of $\Upsilon(U)$ and colour each of these graphs separately.

\begin{definition}[Root class, incoming neighbour, outgoing neighbour]\label{defn:ingoingneighbour}
The vertices of $\Upsilon(U)$ can be partitioned into subsets according to their roots: for each vertex $b$ of $U$, let $\mathcal R'_b$ be the set of hyperplanes $H\in\mathcal R(U)$ such that $b_H=b$.  In other words, $\mathcal R'_b$ is the set of hyperplanes $H$ such that the iterated imprint $IJ(H,U)$ is rooted at $b$.  The set $\mathcal R'_b$ is the \emph{root class} associated to the root $b$.

Let $H\in\mathcal R'_b$ and let $V=f(H)$ be its father.  The hyperplane $H'\in\mathcal R(U)$ is an \emph{incoming neighbour} of $H$ if $HH'$ is an edge of $\Upsilon(U)$ (and in particular $H\coll H'$), and the iterated imprint $IJ(H',U)$ contains $b$ (and in particular $H'\prec H$).
More intuitively, $H'$ is an incoming neighbour of $H$ if $HH'$ is an edge of $\Upsilon(U)$ and $H'\prec H$.  Denote by $\mathcal I_b(H)$ the set of incoming neighbours of $H$.

By Lemma~\ref{lem:contactingfathers}, for each $H'\in\mathcal I_b(H)$, we have that $f(H')$ contacts a potential father of $H$ that crosses $U$.  If $H'$ is adjacent to $H$ in $\Upsilon(U)$ and $H'$ is not an incoming neighbour, then by Lemma~\ref{lem:contactingfathers}, $H\prec H'$, and we call $H'$ an \emph{outgoing neighbour} of $H$.
\end{definition}

The incoming neighbours of a fixed hyperplane are totally ordered by $\prec$; while we do not make explicit use of this fact in colouring $\Gamma({\bf X})$, it is a basic property of $\prec$.

\begin{proposition}[Incoming neighbours]\label{prop:neighbourproperties}
For any vertex $b$ of $U$ and any $H\in R'_b$, the set $\mathcal I_b(H)$ of incoming neighbours of $H$ is totally ordered by $\prec$ and there is a hyperplane $W$ such that $W$ contacts $f(H)$, $W$ is not a potential father of $H$, and $W$ is a potential father of $H_i$ for all $H_i\in\mathcal I_b(H)$.
\end{proposition}

\begin{proof}
The roots $b_{H'}$ of all incoming neighbours $H'\in \mathcal I_b(H)$ of $H$ are all different from the root $b_H$ of $H$ and all belong to the unique path of the tree $U$ between $b^*$ and $b_H$. Therefore the trace of the partial order $\prec$ on $\mathcal I_b(H)$ is a total order $\{H_1,H_2,\ldots,H_m\}$ of the incoming neighbours of $H$. 
For each $i$, let $V_i=f(H_i)$ and let $V=f(H)$; by definition, $H_i\prec H$.

Denote by $P_0$ the path in $U$ from $b^*$ to $b$, and let $P$ be the path in $IF(H,U)$ projecting to $P_0$.  Let $a$ be the vertex of $P$ mapping to $b$.  Let $R\rightarrow N(V)$ be a shortest geodesic joining $a$ to $N(V)\cap N(H)$.

For each $i$, let $P_i\rightarrow N(U)$ be a shortest path joining $a$ to $N(V_i)$ and let $\bar P_i$ be the image of $P_i$ in $U$.  Let $R_i\rightarrow N(V_i)$ be a shortest path joining the terminus $a_i$ of $P_i$ to $N(H_i)$, and let $Q_i,Q'_i\rightarrow N(H),N(H_i)$ be a pair of geodesics whose concatenation joins the terminus of $R$ to the terminus of $R_i$.  Let $D_i\rightarrow\bf X$ be a minimal area disc diagram for $P_iR_i(Q'_i)^{-1}Q_i^{-1}R^{-1}$, as in Figure~\ref{fig:incoming1}.
\begin{figure}[h]
  \includegraphics[width=2.5in]{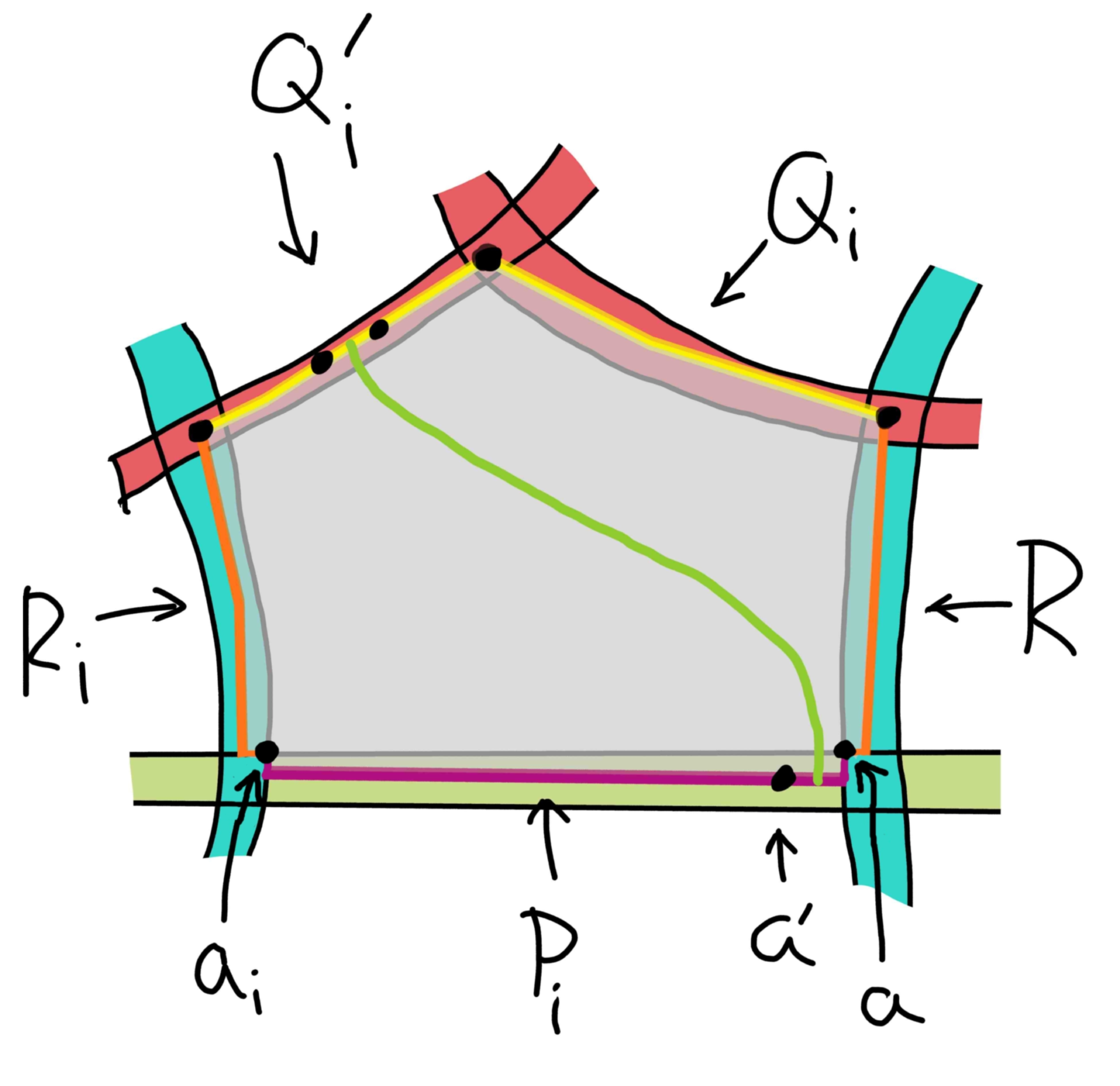}\\
  \caption{The diagram $D_i$.}\label{fig:incoming1}
\end{figure}

Note that for all $i$, we have $\bar P_i\subseteq P_0$, since $\bar a_i$ lies on the path from $b$ to $b^*$ since $H_i\prec H$.  Hence we have that $a_{i+1}$ lies on the path from $a_i$ to $b^*$, and thus $\bar P_1\subseteq \bar P_2\subseteq\ldots\subseteq \bar P_m\subseteq P_0$, i.e. $H_1\prec H_2\prec\ldots\prec H_m\prec H$.

In particular, the initial 1-cube $aa'$ of $P_1$ is contained in $P_i$ for each $i$ ($P_1$ contains at least one 1-cube since $V$ and $V_1$ do not contact, by $\Upsilon(U)$-adjacency of $H$ and $H_1$).  Hence, for all $i$, the dual curve $C_i$ in $D_i$ emanating from $aa'$ maps to the same hyperplane $W$.  Now $C_i$ cannot end on $P_i,R$ or $R_i$ by minimality of the area, and thus $C_i$ ends on $Q_i$ or $Q'_i$.  Moreover, $C_i$ cannot end on $Q_i$.  Indeed, if this were the case, then $W$ would be a potential father of $H$. But since $aa'$ projects to a 1-cube of $P_0$, this would contradict the fact that $V$ is the potential father whose imprint is closest to $b^*$.  Thus $C_i$ ends on $Q'_i$, and moreover $C_i$ ends on a 1-cube of $Q'_i$ that does not contain a 0-cube of $Q_i$.  Hence $W$ is a potential father of $H_i$ for each $i$, and $W$ crosses $U$ and $H_i$.
\end{proof}

\subsection{Diagrams lying over edges in $\Upsilon(U)$}\label{sec:liesover}
Let $H,H'\in\mathcal R(U)$ be hyperplanes such that $H'H$ is an edge of $\Upsilon(U)$ and $H'\prec H$, i.e. let $H'$ be an incoming neighbour of $H$.  Then, as in the proof of Lemma~\ref{lem:contactingfathers}, there is a disc diagram $D\rightarrow\bf X$ associated to the edge $H'H$ as follows.

Let $a',a$ be 0-cubes of $N(U)$ projecting to the roots $b',b$ of $V'=f(H')$ and $V=f(H)$, respectively.  Let $P\rightarrow N(U)$ be a geodesic segment joining $a'$ to $a$, and let $a'$ and $a$ be chosen among the preimage points of $b',b$ in such a way that $|P|$ is minimal.  Let $R'\rightarrow N(V')$ and $R\rightarrow N(V)$ be geodesic segments respectively joining $a'$ and $a$ to closest 0-cubes of $N(H')\cap N(V')$ and $N(H)\cap N(V)$.  Let $Q'\rightarrow N(H')$ and $Q\rightarrow N(H)$ be geodesic segments that have a single common 0-cube in $N(H)\cap N(H')$, so that the concatenation $Q'Q$ joins the terminal 0-cube of $R'$ to the terminal 0-cube of $R$.  Then there is a minimal-area disc diagram $D\rightarrow\bf X$ with boundary path $R'Q'QR^{-1}P^{-1}$; we say that $D$ \emph{lies over} the edge $H'H$, as shown in Figure~\ref{fig:liesover}.
\begin{figure}[h]
  \includegraphics[width=2.25in]{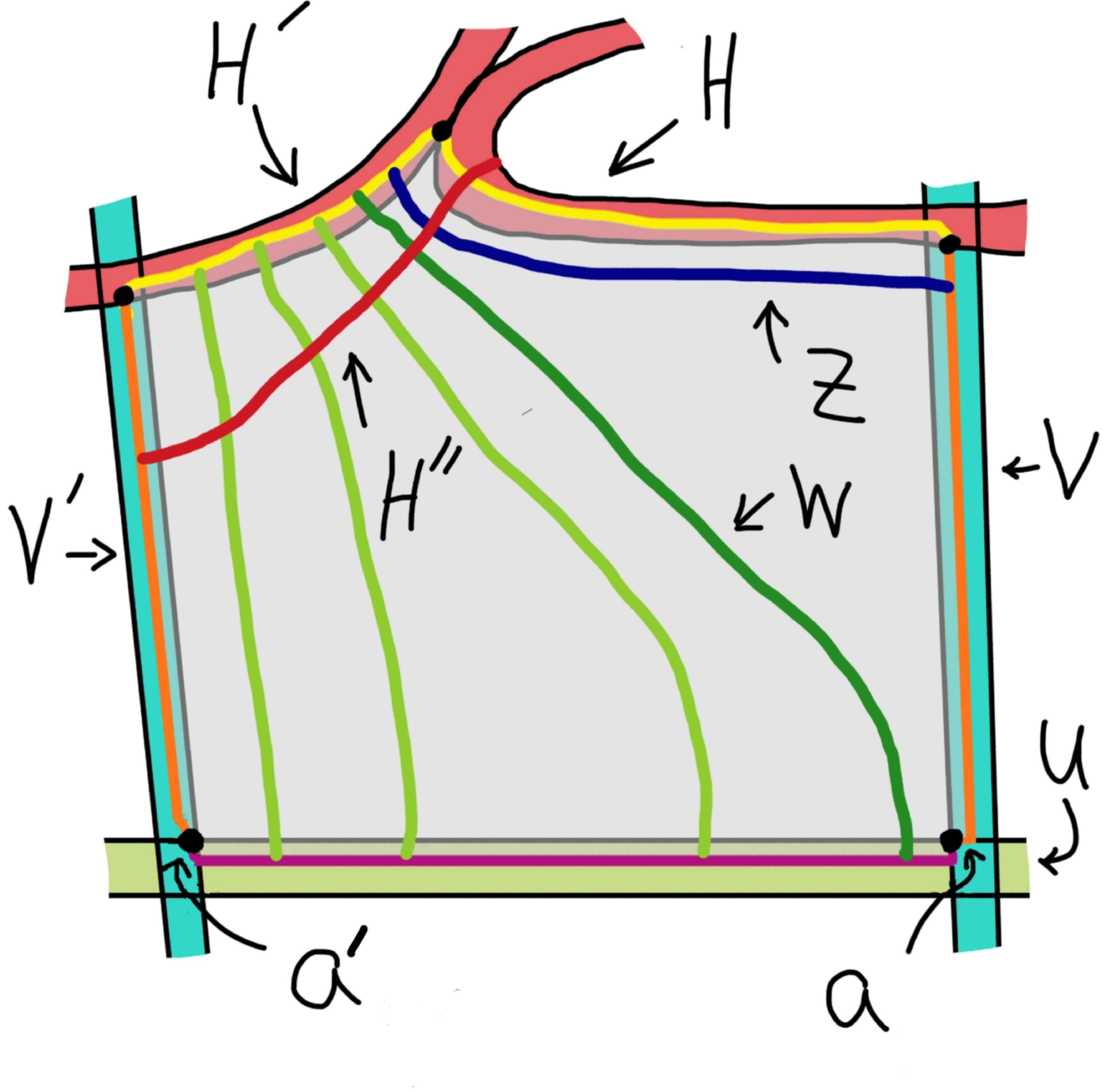}\\
  \caption{A diagram $D$ that lies over the edge $H'H$ of $\Upsilon(U)$, with $H'\prec H$.}\label{fig:liesover}
\end{figure}

Analysis of a diagram $D$ lying over $H'H$ reveals two hyperplanes, denoted $Z=Z(H'H)$ and $W=W(H'H)$ associated to the pair $H'H,$ and the diagram $D$.  Since $H'$ and $H$ have distinct, non-contacting fathers, we see that $|P|>0$, and hence there is a dual curve $L$ emanating from the terminal 1-cube of $P$ (i.e. the 1-cube containing $a$) and mapping to a hyperplane $W$ that crosses $U$ and $H'$ and contacts $V$, as in Lemma~\ref{lem:contactingfathers}.

Now $L$ cannot end on the terminal 1-cube of $Q'$ (i.e. the 1-cube containing the 0-cube $Q'\cap Q$), since $W$ cannot contact $H$.  Hence there is a dual curve $K$ emanating from the terminal 1-cube of $Q'$ and ending on the terminal 1-cube of $R$.  Indeed, $K$ cannot end on $Q'$, on $R'$, or on $Q$ for the usual reasons of minimal area, and $K$ cannot end on $P$, for otherwise the hyperplane $Z$ to which $K$ maps would be a better choice of father for $H$ than $V$.  Thus $K$ ends on $R$, and hence $K$, $Q$ and the subtended part of $R$ bound a triangular subdiagram, which must have area 0.  In particular, $K$ must end on the terminal 1-cube of $R$.  Thus $Z$ crosses $V$ and $H'$ and contacts $H$.  This situation is depicted in Figure~\ref{fig:liesover}.  Note that $Z$ does not, in general, contact $W$ since $K$ and $L$ may be separated by many dual curves traveling from $Q'$ to $R$.

\subsection{Separating  osculators and the graph $\Upsilon_1(U)$}\label{sec:separating}
Denote by $\bf A$ and $\bf B$ the halfspaces associated to $U$.
Recall that since we are colouring $\Upsilon(U)$, and to do so requires only that we colour each component, we have assumed that $\Upsilon(U)$ is connected.
Therefore, all hyperplanes $H$ of $\Upsilon(U)$ belong to one and the same halfspace defined by $U,$ say to the halfspace ${\bf A}$. Let ${\bf A}(H)$ and ${\bf B}(H)$ be the complementary halfspaces  associated to any hyperplane $H$ belonging to $\bf A$, in particular to any hyperplane of $\Upsilon(U).$ Since $U$ and $H$ are not crossing, $U$ belongs to one of these halfspaces, say in the halfspace ${\bf B}(H).$ Then ${\bf A}(H)\subset {\bf A}$ and $H\subset {\bf A}$ for any hyperplane $H\in\mathcal R(U)$.

Let $H\in\mathcal R(U)$.  Then $d(H)\geq 1$, and thus  there exists a hyperplane $W$ such that $U\subset {\bf B}(W)$ and $H\subset {\bf A}(W)$, which is to say that $W$ separates $H$ from $U$.  Since any two convex subspaces of $\bf X$ are separated by finitely many hyperplanes, there exists a hyperplane $S(H)$ such that $S(H)$ osculates with $H$ and separates $H$ from $U$.  Indeed, there must exist $S(H)$ separating $H$ from $U$ such that $S(H)$ is not separated from $H$ by any hyperplane, and thus $S(H)\coll H$.  This contact cannot be a crossing,
for otherwise $S(H)$ would not separate $H$ from $U$, as each of the intersections ${\bf A}(H)\cap{\bf B}(S(H))$, etc., would be nonempty. Accordingly, we define each hyperplane $S(H)$ that osculates with $H$ and separates $H$ from $U$ to be a \emph{separating osculator} of $H$.

\begin{lemma}\label{lem:atmosttwoseparatingosculators}
Let $H\in\mathcal R(U)$.  Then one of the following holds:
\begin{enumerate}
\item If $d(H)\geq 2$, then $H$ has a unique separating osculator $S(H)$ and $f(H)$ crosses $S(H)$.  Moreover, $S(H)\in\mathcal R(U)$.
\item If $d(H)=1$, then $H$ has at most two separating osculators, $S_1(H)$ and $S_2(H)$, and $S_1(H)$ and $S_2(H)$ either cross or coincide.  Moreover, $S_1(H)$ and $S_2(H)$ are potential fathers of $H$, and $f(H)$ either crosses $S_i(H)$ or coincides with $S_i(H)$.
\end{enumerate}
\end{lemma}

\begin{proof}
We first show that if $S_1$ and $S_2$ separate $H$ from $U$ and osculate with $H$, then either $S_1=S_2$ or $S_1$ and $S_2$ cross.  Indeed, suppose that $H\subset{\bf A}(S_1)\cap{\bf A}(S_2)$ and $U\subset{\bf B}(S_1)\cap{\bf B}(S_2)$.  Either ${\bf A}(S_1)={\bf A}(S_2)$, or (say) ${\bf A}(S_1)\subset{\bf A}(S_2)$ or ${\bf A}(S_1)\cap{\bf B}(S_2)\neq\emptyset$.  In the first case, $S_1=S_2$.  In the second case, $S_1$ separates $S_2$ from $H$, a contradiction.  In the third case, $S_1$ and $S_2$ must cross, since each of the quarter-spaces determined by $S_1$ and $S_2$ is nonempty.

Thus the set of separating osculators of $H$ is a set of pairwise-crossing hyperplanes, and hence, since $\dimension{\bf X}=2$, there are at most two separating osculators, $S_1(H)$ and $S_2(H)$.

If $d(H)\geq 2$, then neither $S_1(H)$ nor $S_2(H)$ contacts $U$, and hence $f(H)$ is not a separating osculator of $H$.  Since $N(f(H))\cap N(H)$ and $N(f(H))\cap N(U)$ are both nonempty, $f(H)$ contains points of ${\bf A}(S_i)$ and ${\bf B}(S_i)$ for $i\in\{1,2\}$, i.e. $f(H)$ crosses $S_1(H)$ and $S_2(H)$.  Since $\bf X$ contains no pairwise-crossing triple of hyperplanes, we conclude that $S_1(H)=S_2(H)$, and denote by $S(H)$ the unique separating osculator of $H$.  Moreover, since $H$ is separated from $U$ by at least two hyperplanes, $S(H)$ is separated from $U$ by at least one hyperplane, and hence $d(S(H))\geq 1$.  On the other hand, since $f(H)$ crosses $S(H)$, we have $S(H)\in\mathcal R(U)$.

If $d(H)=1$, then $S_1(H)$ and $S_2(H)$ osculate with $U$ and are thus potential fathers of $H$.  If $F$ is a potential father of $H$ that is distinct from $S_i(H)$, then $F$ crosses both $S_1(H)$ and $S_2(H)$.  Hence, if $S_1(H)\neq
S_2(H)$, then $H$ has exactly two potential fathers, namely $S_1(H)$ and
$S_2(H)$. 
\end{proof}

In summary, if $d(H)>1$, then the unique separating osculator $S(H)$ crosses $f(H)$.  If $d(H)=1$, then since $S_1(H),S_2(H)$ either cross or coincide, we define $S(H)$ to be whichever of $S_1(H)$ or $S_2(H)$ has closer root in $U$ to $b^*$.  In this case, either $S(H)=f(H)$ or $S(H)$ crosses $f(H)$.  If $f(H)=S(H)$, then $H$ is \emph{father-separated} as shown in Figure~\ref{fig:SepOsc}.
\begin{figure}[h]
  \includegraphics[width=5in]{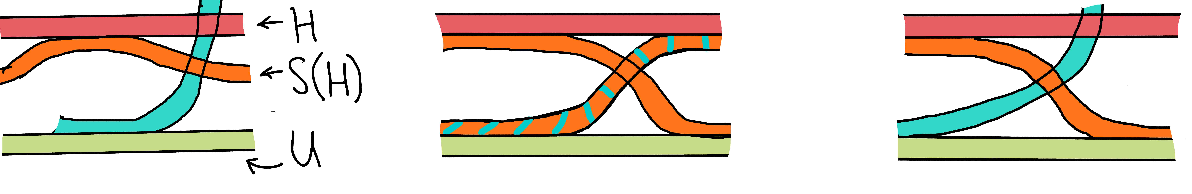}\\
  \caption{At left, $d(H)>1$ and $H$ has a unique separating osculator $S(H)$ that crosses $f(H)$.  In the center, $d(H)=1$ and $H$ has two separating osculators, one of which is the father of $H$; this is the father-separated case.  At right, $d(H)=1$, the father of $H$ crosses $H$, and $H$ has a unique separating osculator.}\label{fig:SepOsc}
\end{figure}

Define the graph $\Upsilon_1(U)$ as follows: $\Upsilon_1(U)$ has ${\mathcal R}(U)$ as the set of vertices and two hyperplanes $H,H'$ are adjacent in $\Upsilon_1(U)$ if and only if $H$ and $H'$ are adjacent in $\Upsilon(U)$ and one of the following conditions holds: either $S(H)=S(H')$ or $S(H)=H'$ or $S(H')=H$.

\begin{proposition}\label{prop:colseposc}
$\chi(\Upsilon_1(U))\le \Delta.$
\end{proposition}

\begin{proof} We colour the hyperplanes of $\Upsilon_1(U)$ in the increasing hyperplane-distance $d(H)$ starting with the separating osculators contacting $U$ (these hyperplanes do not belong to our graph $\Upsilon_1(U),$ but in order to make the colouring process uniform,  we can suppose that they all received the same colour). Namely, suppose that  $W$ is the current hyperplane, $W$ has been coloured and we want to colour all hyperplanes $H$ such that $S(H)=W.$  Notice that the footprints in $N(W)$ of all such hyperplanes belong to one and the same bounding factor $W'\cong W\times\{\pm\frac{1}{2}\}$ of $N(W)\cong W\times[-\frac{1}{2},\frac{1}{2}]$, isomorphic to $W$ and bounding the carrier $N(W)$ of $W.$ Two hyperplanes $H,H'$ with $S(H)=W=S(H')$ define an edge of $\Upsilon_1(U)$ if and only if $H$ and $H'$ contact. Since $H$ and $H'$ both osculate with $W$, by the Helly property we conclude that the footprints of  $H$ and $H'$ in $W'$ intersect. Therefore, in order to colour the hyperplanes having $W$ as their separating osculator, it is enough to colour the intersection graph of the footprints of such hyperplanes in the tree $W'$ so that all such hyperplanes receive a colour different from the colour of $W$. Any vertex $v'$ of $W'$ can belong to at most $\Delta-1$ footprints of such hyperplanes (because $v'$ has a neighbour $v''$ in the second hyperplane $W''$ bounding $N(W))$. Thus the intersection graph of the footprints on $W'$ has clique number $\Delta-1.$ Since this graph is chordal and therefore perfect, we can colour it with $\Delta-1$ colours. Taking into account the colour of $W$ and extending this colouring process,  we obtain a colouring of $\chi(\Upsilon_1(U))$ with at most $\Delta$ colours.
\end{proof}

\subsection{Father osculators and the graph $\Upsilon_2(U)$}\label{sec:nonseparating}
Let $H'H$ be an edge of $\Upsilon(U)$ such that $H'\prec H,$ and $H'H$ is not an edge of $\Upsilon_1(U)$, and $S(H')$ does not contact $H$.  In other words, $H'\coll H$, the fathers of $H$ and $H'$ are distinct and do not contact, and the separating osculators of $H$ and $H'$ are distinct, and distinct from $H$ and $H'$.

Suppose also that $H$ (respectively, $H'$) does not separate $U$ from $H'$ (respectively, $H$), i.e. suppose that $U,H'\subset {\bf B}(H)$ and $U,H\subset {\bf B}(H')$ (this corresponds to the conflict relation in event structures). Then we say that $H'$ is a {\it father osculator} if the triplet of hyperplanes $H',H,$ and $f(H)$ pairwise contact, and $H'$ osculates with $f(H)$, and $S(H)\ne f(H).$  Let $\Upsilon_2(U)$ be the spanning subgraph of $\Upsilon(U)$ consisting of all edges $H'H$ of $\Upsilon(U)$ such that $H'\prec H$ and $H'$ is a father osculator of $H.$


\begin{proposition}\label{prop:colnonseposc}
$\chi(\Upsilon_2(U))\le\Delta.$
\end{proposition}

\begin{proof} Let $H$ be a hyperplane of ${\mathcal R}(U)$ having a father osculator. Let $S(H)$ be the separating osculator of $H.$ Since $S(H)\ne f(H),$ the hyperplanes $S(H)$ and $f(H)$ cross, by Lemma~\ref{lem:atmosttwoseparatingosculators}. Moreover, since $H$ osculates with $S(H)$ and contacts $f(H),$ we conclude that the intersection $c_0$ of the carriers of the hyperplanes $H,S(H),$ and $f(H)$ is a single 0-cube  or a single 1-cube.   Indeed, since $\bf X$ is 2-dimensional, $N(f(H))\cap N(S(H))$ is a single 2-cube $c$, each of whose 1-cubes is dual to $f(H)$ or $S(H)$, and thus $c\cap N(H)=c_0$ is a 0-cube if $H$ and $f(H)$ osculate and a 1-cube dual to $f(H)$  if $H$ and $f(H)$ cross. In the first case, set $c_0:=\{ v_H\}$ and in the second case, set $c_0:=\{ v_H,v'_H\},$ where $v_H$ belongs to the halfspace (denote it by ${\bf A}(f(H))$) of $f(H)$ containing the root $b^*$ of $U$ ($v'_H$ belongs to the complementary halfspace ${\bf B}(f(H))$).

Let $H'$ be a father osculator of $H$. We claim that $v_H\in N(H').$ First we show that $c_0\cap N(H')\ne\emptyset.$ Since $H'$ contacts $H$ and $f(H),$ by the Helly property it suffices to show that $H'$ and $S(H)$ contact.  Since $H\subset {\bf A}(S(H))$ and $H'\coll H,$ it suffices to show that $H'\cap {\bf B}(S(H))\ne \emptyset.$ Suppose not: then $H'\subset {\bf A}(S(H))$ and thus $S(H)$ separates $H'$ from $U.$  Since $H'H$ is not an edge of $\Upsilon_1(U),$ we conclude that $S(H')\ne S(H),$ whence $d(S(H))<d(S(H')).$ Therefore $S(H')$ cannot separate $H$ from $U.$ Since $H'\subset {\bf A}(S(H'))$ and $H'$ contacts $H$, necessarily $S(H)$ crosses $H$, contrary to the definition of edges of $\Upsilon_2(U)$. This shows that the hyperplanes $H',H,f(H),$ and $S(H)$ pairwise contact, and therefore their carriers share a vertex of $c_0$. Suppose that this vertex is $v'_H$ and not $v_H$. Since $H'$ osculates with $f(H)$, $H'$ is contained the halfspace ${\bf B}(f(H))$ of $f(H)$.  But in this case, the root $b_{H'}$ of $H'$ is also contained in ${\bf B}(f(H))$. Since $b_H$ is contained in  ${\bf A}(f(H))$, $b_H$ lies on the unique path of $U$ between $b_{H'}$ and $b^*,$ and we obtain a contradiction with the assumption that $H'H$ is an edge $\Upsilon(U)$ and $H'\prec H.$ This contradictions shows that $v_H\in N(H').$

Now, since $v_H$ belongs to the carrier of any father osculator $H'$ of $H,$ $H$ can have at most $\Delta-1$ father osculators. Thus, in $\Upsilon_2(U)$ the incoming degree of any hyperplane $H$ is at most $\Delta-1$ and therefore $\Upsilon_2(U)$ can be coloured in $\Delta$ colours by the greedy algorithm following the orientation of edges defined by $\prec$.
\end{proof}


\subsection{The graph $\Upsilon_0(U)$}

Let $\Upsilon_0(U)$ be the graph obtained by removing from $\Upsilon(U)$ all edges of the graphs $\Upsilon_1(U)$ and $\Upsilon_2(U),$ i.e., $H'H$ is an edge of $\Upsilon_0(U)$ if and only if $H'H$ is an edge of $\Upsilon(U)$ (i.e., $H'\prec H,$ $H\coll H',$ and $f(H)$ and $f(H')$ are distinct and do not contact), $S(H)$ is different from $H'$ and $S(H')$,  and $H'$ is not a father osculator of $H.$  We will show that $\Upsilon_0(U)$ is bipartite and therefore can be coloured in two colours. Together with Propositions~\ref{prop:colseposc} and~\ref{prop:colnonseposc}, this will show that the graph $\Upsilon(U)$ can be coloured in $2\Delta^2$ colours.
We start with the following classification of edges of $\Upsilon(U).$

\begin{lemma}\label{lem:classification}
Let $H'H$ be an edge of $\Upsilon(U).$ Then one of the following holds:
\begin{enumerate}
\item  $S(H)$ coincides with $S(H')$ or $H'$ (and $H'H$ is an edge of $\Upsilon_1(U)$).
\item $H'$ is a father osculator of $H$ (and $H'H$ is an edge of $\Upsilon_2(U)$).
\item $S(H')$ crosses $H$.
\end{enumerate}
\end{lemma}

\begin{proof}
Suppose that $H'H$ is not an edge of $\Upsilon_1(U)$.  Then, by definition, we have that the hyperplanes $S(H),S(H'),H,$ and $H'$ are all distinct.  We shall argue, using a disc diagram lying over the edge $H'H$, that if $S(H')$ fails to cross $H$, then $f(H),H$ and $H'$ pairwise-osculate.  Moreover, in the case that $S(H)=f(H)$, the same diagram shows that $S(H')$ crosses $H$.

Let $V'=f(H'),\,V=f(H),\,S=S(H),$ and $S'=S(H')$.  As in Section~\ref{sec:liesover}, let $P\rightarrow N(U)$ join the preimages of the roots of $V'$ and $V$, let $R',R\rightarrow N(V'),N(V)$ be shortest geodesic segments joining the initial and terminal 0-cube of $P$, respectively, to $N(V')\cap N(H')$ and $N(V)\cap N(H)$, and let $Q'Q\rightarrow N(H')\cup N(H)$ be a shortest piecewise-geodesic segment joining the terminal 0-cubes of $R'$ and $R$, so that the path $R'Q'QR^{-1}P^{-1}$ bounds a minimal-area disc diagram $D\rightarrow\bf X$ lying over the edge $H'H$, as in Figure~\ref{fig:classification1}.
\begin{figure}[h]
  \includegraphics[width=2.5in]{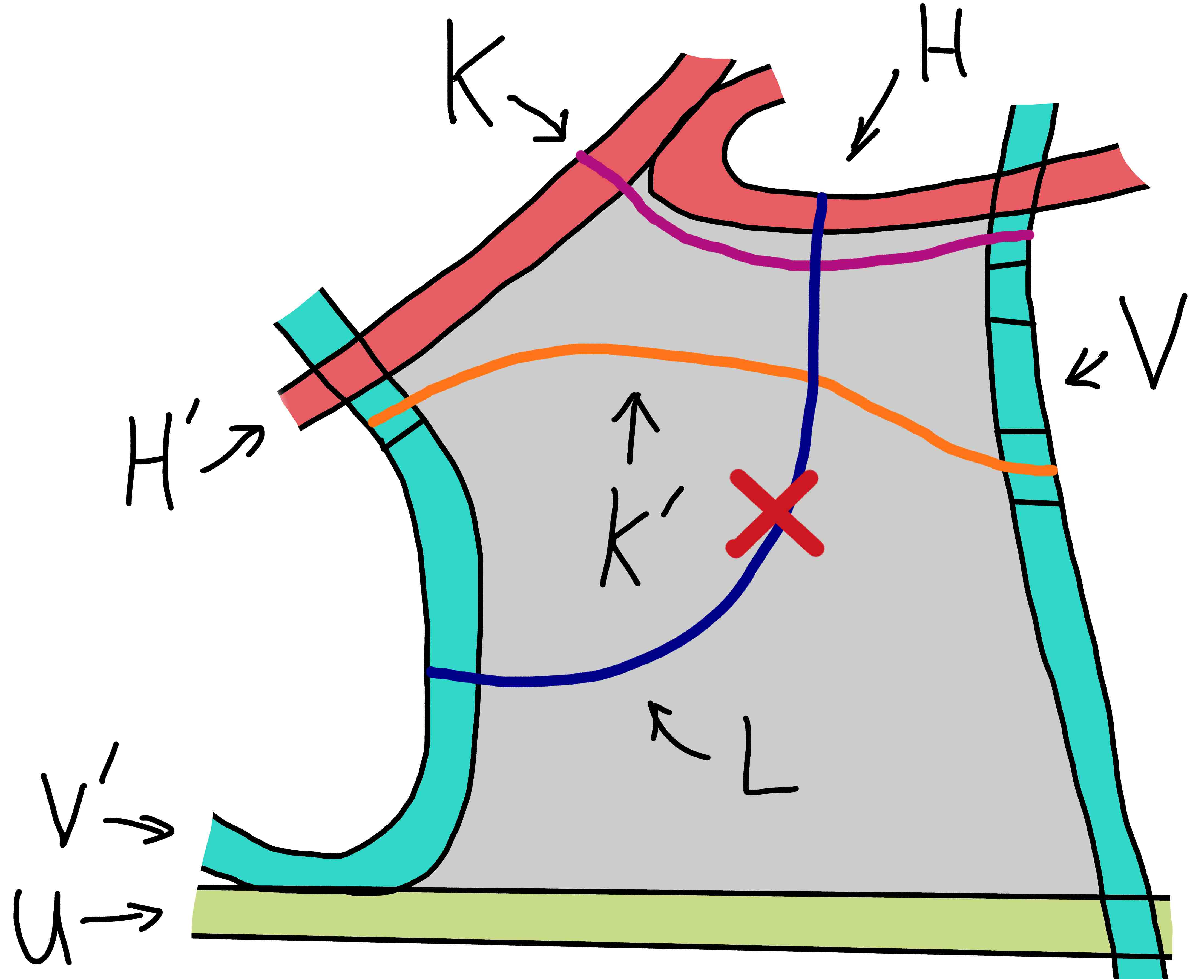}\\
  \caption{The diagram $D$ lying over $H'H$ in the proof of Lemma~\ref{lem:classification}.}\label{fig:classification1}
\end{figure}

Now consider the dual curve $K'$ emanating from the terminal 1-cube of $R'$.  Note that $K'$ either ends on the initial 1-cube of $Q$ or on some 1-cube of $R$, by minimality of the area of $D$.  Moreover, $K'$ maps to $S'$.  Indeed, since $S'$ separates $H'$ from $U$, the geodesic segment $R'$ must contain a 1-cube dual to $S'$.  Let $R'=R''T$, where $R''$ is the subpath joining the initial 0-cube of $R'$ to the initial 0-cube of the 1-cube $c$ dual to $S'$, and $T$ is the subpath, beginning with $c$ and ending at the terminal 1-cube of $R'$.  Let $s\in N(S')\cap N(H')$ be a 0-cube, which must exist since $S'$ and $H'$ osculate.  Let $A\rightarrow N(H')$ join $s$ to the terminal 0-cube of $T$ and let $B\rightarrow N(S')$ join the initial 0-cube of $T$ to $s$.  Then $BAT^{-1}$ is a closed path bounding a minimal-area disc diagram $E$.  Any dual curve in $E$ emanating from $T-c$ crosses $A$ or $B$ and thus leads to a trigon of pairwise-crossing dual curves; we conclude that $T=c$ and that $S'$ crosses $R'$ in its terminal 1-cube.  Hence, since each 1-cube of $\bf X$ is dual to a unique hyperplane, $K'$ maps to $S'$.

If $K'$ ends on $Q$, then $S'$ crosses $H$, and we are done.  Hence $K'$ ends on $R$.  If $K'$ ends on the terminal 1-cube of $R$, then the above argument shows that $K'$ maps to $S$ and hence $S=S'$, and the proof is again complete.

The unique remaining possibility is that $K'$ ends on $R$ at some interior 1-cube, and the dual curve $K$ emanating from the terminal 1-cube of $R$ and mapping to $S$ separates $K'$ from $Q$, as shown in Figure~\ref{fig:classification1}.  Note that $S'$ and $S$ both cross $V$, since $K'$ and $K$ end on 1-cubes of $R$.  Hence $S'$ cannot cross $S$, since otherwise $S',S,$ and $V$ would be a pairwise-crossing triple of hyperplanes, contradicting 2-dimensionality of $\bf X$.  Hence $S'$ separates $S$ from $U$, and therefore $S$ is not a potential father of $H$, and in particular $S\neq V$.

Now suppose that $|Q|>0$, so that there exists a dual curve $L$ emanating from $Q$.  $L$ cannot end on $Q'$ or on $R$, since that would lead to a trigon removal along the boundary path of $D$ and a consequent area reduction.  On the other hand, if $L$ ends on $P$, then there would be a better choice of father for $H$, namely the hyperplane to which $L$ maps, and hence $L$ ends on $Q'$.  But since $K'$ travels from $R'$ to $R$ and emanates from the terminal 1-cube of $R$, the dual curves $L$ and $K'$ must cross, and map to distinct hyperplanes since $R'$ is a geodesic segment.  Hence $S',W,V'$ pairwise-cross, where $W$ is the hyperplane to which $L$ maps, and this contradicts 2-dimensionality of $\bf X$.

Hence $|Q|=0$ and, in particular, $V$ contacts $H'$.  On the other hand, $V$ cannot cross $H'$.  Indeed, since $K$ ends on $Q'$, we see that $S$ crosses $H'$ and that $V$ crosses $S$, and the absence of pairwise-crossing triples ensures that $V$ and $H'$ cannot cross.  Thus $V$ and $H'$ osculate.

In summary, $S(H)\neq f(H)$ and $f(H),H,$ and $H'$ pairwise contact, and $f(H)$ osculates with $H'$.  Hence $H'$ is a father osculator, as shown in Figure~\ref{fig:nonseposc}.

\begin{figure}[h]
  \includegraphics[width=2in]{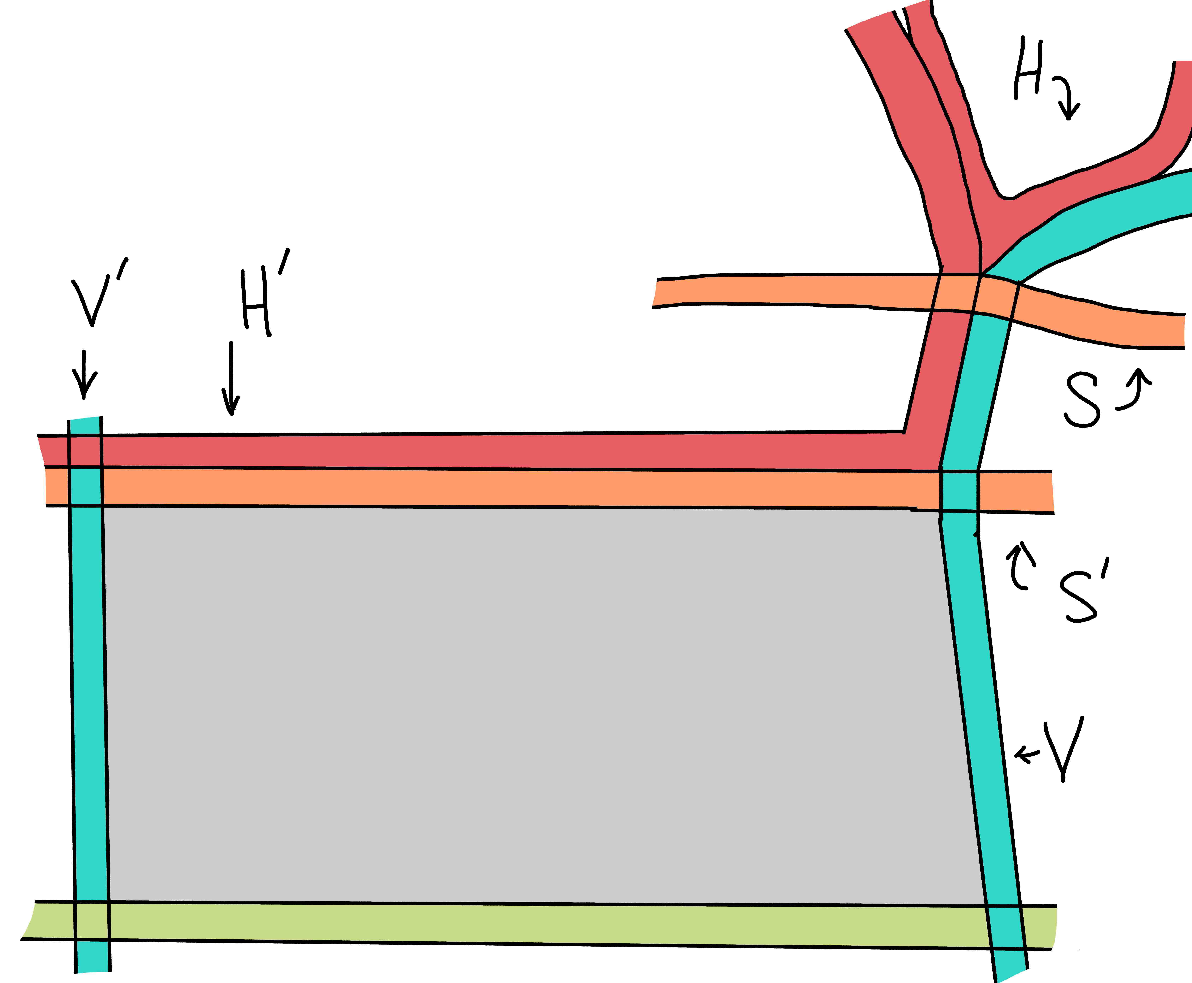}\\
  \caption{When the separating osculators of $H'$ and $H$ are distinct and distinct from $H$ and $H'$, either $H'$ is a father osculator of $H$ (as shown), or $S(H')$ crosses $H$.}\label{fig:nonseposc}
\end{figure}
\end{proof}

The first step in proving that $\Upsilon_0(U)$ is bipartite is to show that it contains no triangles.

\begin{lemma}\label{lem:trianglefree}
The graph $\Upsilon_0(U)$ is triangle-free.
\end{lemma}

\begin{proof}
Let $C=(H_0,H_1,H_2)$ be a 3-cycle in $\Upsilon_0(U)$.  Then without loss of generality, we have $H_1\prec H_0\prec H_2$.  Indeed, one of the three vertices, say $H_2$, does not precede either of the other two, and hence, by Lemma~\ref{lem:contactingfathers}, we have $H_0\prec H_2$ and $H_1\prec H_2$.  But by Lemma~\ref{lem:contactingfathers}, since $H_0$ and $H_1$ are adjacent in $\Upsilon(U)$, they are comparable in the partial ordering $\prec$. Suppose, moreover, that the distance-sum $D(C)=d(H_0)+d(H_1)+d(H_2)$ is minimal among all 3-cycles in $\Upsilon_0(U)$. Denote by $V_i$ the father of $H_i,i=0,1,2.$  By Lemma~\ref{lem:classification}, $S(H_1)$ crosses $H_0$ and $H_2.$

First suppose that $d(H_1)=1$.  If $S(H_1)=V_1$, then $b_{S(H_1)}\prec b_{V_0}\prec b_{V_2}$.  But since $S(H_1)$ crosses $H_0$ and $H_2$, by Lemma~\ref{lem:classification}, we see that $S(H_1)$ is a potential father of $H_0$ and $H_2$.  Hence $S(H_1)=V_0=V_2$ and we reach a contradiction with the fact that $H_1H_0$ and $H_1H_2$ are edges of $\Upsilon(U)$.

The remaining possibility is that in which $S(H_1)$ crosses (and is different from) $V_1$ and $V_1$ crosses at least one of $U$ and $H$, by Lemma~\ref{lem:atmosttwoseparatingosculators}, as shown at right in Figure~\ref{fig:SepOsc}.  If $S(H_1)=V_0$, then we reach a contradiction as above.  Otherwise, the imprint of $V_0$ on $U$ lies between the imprint of $S(H_1)$ and the imprint of $V_1$, i.e. $b_{V_0}$ lies on the unique path in $U$ between $b_{V_1}$ and $b_{S(H_1)}$.  This is because $S(H_1)$ crosses $H_0$, and $S(H_1)$ is not the father of $H_0$, and $H_1\prec H_0$.

However, since $V_1,S(H_1),U$ pairwise-contact, the imprint of $V_1$ on $U$ has nonempty intersection with the imprint of $S(H_1)$ on $U$, as illustrated in Figure~\ref{fig:comment23}.
\begin{figure}[h]
  \includegraphics[width=5in]{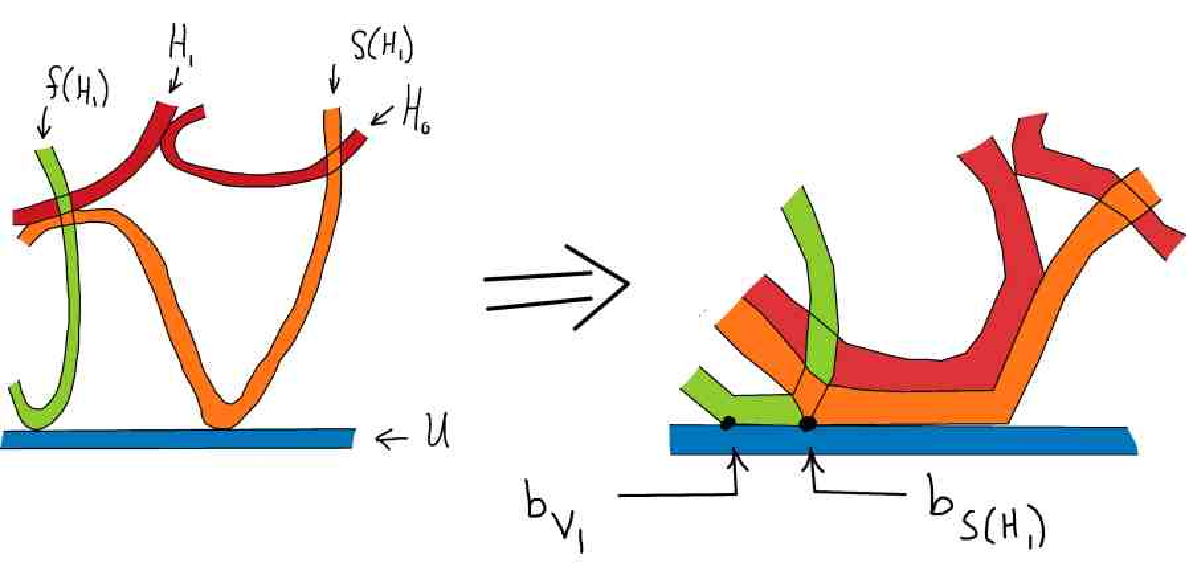}\\
  \caption{A heuristic picture showing that $b_{S(H_1)}$ lies in the imprint of $V_1$.}\label{fig:comment23}
\end{figure}
Since $b_{S(H_1)}$ is the closest point of the imprint of $S(H_1)$ to $b_{V_1}$, we see that $b_{S(H_1)}$ lies in the imprint of $V_1$.  But then $b_{V_0}$ must lie in the imprint of $V_1$, whence $V_0\coll V_1$.  This contradicts the fact that $H_1H_0$ is an edge of $\Upsilon(U)$.  Hence $d(H_1)\geq 2$.


Since $d(H_1)\ge 2$, the hypeprlane $S(H_1)$ is in ${\mathcal R}(U)$ by Lemma~\ref{lem:atmosttwoseparatingosculators}. The father of $S(H_1)$ is either $V_1$ or a hyperplane before $V_1$. Therefore, $S(H_1)\prec H_0\prec H_2$ holds, and $S(H_1)H_0$ and $S(H_1)H_2$ are both edges of $\Upsilon(U).$ Since $S(H_1)$ crosses $H_0$ and $H_2,$ neither of these edges is an edge of $\Upsilon_2(U).$ Now, suppose that $S(H_1)H_0$ is an edge of $\Upsilon_1(U).$ Since $S(H_1)$ and $H_0$ cross, this is possible only if $S(S(H_1))=S(H_0).$  But in this case, the disc diagram lying over the edge $H_1H_0$ will contain a trigon.

Indeed, suppose that $S(S(H_1))=S(H_0)$ and let $D\rightarrow\bf X$ be a diagram lying over the edge $H_1H_0$, as shown in Figure~\ref{fig:inevitabletrigon}.  Then the dual curve $K$ emanating from the terminal 1-cube of $R_1$ and mapping to $S(H_1)$ ends on $Q_0$ at the initial 1-cube.  The subdiagram $D'\subset D$ bounded by $N(K)$, $P$, $R_0$, $R_1,$ and the subtended part of $Q_0$ lies over the edge $S(H_1)H_0$.  The dual curve $L$ emanating from the penultimate 1-cube of $R_1$ maps to $S(S(H_1))$, and by the assumption that $S(S(H_1))=S(H_0)$, we have that $L$ ends on the terminal 1-cube of $R_0$.  As usual, since $V_1$ does not contact $V_0$, there is a dual curve $M$ in $D$ traveling from $P$ to $Q_1$, and $M$ cannot end on the terminal 1-cube of $Q_1$, for otherwise $M$ would map to a hyperplane providing a better father for $H_0$.  Hence there must exist a dual curve $N$ emanating from the terminal 1-cube of $Q_1$ and ending on $R_0$.  But $N$ cannot end on the terminal 1-cube of $R_0$, since that 1-cube is already the origin of $L$, and hence $N$ must cross $L$.  But the hyperplanes to which $N$ and $L$ map both cross $V_0$, and thus cannot cross.  Hence $S(S(H_1))\neq S(H_0)$.
\begin{figure}[h]
  \includegraphics[width=2.5in]{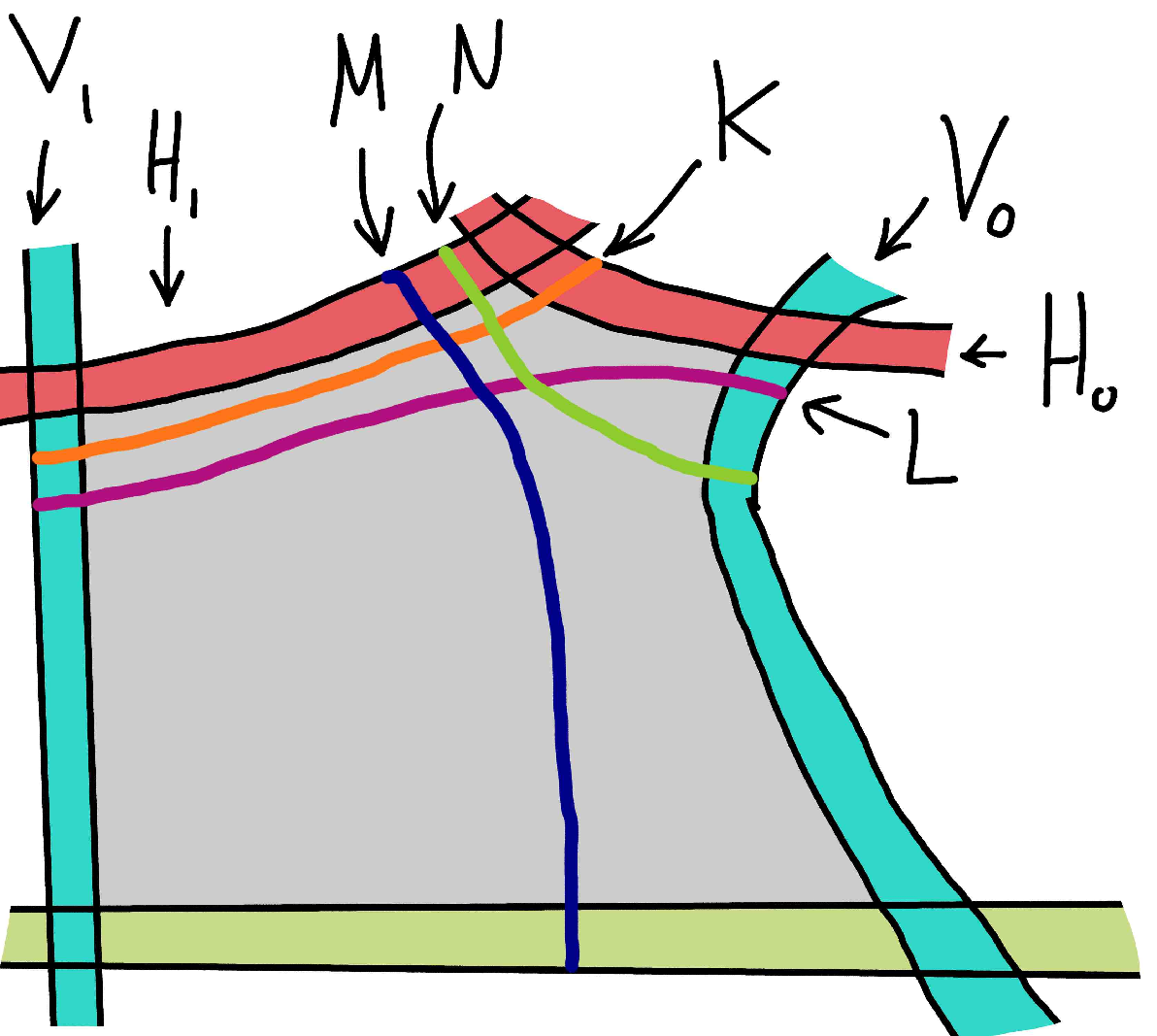}\\
  \caption{The diagram $D$ lying over $H_1H_0$ contains a contradictory trigon when $S(S(H_1))=S(H_0)$.}\label{fig:inevitabletrigon}
\end{figure}

Hence $S(H_1)H_0$ and $S(H_1)H_2$ are not edges  of $\Upsilon_1(U),$ showing that $C'=(H_0,S(H_1),H_2)$ is a 3-cycle of $\Upsilon_0(U).$
Since $D(C')<D(C),$ our choice of $C$ could not have been minimal, a contradiction.
\end{proof}

We now analyze cycles in $\Upsilon_0(U)$, with the goal of showing that there are no cycles of odd length.  Let $C=(H_0,H_1\ldots ,H_{n-1})$ be a simple $n$-cycle in $\Upsilon_0(U)$.  $C$ is \emph{induced} if for all $i\in\mathbb Z_n$ and $j\neq i\pm 1$, the hyperplanes $H_i$ and $H_j$ are not adjacent in $\Upsilon_0(U)$, i.e. either $H_i$ and $H_j$ do not contact, or their fathers contact or coincide, or they have a common separating osculator.  Notice that if $\Upsilon_0(U)$ contains an odd cycle, then each of its odd cycles of minimum length $n=2k+1$ is an induced cycle. Indeed, if $H_iH_j$ is an edge of $\Upsilon_0(U)$ that does not belong to $C$, then at least one of the paths of $C$ connecting  $H_i,H_j$ has even length, and hence $\Upsilon_0(U)$ contains an odd-length cycle that is shorter than $C$.

For a cycle $C=(H_0,H_1,\ldots ,H_{n-1})$  of  $\Upsilon_0(U)$,   let $V_i=f(H_i)$ and let $b_i\in U$ be the root of the imprint of $H_i$ on $U$, for each $i\in\mathbb Z_n$.
Let ${\bf A}_i$ and ${\bf B}_i$ be the two halfspaces of ${\bf X}$ defined by $H_i$ so that $U\subset {\bf B}_i.$
Cycles of $\Upsilon_0(U)$ have the following simple properties.

\begin{lemma}\label{lem:separate} If $(H_0,H_1,\ldots,H_n)$ is an induced path in $\Upsilon_0(U)$ and the hyperplane $H_j$ is contained in the halfspace ${\bf A}_i$ defined by the hyperplane $H_i,$ then either $V_i=V_j$ or $H_i\prec H_j$.
\end{lemma}

\begin{proof} Since $H_i$ separates $H_j$ from $U$, the father $V_j$ of $H_j$ crosses $H_i,$ and thus $V_j$ is a potential father of $H_i.$ From the definition of a father we conclude that either $V_i=V_j$ or $H_i\prec H_j$.
\end{proof}

The hyperplane $H_i$ of $C$ is \emph{normal} if exactly one of $H_{i-1}$ and $H_{i+1}$ is an incoming neighbour, and the other neighbour of $H_i$ in $C$ is outgoing.  By Lemma~\ref{lem:contactingfathers}, if $H_i$ is not normal, then $H_{i-1}$ and $H_{i+1}$  are both incoming or both outgoing neighbours of $H_i$.

\begin{lemma}\label{lem:normal}
Any induced cycle $C$ of $\Upsilon_0(U)$ of odd length $n=2k+1$ contains at least one normal hyperplane $H_i$.
\end{lemma}

\begin{proof} By Lemma \ref{lem:trianglefree}, $n>3.$ Suppose that no hyperplane in $C$ is normal.  Then the hyperplanes of $C$ can be partitioned into two sets $\mathcal I$ and $\mathcal O$, where $\mathcal I$ is the set of $H_i$ for which both neighbours in $C$ are incoming, and $\mathcal O$ is the set of $H_i$ for which both neighbours in $C$ are outgoing.  By definition, every edge of $C$ joins an element of $\mathcal I$ to an element of $\mathcal O$.  Hence $C$ is bipartite, and in particular has even length.
\end{proof}

\begin{lemma}\label{lem:badcasenormal} Any cycle  of $\Upsilon_0(U)$ does not contain normal hyperplanes.
\end{lemma}

\begin{proof} We proceed by way of contradiction. Suppose that $n$ is the smallest value for which there exists a cycle  of length $n$ containing a normal hyperplane. Now, among all minimal cycles of $\Upsilon_0(U)$ of length $n$ containing normal hyperplanes, let $C=(H_0,H_1,\ldots ,H_{n-1})$ have minimal distance sum $D(C).$   Let $H_1$ be a hyperplane of $C$ having the closest root $b_1$ to $b^*$.

Then the neighbours $H_0$ and $H_2$ of $H_1$ in $C$ are both outgoing neighbours of $H_1$ in $\Upsilon(U),$ and thus $H_1$ is not a normal hyperplane of $C$. We proceed as in the proof of Lemma \ref{lem:trianglefree}. Let $S(H_1)$ be the separating osculator of $H_1.$ By Lemma \ref{lem:classification}, $S(H_1)$ crosses $H_0$ and $H_2.$ If $d(H_1)=1,$ then, as in the proof of Lemma~\ref{lem:trianglefree}, we reach the contradictory conclusion that the fathers of $H_1$ and $H_0$ coincide or contact.


Now suppose that $d(H_1)\ge 2.$ Then $S(H_1)$ is a hyperplane of ${\mathcal R}(U)$ by Lemma~\ref{lem:atmosttwoseparatingosculators}. The father of $S(H_1)$ is either $V_1$ or a hyperplane before $V_1$, by Lemma~\ref{lem:separate}. Therefore, $S(H_1)\prec H_0\prec H_2$ holds, and $S(H_1)H_0$ and $S(H_1)H_2$ are both edges of $\Upsilon(U).$ Since $S(H_1)$ crosses $H_0$ and $H_2,$ neither of these edges is an edge of $\Upsilon_2(U).$ Now, suppose that $S(H_1)H_0$ is an edge of $\Upsilon_1(U).$ Since $S(H_1)$ and $H_0$ cross, this is possible only if $S(S(H_1))=S(H_0).$  But in this case, the disc diagram over the edge $H_1H_0$ will contain a trigon, as in the proof of Lemma~\ref{lem:trianglefree}. Hence $S(H_1)H_0$ and $S(H_1)H_2$ are not edges  of $\Upsilon_1(U),$ and therefore they are edges of $\Upsilon_0(U),$ showing that $C'=(H_0,S(H_1),H_2,\ldots,H_{n-1})$ is a cycle of the graph $\Upsilon_0(U).$ Since $S(H_1)\prec H_1,$ the choice of $H_1$ in $C$ implies that if $H_i$ is a normal hyperplane of $C,$ then $i\ne 1$ and $H_i$ is a normal hyperplane of $C'.$
Since $D(C')<D(C),$ we obtain a contradiction with the minimality choice of $C.$ This contradiction shows that no cycle of $\Upsilon_0(U)$ contains normal hyperplanes.
\end{proof}

\begin{proposition}\label{prop:upsilonbipartite}
The graph $\Upsilon_0(U)$ is bipartite; therefore $\chi(\Upsilon_0(U))=2$.
\end{proposition}

\begin{proof} From Lemma~\ref{lem:normal}  we know that any induced odd cycle of $\Upsilon_0(U)$ must contain a normal hyperplane. On the other hand, Lemma~\ref{lem:badcasenormal} asserts
that no cycle of $\Upsilon_0(U)$ can contain a normal hyperplane. Since any graph containing odd cycles also contain induced odd cycles, we conclude that $\Upsilon_0(U)$  cannot contain any odd cycle, i.e. $\Upsilon_0(U)$ is bipartite.
\end{proof}

Now, we are ready to prove the main result of this section:

\begin{proposition}\label{prop:colupsilon} $\chi(\Upsilon(U))\le 2\Delta^2.$
\end{proposition}

\begin{proof}
To show that $\chi(\Upsilon(U))\le 2\Delta^2,$ associate to each hyperplane $H$ of $\Upsilon(U)$ the three colours of $H$ in the colourings of the graphs $\Upsilon_0(U),\Upsilon_1(U),$ and $\Upsilon_2(U)$ provided by Propositions \ref{prop:colseposc},\ref{prop:colnonseposc} and \ref{prop:upsilonbipartite}. Since $\chi(\Upsilon_0(U))=2$ and $\chi(\Upsilon_1(U))\le \Delta, \chi(\Upsilon_2(U))\le \Delta,$ the hyperplanes  of $\Upsilon(U)$ will be coloured with at most $2\Delta^2$ colours. Since each edge $H'H$ of $\Upsilon(U)$ is contained in at least one of the graphs $\Upsilon_0(U),\,\,\Upsilon_1(U),$ and $\Upsilon_2(U),$ the triplets of $H'$ and $H$ differ in at least one coordinate; thus the resulting triplet-colouring is a correct colouring in at most $2\Delta^2$ colours of each connected component of  $\Upsilon(U),$ and therefore of the whole graph  $\Upsilon(U)$.
\end{proof}

%
%

\section{Proof of Theorem 1}

\subsection{Colouring the contact graph $\Gamma ({\bf X})$}
We now colour the contact graph $\Gamma({\bf X})$, proving the first assertion of Theorem~\ref{theorem1}. The proof
is divided into several steps.



\medskip\textbf{Strategy:}  To show that $\chi(\Gamma({\bf X}))$ is bounded by a function of the maximum degree $\Delta$ of $G({\bf X}),$ we take an arbitrary but fixed base hyperplane $H_0$ and partition the contact  graph  $\Gamma({\bf X})$ into the spheres $S_k, k=0,1,\ldots,$  centered at $H_0.$ Now, if we show that the subgraph of $\Gamma({\bf X})$ induced by each sphere $S_k$ can be coloured in $\alpha(\Delta)$ colours, then combining a colouring of the spheres with even radius using the same set of $\alpha(\Delta)$ colours and a colouring of the spheres with odd radius using another set of $\alpha(\Delta)$ colours, we will obtain a correct colouring of $\Gamma({\bf X})$ into $2\alpha(\Delta)=\epsilon(\Delta)$ colours.

In order to colour $S_k$ in $\alpha(\Delta)$ colours, it suffices to colour each cluster $\mathcal C$ of $S_k$ in the contact graph $\Gamma({\bf X})$ with this number of colours, since distinct clusters in $S_k$ do not contain adjacent vertices. It follows from Corollary \ref{cor:diameter}  that each such cluster $\mathcal C$ has diameter at most 5. Therefore, if we pick an arbitrary hyperplane $V_0:=V_0^{\mathcal C}$ in $\mathcal C,$ then all hyperplanes $V$ of $\mathcal C$ have distance at most 5 to $V_0,$ i.e., $\mathcal C\subset B_5(V_0),$ where $B_r(V_0)=\{ V: \rho(V_0,V)\le r\}$ is the ball of $\Gamma({\bf X})$ of radius $r$ centered at $V_0.$ Therefore, if we show that $B_5(V_0)$ can be coloured with $\alpha(\Delta)$ colours, then taking the restriction of this colouring to $\mathcal C,$ we will obtain a colouring of $\mathcal C$ into at most $\alpha(\Delta)$ colours. Repeating this colouring procedure for each cluster $\mathcal C$ of $S_k$ with the same set of $\alpha(\Delta)$ colours, we will obtain the required colouring of $S_k$.

Let $q(r)$ be the number of colours necessary to colour the ball $B_r(V_0)$ of radius $r$ centered at $V_0.$  The main part of our proof is to establish the following recurrence $q(r)\le q(r-1)\cdot q(r-2)\cdot(2\Delta)\cdot(2\Delta^2)+q(r-1),$ yielding a bound $\alpha(\Delta)\leq q(5).$ Suppose that the ball $B_{r-1}(V_0)$ has been coloured in $q(r-1)$ colours and let $c$ be a colouring of $B_{r-1}(V_0)$ with this number of colours obtained in the recursive way. We will show how to extend $c$ to a colouring of $B_r(V_0)$ using the required number of colours by showing how to colour the hyperplanes from $S_r(V_0)$ using at most $q(r-1)\cdot q(r-2)\cdot(2\Delta)\cdot 2\Delta^2$ extra colours.

\medskip\textbf{Choosing fathers and grandfathers:}  Suppose that the hyperplanes of $\mathbf X$ are graded according to their distance in $\Gamma(\mathbf X)$ from $H_0$.  For each grade-$r$ hyperplane $H$, with $r\geq 2$, fix once and for all a canonical path $\gamma(H)$ in $\Gamma(\mathbf X)$ joining $H_0$ to $H$.  This determines a grandfather $f^2(H)=\gamma(H)(r-2)$ for $H$, and a set of potential fathers $V$ of $H$: as before, these are the hyperplanes $V$ with $f^2(H)\coll V\coll H$.  (The case $r<2$ is dealt with separately below.)  Now, fix a root vertex in each hyperplane.  The choice of root in $f^2(H)$ determines a father $f(H)$ of $H$, as above.

\medskip\textbf{Colouring:}  For each hyperplane $V,$ consider a colouring $c'$ in at most $2\Delta$ colours (we use the same set of at most $2\Delta$ colours for each hyperplane) of  the families of footprints or imprints of the set of all hyperplanes $H$ having $V$ as their father (this colouring is provided by Proposition~\ref{lem:footprintscolours} showing that $\chi({\mathcal F}(V))\le \chi({\mathcal J}(V))\le 2\Delta$). Additionally, for each hyperplane $U$ define a colouring $c''$ of the graph $\Upsilon(U)$ in at most $2\Delta^2$ colours (we use the same set of  $2\Delta^2$ colours to colour the graph $\Upsilon(U)$ for each hyperplane $U$). This colouring is provided by the Proposition~\ref{prop:colupsilon}, which shows that $\chi(\Upsilon(U))\le 2\Delta^2$. Recall that $\Upsilon(U)$ has as its vertex set the set ${\mathcal R}(U)$ of hyperplanes $H$ with $f^2(H)=U$. For $H\in {\mathcal R}(U)$ let $c''(H)$ be the colour of $H$ in the colouring of $\Upsilon(U)$ with at most $2\Delta^2$ colours. Notice that it suffices to define the colourings $c'$ and $c''$ only on hyperplanes of grades $r-1$ and $r-2,$ respectively.

Now, for a hyperplane $H$ of grade $r,$ we assign as a colour the ordered quadruplet
\[c(H)=(c(f(H)),c(f^2(H)),c'(F(H,f(H))),c''(H))\]
Clearly, the hyperplanes of $S_r(V_0)$ will be coloured with at most $q(r-1)\cdot q(r-2)\cdot(2\Delta)\cdot(2\Delta^2)$ colours. Notice also that two contacting grade $r$ and grade $r-1$ hyperplanes will be coloured differently because we use new colours for colouring $S_r(V_0).$ It remains to show that $c$ is a correct colouring of the hyperplanes of $S_r(V_0).$  This is the content of Lemma~\ref{lem:colouringcontact}.

We conclude that $q(r)\le 4\cdot\Delta^3\cdot q(r-1)\cdot q(r-2)+q(r-1).$ Notice that $q(0)=1$ since $S_0=B_0(V_0)=\{ V\}$. On the other hand, $q(1)\le 2\Delta$ because colouring the hyperplanes of the sphere $S_1$ is equivalent to colouring the intersection graph of their imprints in $V_0$ and this can be done with at most $2\Delta$ colours by Proposition \ref{lem:footprintscolours}. Easy calculations show that $q(2)\le 8\Delta^4+2\Delta$ and that $q(3)\leq 64\Delta^8+16\Delta^5+8\Delta^4+2\Delta.$  Hence, assuming $\Delta\geq 2$, we have:
\begin{eqnarray*}
q(5)&\leq&q(4)\left(4\Delta^3q(3)+1\right)\\
&\leq&\left(4\Delta^3q(3)+1\right)\left[4\Delta^3(8\Delta^4+2\Delta)+1\right]q(3)\\
&\leq&q(3)\left(Aq(3)+B\right),
\end{eqnarray*}

\noindent where $A=2^7\Delta^{10}+2^5\Delta^7+4\Delta^3$ and $B=2^5\Delta^{7}+8\Delta^4+1$.  Combining this estimate with our assumption $\Delta\geq 2$ implies that
\begin{eqnarray*}
q(5) \leq(582608+\frac{318064260}{2^{26}})\Delta^{26},
\end{eqnarray*}
whence we obtain $$\alpha(\Delta)=q(5)\leq582613\Delta^{26}.$$

Hence each sphere $S_k$ admits a colouring using at most $\alpha(\Delta)$ colours, and thus
\[\chi(\Gamma({\bf X}))\leq 2\alpha(\Delta)=\epsilon(\Delta)\]
for $\Delta\ge 2.$  When $\Delta\leq 1$, $\mathbf X$ is already a tree and has at most one hyperplane. This concludes the proof of the first assertion of Theorem 1, with $M\leq1165226$.

\medskip\textbf{Correctness:} The following lemma establishes that $c$ is a correct colouring of the contact graph $\Gamma(\mathbf X)$.
\begin{lemma} \label{lem:colouringcontact}
If $H,H'\in S_r(V_0)$ and $H\coll H',$ then $c(H)\ne c(H').$
\end{lemma}

\begin{proof} By the weak combing property established in Corollary \ref{cor:weakcombing},  the grandfathers of $H$ and $H'$ either contact or coincide. If $f^2(H)\coll f^2(H'),$ then by induction $c(f^2(H))\ne c(f^2(H')),$ whence $c(H)\ne c(H')$ because the quadruplets $c(H)$ and $c(H')$ differ in the second coordinate. So, further we will assume $H$ and $H'$ have the same grandfather, say $U=f^2(H)=f^2(H')$ (i.e., $H,H'\in {\mathcal R}(U)$.) Analogously, if $H$ and $H'$ have different but contacting fathers $f(H)$ and $f(H'),$ then $c(f(H))\ne c(f(H')),$ and thus $c(H)\ne c(H')$ because the quadruplets $c(H)$ and $c(H')$ differ in the first coordinate. On the other hand, if $H$ and $H'$ have the same father $V,$ then Lemma~\ref{lem:contactingfootprints} implies $F(H,V)\cap F(H',V)\neq\emptyset$, so that $c'(F(H,V))\ne c'(F(H',V)),$ whence $c(H)\ne c(H')$ because $c(H)$ and $c(H')$ differ in the third coordinate.

Finally, suppose that $H\coll H',$ $U=f^2(H)=f^2(H')$ and  the fathers  of $H$ and $H'$ are different and do not contact. 
By definition of the adjacency in the graph $\Upsilon(U),$ the hyperplanes $H$ and $H'$ are adjacent in $\Upsilon(U),$ and therefore the colours of $H$ and $H'$ are different in the colouring $c''$ of $\Upsilon(U).$ Hence $c(H)\ne c(H')$ because $c(H)$ and $c(H')$ differ in the fourth coordinate.
\end{proof}

\subsection{Embeddings in products of trees: colouring the crossing graph $\Gamma_{\#}({\bf X})$}\label{sec:productoftrees}
We now deduce from the existence of a finite colouring of $\Gamma({\bf X})$ that $\bf X$ isometrically embeds in the product of at most $M\Delta^{26}$ trees, using Proposition~\ref{prop:isometricembedding}.  Adopting the median graph point of view, one sees that Proposition~2 of~\cite{BandeltChepoiEppstein} also
suffices to embed $\mathbf X$ in the product of finitely many trees.
Since $\xing X$ is a subgraph of $\Gamma({\bf X})$, we have a colouring $c$ of the vertices of $\xing X$ by a set $\mathcal K$ of at most $M\Delta^{26}$ colours.  The result now follows from Corollary~\ref{cor:embeddingintrees}.



\subsection{The nice labeling problem: colouring the pointed contact graph $\Gamma_{\alpha}({\bf X})$.}
Let ${\bf X}_{\alpha}$ be a 2-dimensional CAT(0) cube complex, pointed at $\alpha$, and suppose that 0-cubes in $\bf X$ have maximal degree $\Delta$ and maximal out-degree $\Delta_0$.  Let $\Gamma_{\alpha}({\bf X})$ be the pointed contact graph. In view of first assertion of the theorem, it suffices to show that $\Delta\le \Delta_0+2$ for any 2-dimensional CAT(0) cube complex $\bf X$. (In fact, $\Delta\le \Delta_0+n$ holds for any $n$-dimensional CAT(0) cube complex and the proof is a consequence of the fact that intervals in median graphs are distributive lattices \cite{BandeltHedlikova}.) Let $\alpha$ be the basepoint and suppose by way of contradiction that a vertex $v$ of $G_{\alpha}({\bf  X})$ contains three incoming neighbours $v_1,v_2,v_3$. From the definition of the basepoint order on  $G({\bf  X})$ it follows that $v_1,v_2,v_3$ are closer to $\alpha$ than the vertex $v,$ i.e., $v_1,v_2,v_3\in I({\alpha},v).$ Denote by $u_{i,j}$ the median of the triplet ${\alpha},v_i,v_j$. Since $v_i$ and $v_j$ are at distance 2, $u_{i,j}$ is adjacent to $v_i$ and $v_j.$  The vertices $u_{1,2},u_{1,3},$ and $u_{2,3}$ are pairwise distinct, otherwise $G({\bf  X})$ would contain a $K_{2,3}$, which is impossible in a median graph. Now, let $u$ be the median of the triplet $u_{1,2},u_{1,3},u_{2,3}.$ The vertex $u$ is different from $v$ and is adjacent to each vertex of this triplet. As a result, the vertices $v,v_1,v_2,v_3,u_{1,2},u_{1,3},u_{2,3},u$ define a 3-dimensional cube of $G({\bf X})$ contrary to the 2-dimensionality of $\bf X$.



\section{Proof of Theorem 2}
In this section, we construct an example establishing Theorem~\ref{theorem2} by applying to the construction in~\cite{ChepoiNiceLabeling} the ``recubulation'' construction in~\cite{HagenQuasiArb}.  Given a CAT(0) cube complex $\bf X$, the contact graph $\Gamma({\bf X})$ can be realized as the crossing graph of a larger cube complex $\recube X$ that contains $\bf X$ as an isometrically embedded subcomplex; $\recube X$ is the \emph{recubulation of $\bf X$}, whose construction is given in~\cite{HagenQuasiArb}.  In the next proposition, we review the construction of $\recube X$ and establish several useful properties.

\begin{proposition}\label{prop:recube}
Let $\bf X$ be a CAT(0) cube complex and let $\Delta$ be the maximum degree of a 0-cube in $\bf X$ (i.e. the cardinality of a largest clique in the contact graph $\Gamma({\bf X})$).  Then there exists a CAT(0) cube complex $\recube X$ and a combinatorial isometric embedding ${\bf X}\rightarrow\recube X$ such that:
\begin{enumerate}
\item The hyperplanes of $\recube X$ are in a bijection with those of $\bf X$.
\item $\Gamma({\bf X})=\xing{\recube X}$.
\item $\dimension{\recube X}=\Delta$, and each 0-cube of $\recube X$ has degree at most $\Delta^2+\Delta$.

\end{enumerate}
More generally, if $\xing X\subseteq\Gamma_{\alpha}\subseteq\Gamma({\bf X})$, then there exists a CAT(0) cube complex $\mathfrak R_{\alpha}$ of dimension at most $\omega(\Gamma_{\alpha})$ and maximal degree at most $\Delta^2+\Delta$ such that $\mathfrak R_{\alpha}$ contains an isometric copy of $\bf X$ and the crossing graph of $\mathfrak R_{\alpha}$ is equal to $\Gamma_{\alpha}$.
\end{proposition}

\begin{proof}
The idea of recubulation is to force each pair of osculating hyperplanes (corresponding to an osculation-edge in $\Gamma_{\alpha}$) in $\bf X$ to cross by the addition of a set of 2-cubes into which those hyperplanes extend; further cubes of higher dimension are added where necessary to satisfy the link condition of~\cite{GromovHyperbolic} and make $\recube X$ a CAT(0) cube complex.  The following argument shows that, for any subgraph $\Gamma_{\alpha}\subseteq\Gamma(\mathbf X)$ that contains $\xing X$, there is a CAT(0) cube complex $\mathfrak R_{\alpha}$ with $\mathbf X\subseteq\mathfrak R_{\alpha}\subseteq\recube X$, such that $\mathfrak R_{\alpha}$ contains $\mathbf X$ as an isometrically embedded subcomplex and has dimension at most $\omega(\Gamma_{\alpha})$ and degree at most $\Delta^2+\Delta$.

\noindent{\textbf{The intermediate complex $\mathbf X'$:}}  Let $H$ and $H'$ be osculating hyperplanes in $\mathbf X$.  By definition, there exists a 0-cube $v$ and distinct 1-cubes $e,e'$ incident to $v$ such that $e$ is dual to $H$ and $e'$ to $H'$.  Attach a 2-cube $s$ to $\mathbf X$ by identifying two consecutive 1-cubes of $s$ with the $e$ and $e'$ respectively, so that $s\cap\mathbf X$ consists of the path $ee'$.  Perform this procedure for each pair $(e,e')$ of 1-cubes of $\mathbf X$ corresponding to a pair of osculating hyperplanes.  (If one is recubulating with respect to a subgraph $\Gamma_{\alpha}\subseteq\Gamma(\mathbf X)$, then one only performs this construction on pairs $(e,e')$ that realize an osculation-edge of $\Gamma_{\alpha}$.) Denote the resulting cube complex by $\mathbf X'$.

\noindent{\textbf{Constructing $\recube X$ and $\mathfrak R_{\alpha}$:}}  By construction, each hyperplane $H\in \mathcal H$ of $\mathbf X$ extends to a subspace $W(H)$ of $\mathbf X'$ that separates $(\mathbf X')^0$ into exactly two components.  Indeed, $W(H)$ consists of $H$, together with the midcube of $s$ dual to the 1-cube $e$ of $s\cap\mathbf X$ dual to $H$, for each new 2-cube $s$ that was attached at the site of an osculation of $H$ with some other hyperplane of $\mathbf X$.  In the language of~\cite{HaglundPaulin}, $W(H)$ is a \emph{wall} in $(\mathbf X)^0$ whose halfspaces $\mathfrak h(W(H))$ and $\mathfrak h^*(W(H))$ respectively contain the 0-skeleta of $\mathbf A(H)$ and $\mathbf B(H)$.  More precisely, if $H$ extends into a new 2-cube $s$ of $\mathbf X'$, then $H$ is dual in $\mathbf X$ to a unique 1-cube $e$ of $s$ and $W(H)$ intersects $s$ in a midcube $c$ of $s$.  $\mathfrak h(W(H))$ consists of $\mathbf A(H)$ together with the halfspace of each such $s$ induced by $c$ that contains the 0-cube of $e$ lying in $\mathbf A(H)$.

By construction, the assignment $H\mapsto W(H)$ is bijective, and the walls $W(H)$ and $W(H')$ cross if and only if $H\coll H'$ (or, more generally, if and only if $H$ and $H'$ are adjacent in $\Gamma_{\alpha}$).  Let $\recube X$ be the cube complex dual to the wallspace $\left((\mathbf X')^0,\{W(H):H\in{\mathcal H}\}\right)$.  By the definition of the cube complex dual to a wallspace (see e.g.~\cite{ChatterjiNiblo}), the set of hyperplanes of $\recube X$ corresponds bijectively to the set of walls, and therefore to the set $\mathcal H$ of hyperplanes of $\mathbf X$. This establishes assertion (1).

Let $v_0\in\mathbf X^0\subseteq(\mathbf X')^0$ be a 0-cube.  Recall that a 0-cube $x\in\recube X$ is a choice $x(W(H))$ of halfspace associated to each wall $W(H)$ such that $v_0\in x(W(H))$ for all but finitely many hyperplanes $H$ of $\mathbf X$ and $x(W(H))\cap x(W(H'))\neq\emptyset$ for all $H,H'$.  In particular, if $x\in\mathbf X$ is a 0-cube, then one can make a choice $\phi(x)(W(H))$ of halfspace of $(\mathbf X')^0$, for each wall $W(H)$, by declaring $\phi(x)(W(H))$ to be the halfspace of $\mathbf X'$ containing $v_0$ if and only if $x(H)$ contains $v_0$.  This yields an injective map $x\mapsto\phi(x)$ from $\mathbf X^0$ to $\recube X^0$.  Moreover, it is easily checked that $\phi(x)$ and $\phi(x')$ differ on the wall $W(H)$ if and only if $x$ and $x'$ are separated in $\mathbf X$ by the hyperplane $H$.  It follows that $\mathbf X$ embeds in $\recube X$ in such a way that each hyperplane $H$ of $\mathbf X$, viewed as a subcomplex of $\recube X$, is equal to the intersection of $\mathbf X$ with a hyperplane $\widehat H$ of $\recube X$.

\noindent{\textbf{Isometric embedding:}}  By the construction of $\recube X$, any hyperplane $H$ of $\mathbf X$ is the trace on $\mathbf X$ of a hyperplane of $\recube X$ and, vice-versa, each hyperplane of $\recube X$ is the extension to $\recube X$ of a hyperplane of $\mathbf X$. Therefore any two 0-cubes $x,y$ of $\mathbf X$  are separated in $\recube X$ and $\mathbf X$ by the same number of hyperplanes. Hence, the graph $G({\mathbf X})$ is isometrically embedded in the 1-skeleton $G({\recube X})$ of $\recube X$.

\noindent{\textbf{Comparing $\Gamma(\mathbf X)$ and $\xing{\recube X}$:}}  If $H\coll H'$, then the walls $W(H)$ and $W(H')$ cross, hence $\Gamma(\mathbf X)\subseteq\xing{\recube X}$.
Conversely, if the walls $W(H)$ and $W(H')$ cross, then either $H$ and $H'$ already cross in $\mathbf X$, or $W(H)$ and $W(H')$ cross in a 2-cube $s$ of $\mathbf X'$ with the property that $s\cap\mathbf X$ is a path $ee'$ with $e$ dual to $H$ and $e'$ to $H'$.  Hence $H\coll H'$, whence $\xing{\recube X}\subseteq\Gamma(\mathbf X)$, establishing thus assertion (2).

\noindent{\textbf{Bounds on the dimension:}}  Let $\Gamma_{\alpha}\subseteq\Gamma(\mathbf X)$ be a subgraph containing $\xing X$.  Let $\mathfrak R_{\alpha}$ be the recubulation of $\mathbf X$ corresponding to $\Gamma_{\alpha}$ and let $\recube X$ be that corresponding to $\Gamma(\mathbf X)$, i.e. hyperplanes in $\mathfrak R_{\alpha}$ cross if and only if the corresponding hyperplanes of $\mathbf X$ are adjacent in $\Gamma_{\alpha}$, and hyperplanes in $\recube X$ cross if and only if they contact in $\mathbf X$.  By construction, $\mathbf X\subseteq\mathfrak R_{\alpha}\subseteq\recube X$, so that it suffices to bound the dimension and degree of $\recube X$.  Since a maximal family of pairwise-crossing hyperplanes in $\recube X$ corresponds bijectively to a maximal family of pairwise-contacting hyperplanes in $\mathbf X$, it is clear that $\dimension(\recube X)\leq\Delta$, proving the first inequality in assertion (3).  (In fact, an identical argument shows that $\mathfrak R_{\alpha}$ has dimension bounded by the clique number $\omega(\Gamma_{\alpha})$ of $\Gamma_{\alpha}$.)

\noindent{\textbf{The intersection with $\mathbf X$ of a maximal cube of $\recube X$:}}  For each maximal cube $C$ of $\recube X$, we shall show that $C^*=C\cap\mathbf X$ is nonempty and 
$C^*$ contains a 1-cube dual to each hyperplane of $\recube X$ that crosses $C$.  Indeed, let $0<d\leq\Delta$ be the dimension of $C$ and let $\widehat H_1,\ldots,\widehat H_d$ be the hyperplanes of $\recube X$ that cross $C$.  For $1\leq i\leq d$, let $H_i=\widehat H_i\cap\mathbf X$ be the corresponding hyperplane of $\mathbf X$.  Now, $K=\bigcap_{i=1}^dN(H_i)\neq\emptyset$.  Indeed, since the hyperplanes $\widehat H_i$ pairwise-cross,  the hyperplanes $H_i$ pairwise-contact in $\mathbf X$, whence $K$ is a nonempty convex subcomplex of $\mathbf X$ by the Helly property. Suppose that there exists a hyperplane $\widehat H$ of $\recube X$ that separates $K$ from $C$.  Any hyperplane $\widehat H_i$ that crosses both $K$ and $C$ must cross $\widehat H$, and thus $\{\widehat H_i\}_{i=1}^d\cup\{\widehat H\}$ is a family of pairwise-crossing hyperplanes in $\recube X$, contradicting the fact that $C$ is a maximal cube.  Hence no hyperplane of $\recube X$
can separate $K$ from $C$.

\medskip\noindent
{\bf Claim 1:} {\it $K\subseteq C$.}

\medskip\noindent
{\bf Proof of Claim~1:}  Suppose by way of contradiction that there exists a 0-cube $k\in K-C.$ Then some hyperplane of $\recube X$ separates $k$ from $C$.  
Among the hyperplanes of $\recube X$ separating $k$ from $C,$ let $\widehat H$ be a closest one to $k$. Then $k\in N(\widehat H)$. Let $kk'$ be the 1-cube of $\recube X$ dual to $\widehat H$.
We assert that $\widehat H$ crosses any hyperplane $\widehat H_i, i=1,\ldots d,$ which crosses the cube $C$. Let $u_iv_i$ be an 1-cube of $C$ dual to $\widehat H_i$.
Since $k\in K\subseteq N(H_i),$ there exists a 1-cube $kk_i$ of $\mathbf X$ dual to $H_i$ and therefore to $\widehat H_i$. Suppose that $k$ and $u_i$ belong to the same halfspace,
say ${\mathbf A}({\widehat H_i})$, of $\recube X$ defined by $\widehat H_i$, while $k_i$ and $v_i$ belong to the complementary halfspace ${\mathbf B}({\widehat H_i})$.

The vertex-set $C^0$ of the cube $C$ is a convex subset of the median graph $G({\recube X})$ and, since the convex sets of median graphs are gated,  $C^0$ is a gated
subset of $G({\recube X}).$ Let $x$ be the gate of $k$ in $C^0$. From the choice of $\widehat H$, we conclude that $k'$ belongs in $G({\recube X})$ to the interval
$I(k,x)$.  Since, $x\in I(k,u)$ for any $u\in C,$ in particular $x\in I(k,v_i),$ necessarily $k'\in I(k,v_i)$. Analogously,
since $k\in {\mathbf A}({\widehat H_i})$ is adjacent to $k_i$ and $k_i,v_i\in {\mathbf B}({\widehat H_i}),$ $k_i$ lies on a shortest path between $k$
and $v_i$, thus $k_i\in I(k,v_i)$. Let $v$ be the median
in $G({\recube X})$ of the triplet $v_i,k_i,k'$ (recall that $\{ v\}=I(v_i,k_i)\cap I(k_i,k')\cap I(k',v_i)$). Then $v\ne k$ and $vk',vk_i$ are 1-cubes of $\recube X$.
The 0-cubes $v,k',k,k_i$ define a 4-cycle of  $G({\recube X})$
and therefore a 2-cube of $\recube X$. This implies that the 1-cube $vk'$ is dual to the hyperplane  $\widehat H$ while the 1-cube $vk_i$ is dual to the hyperplane
$\widehat H_i,$ hence $\widehat H$ and $\widehat H_i$ cross in $\recube X$. As a result, we conclude that  $\{\widehat H_i\}_{i=1}^d\cup\{\widehat H\}$ is a family of
pairwise-crossing hyperplanes of $\recube X$, contradicting the fact that $C$ is a maximal cube of $\recube X$. This contradiction shows that indeed $K\subseteq C$, thus proving that $C^*$ is nonempty. $\Box$


\medskip\noindent
{\bf Claim 2:} {\it $C^*$ contains a 1-cube dual to each hyperplane of $\recube X$ that crosses $C$.}

\medskip\noindent
{\bf Proof of Claim~2:} Consider any hyperplane $\widehat{H_i}$ crossing the cube $C$.  If $\widehat{H_i}$ crosses $K,$ then necessarily $\widehat{H_i}$ crosses $C^*$ and we are done. So, suppose that $\widehat{H_i}$ is disjoint from $K$. Since $K$ is contained in the carrier of $H_i=\widehat{H_i}\cap {\mathbf X},$ any 0-cube $k$ of $K$ belongs to a 1-cube $kk'$ of $\mathbf X$ dual to $H_i$. On the other hand, since $k\in K\subseteq C$ and $\widehat{H_i}$ crosses the cube $C,$ necessarily there exists a 1-cube $kk''$ of $C$ dual to $\widehat{H_i}$. Since the 1-cube $kk'$ is also dual to  $\widehat{H_i}$, we conclude that $k'=k'',$ i.e., $k'\in C\cap {\mathbf X}=C^*,$ whence $kk'$ is a 1-cube of $C^*$ dual to $\widehat{H_i}$. $\Box$



\medskip

\noindent{\textbf{Bounds on the maximum degree:}} Next we will show that the degree of any 0-cube $x$ of ${\recube X}$ is  bounded by $\Delta^2+\Delta$. First suppose that $x\in\mathbf X$.  There are at most $\Delta$ 0-cubes in $\mathbf X$ adjacent to $x$.  If $y\in\recube X-\mathbf X$ is a 0-cube adjacent to $x$ that was added to $\mathbf X$ during recubulation, then by construction, there is a path $[x,y,z]$ of length 2 in $\recube X$ such that $z\in\mathbf X^0$ and $P=[x,y,z]$ is a concatenation of two 1-cubes lying on the boundary  of a 2-cube $s\subset\recube X$ that was added during recubulation.  Let $Q=[x,w,z]$ be another path of length 2 of $s$.   For each $y\in\recube X-\mathbf X$ adjacent to $x$, there is thus a 0-cube $z=z(y)\in\mathbf X$ at distance 2 from $x$ such that, for some $w(y)\in\mathbf X$ adjacent to $x$, the 4-cycle $[x,w(y),z(y),y,x]$ bounds a 2-cube in $\recube X$ that does not appear in $\mathbf X$.  Moreover, each path $[x,y',z(y)]$ with $y'\in\recube X-\mathbf X$ lies on the boundary of a 2-cube $s'$ with the same dual hyperplanes as $s$, so since the hyperplanes dual to 1-cubes incident to $x$ are all distinct, $y'=y$.  Hence the assignment $y\mapsto z(y)$ is injective, and the degree of $x$ is thus bounded by the number of 0-cubes of $\mathbf X$ at distance 1 or 2 from $x$, i.e. by $\Delta^2+\Delta$.

Now suppose that $x\in \recube X-\mathbf X$ and let $\mathcal C$ be the set of all maximal cubes of $\recube X$ containing $x$.

\medskip\noindent
{\bf Claim 3:} {\it For each $C$ in $\mathcal C$, the intersection $C^*$ of $C$ and $\mathbf X$ is non-empty and convex in $\mathbf X.$}

\medskip\noindent
{\bf Proof of Claim~3:} It was shown in Claim 1 that $C^*\neq\emptyset$.  Now, since each cube $C$ of $\recube X$ is convex (namely, the set of 0-cubes of $C$ is convex in the graph $G({\recube X})$ and $G({\bf X})$ is an isometric subgraph of $G({\recube X})$, necessarily the set of vertices of $C^*$ is convex in $G({\bf X})$, hence $C^*$ is convex in $\bf X$. $\Box$

\medskip\noindent
{\bf Claim 4:} {\it For all $C_1, C_2$ in $\mathcal C$, the intersection of $C^*_1$ and $C^*_2$ is nonempty.}

\medskip\noindent
{\bf Proof of Claim~4:}  Suppose by way of contradiction that $C^*_1\cap C^*_2=\emptyset$. Since by Claim 3, $C^*_1$ and $C^*_2$ are convex in $\mathbf X$ and median graphs satisfy the Kakutani separation property (see~\cite{vandeVel_book}, Chapter~I.3), there exists a hyperplane $H$ of $\mathbf X$ separating $C^*_1$ from $C^*_2$. This hyperplane cannot cross either of the subcomplexes $C^*_1$ or $C^*_2$.  On the other hand, since $C_1$ and $C_2$ both contain $x$, the hyperplane $H$ extends to a hyperplane of $\recube X$ that crosses $C_1$ or $C_2$ and, as proved in Claim 2, $H$ is dual to a 1-cube of $C^*_1$ or $C^*_2$.  This is a contradiction. $\Box$

\smallskip
From Claims 3 and 4, and the Helly property, it follows that the intersection of all of the $C^*$ is nonempty as $C$ varies in $\mathcal C$, since $\mathcal C$ is finite.  In other words, the intersection of all maximal cubes of $\recube X$ that contain $x$ contains a 0-cube $y$ of $\mathbf X$.  Hence the degree in $\recube X$ of $x$ is bounded by the degree of $y$ in $\recube X$, which was shown before to be at most $\Delta^2 + \Delta.$ This complete the proof of the assertion (3) of Proposition \ref{prop:recube}.
\end{proof}

From Propositions~\ref{prop:recube} and \ref{prop:isometricembedding}, we now obtain:
%
%
%

\begin{corollary}\label{cor:noembeddinginproductoftrees}
For any CAT(0) cube complex $\bf X$, if there is an isometric embedding $\recube X\rightarrow\bf Y$, where $\bf Y$ is the product of $n$ trees, then $n\geq\chi(\Gamma({\bf X}))$.  In particular, if there exists a CAT(0) cube complex $\bf X$ with maximum degree $\Delta$ such that the chromatic number of $\Gamma({\bf X})$ is infinite, then there exists a CAT(0) cube complex, namely $\recube X$, such that the maximum degree of a 0-cube in $\recube X$ is at most $\Delta^2+\Delta$, and $\recube X$ does not embed isometrically in the product of finitely many trees.
\end{corollary}

The main ingredient in the remaining part of the proof of Theorem~\ref{theorem2} is the following result from~\cite{ChepoiNiceLabeling}.

\begin{proposition}\label{prop:pointedcontact}
There exists $\Delta<\infty$ such that, for each $n\geq 0$, there exists a finite, 4-dimensional pointed CAT(0) cube complex ${\bf X}'_n$ such that each 0-cube of ${\bf X}'_n$ has degree at most $\Delta$ and the contact graph of ${\bf X}'_n$ contains a subgraph $\Gamma_{n,\alpha}$ such that $\Gamma_{\#}({\bf X}'_n)\subseteq\Gamma_{n,\alpha}$ and $\Gamma_{n,\alpha}$ has clique number at most 5 and chromatic number greater than $n$.
\end{proposition}

\begin{remark}
The construction in~\cite{ChepoiNiceLabeling} shows that we can take $\Delta=8$.  The graph $\Gamma_{n,\alpha}$ is the pointed contact graph of ${\bf X}'_n$ pointed at a particular 0-cube $\alpha_n$.
\end{remark}

The construction of ${\bf X}'_n$ relies on an example due to Burling~\cite{Burling}.  A rigorous proof of Proposition~\ref{prop:pointedcontact}, using Burling's construction, is given in~\cite{ChepoiNiceLabeling}; here we summarize the basic notions and provide a sketch of the proof.

A \emph{(3-dimensional) box} is a closed parallelepiped in $\reals^3$ whose edges are parallel to the coordinate axes.  Given a (finite) collection $\mathcal B$ of boxes, let $\Omega(\mathcal B)$ denote the intersection graph of $\mathcal B$, and define $\omega(\mathcal B)$ to be the clique number of $\mathcal B$ and $\chi(\mathcal B)$ to be the chromatic number of $\Omega(\mathcal B)$.  As described in the survey of Gy\'{a}rf\'{a}s~\cite{Gyarfas}, Burling constructed for each $n$ a collection $\mathcal B_n$ of boxes such that $\omega(\mathcal B_n)=2$ and $\chi(\mathcal B_n)>n$ for all $n\geq 0$.

Following~\cite{ChepoiNiceLabeling}, Construction~\ref{cons:boxes} takes as its input a suitably-defined collection $\mathcal B$ of axis-parallel boxes and returns a 4-dimensional pointed CAT(0) cube complex ${\bf X}''$ whose pointed contact graph has clique number at most 5 and chromatic number greater than $n$.

\begin{construction}\label{cons:boxes}
We first define a certain cube complex $K$ arising from a box and an associated family of planes in $\reals^3$, then produce from $K$ a CAT(0) cube complex $\widetilde K$ of dimension 4 whose pointed contact graph has the desired colouring properties.  We then apply this construction in the context of Burling's example to build the cube complexes ${\bf X}'_n$.

\textbf{The box complex:}  Let $B_0$ be a box with one corner at the origin in $\reals^3$.  Any family of hyperplanes in $\reals^3$ (i.e. 2-dimensional, axis-parallel affine subspaces) that cross the interior of $B_0$ partitions $B_0$ into a family of boxes and thus defines a ``box complex'' realized by $B_0$.  A cell of $B_0$ is an ``elementary box''.  By rescaling each of the constituent elementary boxes of $B_0$ so that the sides have length 1, we obtain a CAT(0) cube complex $K$, defined by $B_0$ and the initial family of hyperplanes.  
Note that $K$ decomposes as a product, but is not the cube complex dual to the wallspace whose underlying set is $B_0$ and whose walls are the initial 2-dimensional hyperplanes.  Indeed, each maximal family of $m$ parallel 2-planes determines $m+1$ walls.  Instead, these 2-planes are part of the 2-skeleton of $K$, and each is parallel to a hyperplane.  Note that $K$ decomposes as a product of at most three 1-dimensional CAT(0) cube complexes and hence has dimension at most 3.  The dimension of $K$ is equal to the number of coordinate planes through the origin in $\reals^3$ that are parallel to at least one of the given 2-planes.

\textbf{Lifting to produce the pointed cube complex $\widetilde K$:}  Let $\mathcal B$ be a finite family of boxes.  Then there exists a box $B_0$ whose interior contains each $B\in\mathcal B$, and the family $\mathcal B$ induces a CAT(0) cubical structure on $B_0$.  Indeed, each of the six 2-planes determining $B\in\mathcal B$ becomes part of the cubical structure on $B_0$ constructed above.  For each $B_i\in\mathcal B$, let $K^i$ be the cubical subcomplex of $K$ consisting of the cubes arising as cells of $B_0$ lying in $B_i$.  Then $K^i$ is a product of (at most) three subdivided intervals and is a convex subcomplex of $K$.

$\widetilde K$ is defined as a subset of $\reals^{m+3}$, where $m=|\mathcal B|$.  First, let $B_0$ lie in the subspace $\reals^3\subset\reals^{m+3}$ consisting of points of the form $(0,0,\ldots,a,b,c)$ with $a,b,c\geq 0$.  Then for each $B_i\in\mathcal B$, define $0\leq a'_i<a''_i,\,0\leq b'_i<b''_i,\,0\leq c'_i<c''_i$ to be numbers such that $B_i$ consists of points of the form $p=(0,0,\ldots,a,b,c)$ with $a'_i\leq a\leq a''_i\,b'_i\leq b\leq b''_i,\,c'_i\leq c\leq c''_i$.  For $1\leq i\leq m$, let $s_i$ be the unit segment of the $i^{th}$ coordinate axis of $\reals^{m+3}$ and let $\widetilde B_i=s_i\times B_i$, so that $\widetilde B_i$ consists of points of the form $p=(p_1,p_2,\ldots,p_m,a,b,c)$, where $(a,b,c)\in[a'_i,a''_i]\times[b'_i,b''_i]\times[c'_i,c''_i]$ and $p_j=\delta_{ij}\cdot[0,1]$.  This gives rise to a box hypergraph $\widetilde{\mathcal B}$ in $\reals^{m+3}$ consisting of the boxes $\widetilde B_i$ with $1\leq i\leq m$.

Note that each elementary box $C$ of $K$ gives rise to a 4-cube $\widetilde C_i$ in $\reals^{m+3}$, isomorphic to $C\times s_i$, for each 3-dimensional box $B_i$ in $K$ that contains $C$.  $\widetilde C_i$ is a \emph{lifted elementary box}.  Let $\widetilde K'$ be the 4-dimensional box complex consisting of all lifted elementary boxes $\widetilde C_i$, together with all of the elementary boxes corresponding to 3-cubes of $K$.  Let $\widetilde K$ be the cube complex obtained by rescaling the edges of each elementary box so that they have length 1.  Note that $K\subset\widetilde K$.

In Lemma~1 of~\cite{ChepoiNiceLabeling}, it is shown that the 1-skeleton of $\widetilde K$ is a median graph, and therefore that $\widetilde K$ is a CAT(0) cube complex, by taking gated amalgams of the 1-skeleta of the complexes $K^i$, which are themselves median graphs.

We now let $\alpha$ be the corner of $B_0$ corresponding to the origin in $\reals^3\subset\reals^{m+3}$ and point $\widetilde K$ from $\alpha$, i.e. orient each 1-cube of $\widetilde K$ away from the halfspace of its dual hyperplane that contains $\alpha$.  Let $\Gamma_{\alpha}$ be the corresponding pointed contact graph.

\textbf{The pointed contact graph $\Gamma_{\alpha}$:}  Lemma~4 of~\cite{ChepoiNiceLabeling} states that, if $\omega=\omega(\mathcal B)$ is the clique number of the intersection graph of the family $\mathcal B$, then the clique number of the contact graph $\Gamma(\widetilde K)$ (i.e. the maximum cardinality of a 0-cube of $\widetilde K$) is at most $\omega+6$, and the maximum number of 1-cubes oriented outward from a 0-cube of $\widetilde K$, i.e. the clique number of $\Gamma_{\alpha}$, is exactly $\omega +3$.  On the other hand, since $\Gamma_{\alpha}$ contains the intersection graph of $\mathcal B$, we have that $\chi(\Gamma_{\alpha})\geq\chi(\mathcal B)$.

\textbf{Application to Burling's boxes:}  For each $n>0$, let $\mathcal B_n$ be a finite family of boxes such that $\omega(\mathcal B_n)=2$ and $\chi(\mathcal B_n)>n$.  Then since $\mathcal B_n$ is finite, there exists a box $B_{0,n}$ that contains each box $B\in\mathcal B_n$.  For each box $B\in\mathcal B_n$, there are six 2-planes crossing $B_{0,n}$ that together pass through the eight corners of $B$ and determine $B$ as a box in $\reals^3$.  The set of all such 2-planes determines a CAT(0) cube complex from $B_{0,n}$ as above, denoted $K_n$.  Let ${\bf X}'_n$ be the 4-dimensional CAT(0) cube complex constructed from $K_n$ as above.  Then ${\bf X}'_n$ has maximal degree $\Delta=\omega(\mathcal B_n)+6=8$.

Moreover, let $\alpha_n$ and $\beta_n$ be a pair of opposite corners of $B_{0,n}$.  Then $\alpha_n$ and $\beta_n$ are 0-cubes of $K_n$ and lift to distinct 0-cubes of ${\bf X}'_n$; denote these lifts also by $\alpha_n$ and $\beta_n$.  Letting ${\bf X}'_n$ be pointed from $\alpha_n$, we see that the pointed contact graph $\Gamma_{n,\alpha}$ of ${\bf X}'_n$ has clique number at most $\omega(\mathcal B_n)+3=5$ and chromatic number at least $\chi(\mathcal B_n)>n$.
\end{construction}

\begin{proof}[Proof of Theorem~\ref{theorem2}]
Let ${\bf X}'_n$ be a 4-dimensional CAT(0) cube complex given by Proposition~\ref{prop:pointedcontact}, so that the degree of 0-cubes in ${\bf X}'_n$ is at most $8$ and the clique number of the pointed contact graph $\Gamma_{n,\alpha}$ is at most 5.  Note that $\Gamma_{\#}({\bf X}'_n)\subseteq\Gamma_{n,\alpha}\subseteq\Gamma({\bf X}'_n)$.  Hence, by Proposition~\ref{prop:recube}, there exists a CAT(0) cube complex ${\bf X}_n$, obtained from ${\bf X}'_n$ by applying recubulation to the subgraph $\Gamma_{n,\alpha}$ of the contact graph, such that the hyperplanes of ${\bf X}_n$ correspond bijectively to those of ${\bf X}'_n$, and any $d$-cube in ${\bf X}_n$ corresponds to a $d$-clique in $\Gamma_{n,\alpha}$.  Hence $d\leq 5$ for each $d$-cube, i.e. $\dimension{\bf X}_n\leq 5$, and this bound is realized for some $n$.  By Proposition~\ref{prop:recube}(ii), $\Gamma_{\#}({\bf X}_n)=\Gamma_{n,\alpha}$, and hence
\[\chi(\Gamma_{\#}({\bf X}_n))=\chi(\Gamma_{n,\alpha})>n.\]
The cube complex ${\bf X}_n$ is thus finite, and any isometric embedding of ${\bf X}_n$ into the Cartesian product of trees requires at least $n+1$ trees.  Moreover, by Proposition~\ref{prop:recube}(iii) the clique number of the contact graph of ${\bf X}_n$ is at most $\Delta^2+\Delta=8^2+8=72$ for each $n$.

By Proposition~\ref{prop:recube}, there is an isometric embedding ${\bf X}'_n\rightarrow{\bf X}_n$, and hence ${\bf X}_n$ contains distinct 0-cubes $a_n$ and $b_n$ that are the images of $\alpha_n$ and $\beta_n$ respectively.  Let $\bf X$ be formed from $\bigsqcup_{n\geq 0}{\bf X}_n$ by identifying $b_n$ with $a_{n+1}$ for each $n$.  Since each 0-cube of $\bf X$ is contained in at most 2 of the ``blocks'' ${\bf X}_n$, the clique number of $\Gamma({\bf X})$ is at most $\Delta^2+\Delta$.  Indeed, each of $\beta_n$ and $\alpha_{n+1}$ has degree at most 3, and we are taking $\Delta=8$.  On the other hand, by construction,
\[\xing X=\bigsqcup_{n\geq 0}\Gamma_{\#}({\bf X}_n)\]
and hence
\[\chi(\xing X)\geq\chi(\Gamma_{\#}({\bf X}_n))>n\]
for all $n$, i.e. the chromatic number of the crossing graph of $\bf X$ is infinite.  Thus $\bf X$ is a 5-dimensional, uniformly locally finite CAT(0) cube complex that does not admit an isometric embedding in the Cartesian product of finitely many trees.
\end{proof}


Theorem~\ref{theorem1} states that a uniformly locally finite CAT(0) cube complex of dimension at most 2 embeds isometrically in the product of finitely many trees, while Theorem~\ref{theorem2} asserts the existence of a 5-dimensional uniformly locally finite CAT(0) cube complex not admitting such an embedding.  It is thus natural to ask about the situation in dimensions 3 and 4:

\medskip\noindent
{\bf Question 4.}  {\it Are 3-~or 4-dimensional CAT(0) cube complexes, whose 1-skeleta have uniformly bounded degree, embeddable  in the Cartesian product of finitely many trees?}

\medskip\noindent
If Question 4 has a positive answer in the 3-dimensional case, then to find this answer it seems quite likely that one has to be able to colour the imprints of a collection of hyperplanes on a given hyperplane $H$. This hyperplane $H$ is a 2-dimensional  CAT(0) cube complex and the imprints in $H$ are convex (gated) subcomplexes of $H.$ Therefore, as an warm-up to Question 4, one can ask for the following common generalization of Asplund-Gr\"unbaum's result for axis-parallel rectangles and of the fact that any family $\mathcal T$ of subtrees  of a tree can be coloured with $\omega({\mathcal T})$ colours:

\medskip\noindent
{\bf Question 5.}  {\it Let $\mathcal C$ be a collection of convex subcomplexes of a 2-dimensional CAT(0) cube complex. Is $\mathcal C$ $\chi$-bounded?}

\section*{Acknowledgements}
We would like to thank anonymous referees for numerous useful comments.
V.C. would like to acknowledge Cornelia Dru\c{t}u  for (implicitly) pointing to the first version of the paper \cite{HagenQuasiArb} of M.H. and thus making this collaboration possible. He also would like to acknowledge Hans-J\"urgen Bandelt  and David Eppstein for exchanges on Question 1 during the work on the papers \cite{BandeltChepoiEppstein_squaregraph,BandeltChepoiEppstein}. The work of V.C. was supported in part by the ANR grant GGAA (grant ANR-10-BLAN 0116).  M.H. wishes to thanks Michah Sageev and Dani Wise for helpful discussions on the hyperplane-colouring problem.  The work of M.H. was supported by a McGill University Schulich Graduate Fellowship.  Both authors thank  Jason Behrstock for a helpful correction.


\newcommand{\etalchar}[1]{$^{#1}$}

\newpage

\appendix
\renewcommand{\sectionname}{}
\setcounter{figure}{0} \renewcommand{\thefigure}{A}
\begin{center}
\Large{Appendix: correction to \emph{On embeddings of CAT(0) cube complexes into products of trees via colouring their hyperplanes}}
\end{center}

\medskip

\begin{center}
\textsc{Victor Chepoi and Mark Hagen}

\smallskip
\today   
\end{center}

\newcommand{\incoming}{\stackrel{\prec}{\to}}

Theorem 1 of \cite{ChepoiHagen} is incorrect in the stated generality due to an error in \cite[Lem. 12]{ChepoiHagen}. We correct this under an extra hypothesis on $\mathbf X$ by proving:

\begin{thmi} \label{thm:main} Let ${\bf X}$ be a 2-dimensional CAT(0) cube complex whose vertices have degree bounded by $\Delta<\infty$.  Suppose that no vertex of $\bf X$ has link containing a $5$--cycle, or, equivalently, the crossing graph $\Gamma_\#(\mathbf X)$ has no $5$--cycle. Then:  
\begin{enumerate}
    \item The contact graph $\Gamma(\mathbf X)$, and hence $\Gamma_{\#}(\mathbf X)$, has chromatic number at most $\epsilon(\Delta):=M\Delta^{26}$, where $M<\infty$ is a constant independent of $\mathbf X$. 

    \item Hence the $1$--skeleton of $\mathbf X$ admits an isometric embedding into the Cartesian product of at most $\epsilon(\Delta)$ trees.

    \item Any event structure of (out)degree $\Delta_0$, whose domain is a CAT(0) square complex with no forbidden diagram, admits a nice labelling with at most $\epsilon(\Delta_0+2)$ labels.
\end{enumerate}
\end{thmi}

The only difference between Theorem \ref{thm:main} and \cite[Thm. 1]{ChepoiHagen} is the addition of the hypothesis about $5$--cycles. Vertex links are permitted to have $4$--cycles, and cycles of any size larger than $5$, so no extant result that we know of implies Theorem \ref{thm:main}.  In particular, Theorem \ref{thm:main} does not impose any constraint on the size of planar grids in $\mathbf X$.

We are very grateful to James Davies and Harry Petyt for showing us examples that led us to find the error in our original argument; these will be in their forthcoming paper \cite{DaviesPetyt}.  We also thank them for a discussion about Theorem \ref{thm:main} and for reading this correction.

\section{Summary of corrections}\label{sec:issue}
The error in \cite{ChepoiHagen} is in the proof of \cite[Lem. 12]{ChepoiHagen}, which is false as stated; Figure \ref{fig:counterexample} shows a counterexample. We correct the lemma below, using the extra hypothesis.

\begin{figure}[h]\label{fig:counterexample}
\begin{overpic}[width=0.75\textwidth]{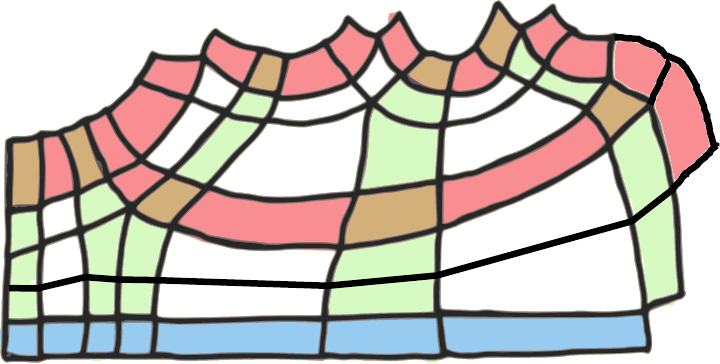}
\put(18,34){$H_1$}
\put(31,46){$H_2$}
\put(56,48){$H_3$}
\put(82,47){$H_4$}
\put(33,19){$H_0$}
\put(33,2){$U$}
\end{overpic}
\caption{The paths $H_1\incoming H_2\incoming H_3\incoming H_4$ and $H_1\incoming H_0\incoming H_4$ form a $5$--cycle in the graph $\Upsilon_0(U)$.  The root $b^*$ in $U$ is positioned on the far left of the picture.  The fathers contact $U$ in the following order, from left to right: $f(H_1),f(H_0),f(H_2),f(H_3),f(H_4)$.  Each $H_i$ and $f(H_i)$ has its carrier shaded.}
\end{figure}

As explained in the proof of Theorem \ref{thm:main} below, the only difference between the proof of \cite[Thm. 1]{ChepoiHagen} given in that paper, and the corrected argument given here, concerns statements in \cite[Sec. 6.4]{ChepoiHagen}.  The changes are:
\begin{itemize}
    \item Proposition \ref{prop:bipartite} replaces Proposition 11 and Lemma 12 of \cite{ChepoiHagen}.
    \item Lemma 8 of \cite{ChepoiHagen} is correct as written, but in the present argument can be replaced by the weaker Lemma \ref{lem:no-father-separation} below.
    \item Lemma 9 of \cite{ChepoiHagen} also holds as written, but is now subsumed into Proposition \ref{prop:bipartite}.
\end{itemize}

The other results in \cite{ChepoiHagen} are unaffected.  In particular, Lemmas 10 and 11 of \cite{ChepoiHagen} are easy observations requiring no correction.

\section{Proof of the corrected theorem}\label{sec:proof}
Throughout, $\mathbf X$ is a $2$--dimensional CAT(0) cube complex whose vertices have maximum degree $\Delta<\infty$. The notation and definitions used here are exactly as in \cite{ChepoiHagen}.  In particular, the contact graph $\Gamma(\bf X)$ has subgraphs $\Upsilon(U),\Upsilon_i(U),\ i\in\{0,1,2\}$ defined as in \cite{ChepoiHagen}.

First we prove the equivalence of the hypotheses on $5$--cycles in Theorem \ref{thm:main}:

\begin{lem}[$5$--cycles in crossing graphs]\label{lem:5-cycle-implies-5-cycle}
The crossing graph $\Gamma_\#(\mathbf X)$ contains a $5$--cycle if and only if there is a $0$--cube $v\in\mathbf X$ such that $\link_{\mathbf X}(v)$ contains a $5$--cycle.
\end{lem}

\begin{proof}
Suppose that $x\in\mathbf X$ is a $0$--cube with incident $1$--cubes $e_0,\ldots,e_4$ such that $e_i,e_{i+1}$ span a square $s_i$ for all $i\in\integers/5\integers$.  Let $H_i$ be the hyperplane dual to $e_i$.  Then, for each $i$, the hyperplanes $H_i,H_{i+1}$ define an edge of $\Gamma_\#(\mathbf X)$ because they intersect in the centre of $s_i$.

Conversely, suppose that $H_0,\ldots,H_4$ are distinct hyperplanes such that $H_i$ crosses $H_{i+1}$ for all $i\in\integers/5\integers$.  Since $\dimension\mathbf X\leq 2$, the hyperplanes $H_i,H_j$ cannot cross unless $j=i\pm 1$.  Hence, for each $i$, there is a halfspace $H_i^+$ of $\mathbf X$ bounded by $H_i$ and not containing $H_{i+2}\cup H_{i-2}$.  Let $\alpha_i$ be the unique geodesic in the tree $N(H_i)\cap H_i^+$ that has minimum possible length among geodesics joining $N(H_i)\cap N(H_{i-1})$ to $N(H_i)\cap N(H_{i+1})$.  Then $\alpha_i$ terminates at the initial point of $\alpha_{i+1}$, so there is an immersed cycle $\alpha_0\alpha_1\alpha_2\alpha_3\alpha_4$ in $\mathbf X$.  Let $D\to\mathbf X$ be a minimal area disc diagram bounded by this cycle.  Boundary vertices in $D$ have degree $3$ in $D$, except for the five vertices that are endpoints of the paths $\alpha_i$, which are corners of squares $s_0,\ldots,s_4$ of $D$ such that each $s_j$ has exactly two edges on the boundary path of $D$. The combinatorial Gauss-Bonnet theorem \cite{BallmannBuyalo} provides a vertex of $D$ whose link in $D$ is a $5$--cycle.  It thus suffices to show that the map $D\to\mathbf X$ is injective on vertex links.  This holds because crossing dual curves map to distinct hyperplanes \cite[Prop. 3]{ChepoiHagen} and osculating dual curves in $D$ cross a common $\alpha_i$, and hence map to distinct hyperplanes since $\alpha_i$ is a geodesic.
\end{proof}

Next we prove Theorem \ref{thm:main}.

\begin{proof}[Proof of Theorem \ref{thm:main}]
The proof proceeds exactly as in \cite[Sec. 7]{ChepoiHagen}, with one difference.  In \cite{ChepoiHagen}, we colour $\Upsilon(U)$ in $2\Delta^2$ colours, which is the content of Proposition 12 of \emph{loc. cit.}, whose proof is just an application of Propositions 9, 10, and 11.  Propositions 9 and 10 apply with no change; Proposition 11 says that $\Upsilon_0(U)$ is bipartite, and is the only place where the incorrect Lemma 12 is applied.  Under the present hypotheses, we can substitute Proposition \ref{prop:bipartite} for \cite[Prop. 11]{ChepoiHagen}.  Once $\Upsilon_0(U)$ is bipartite, the rest of the argument goes through verbatim, and we conclude that $\chi(\Gamma(\mathbf X))\leq \epsilon(\Delta)$ provided $\Gamma_{\#}(\mathbf X)$ has no $5$--cycles, or, equivalently (by Lemma \ref{lem:5-cycle-implies-5-cycle}), $\mathbf X$ has no vertex whose link contains a $5$--cycle.
\end{proof}

Therefore it remains to prove:

\begin{prop}\label{prop:bipartite}
Suppose that $\Gamma_{\#}(\mathbf X)$ has no $5$--cycle.  Then $\Upsilon_0(U)$ is bipartite.
\end{prop}

Before the proof, we recall the setting of \cite[Sec. 6]{ChepoiHagen}: $U$ is a fixed hyperplane and $\mathcal R(U)$ is the set of hyperplanes $H$ such that the grandfather $f^2(H)=U$.  Recall that $\Upsilon(U)$ is the subgraph of the contact graph $\Gamma(\mathbf X)$ whose vertices are hyperplanes in $\mathcal R(U)$, with $H,H'$ joined by an edge of $\Upsilon(U)$ if and only if all of the following hold:
\begin{enumerate}
    \item $H\coll H'$, and
    \item $f(H)\neq f(H')$, and
    \item $f(H)$ does not contact $f(H')$.
\end{enumerate}

Hyperplanes $H,H'$ that are adjacent in $\Upsilon(U)$ are adjacent in $\Upsilon_0(U)$ if and only if the following conditions hold, up to relabelling $H$ and $H'$ (these are from \cite[Sec. 6.4]{ChepoiHagen}):
\begin{enumerate}
    \item The \emph{separating osculators} (see \cite[Sec. 6.2]{ChepoiHagen}) $S(H)$ and $S(H')$ are distinct;
    \item $H'\prec H$ and $S(H)\neq H'$.
    \item $H'$ is not a \emph{father osculator} of $H$, which means that either $f(H)$ does not osculate with $H'$, or $f(H)=S(H)$.
\end{enumerate}
The partial order $\prec$ on $\mathcal R(U)$ is as defined in \cite[Sec. 5]{ChepoiHagen}.

\begin{defn}\label{defn:precnot}
We write $H'\incoming H$ to mean that $H'$ is adjacent to $H$ in $\Upsilon_0(U)$ and $H'\prec H$.  In this case, $H'$ is an \emph{incoming neighbour} of $H$ in the sense of \cite[Defn. 8]{ChepoiHagen}.
\end{defn}

\begin{lem}\label{lem:no-separation}
If $H'\incoming H$, then neither $H$ nor $H'$ separates the other from $U$.
\end{lem}

\begin{proof}
Since $H'\prec H$, Lemma 10 of \cite{ChepoiHagen} implies $H$ cannot separate $H'$ from $U$.  If $H'$ separates $H$ from $U$, then $H$ and $H'$ osculate.  Moreover, $H'$ cannot contact $U$, since $f^2(H')=U$, so by Lemma 7 of \cite{ChepoiHagen}, $H'=S(H)$.  This implies that $H'H$ is an edge of $\Upsilon_1(U)$, contradicting that $H'H$ is an edge of $\Upsilon_0(U)$.
\end{proof}

The next lemma handles the special case of \cite[Lem. 8]{ChepoiHagen} needed below.

\begin{lem}\label{lem:no-father-separation}
If $H'\incoming H$, then $f(H')\neq S(H')$.  Moreover, if there exists $H''$ such that $H\incoming H''$, then $S(H')$ crosses $H$.
\end{lem}

\begin{proof}
Suppose to the contrary that $f(H')=S(H')=:V'$, so $V'$ separates $H'$ from $U$ and osculates with $H'$.  If $V'$ contacts $H$, then $V'$ is a potential father of $H$ whose root in $U$ precedes that of $f(H)$, contradicting that the fathers of $H$ and $H'$ are distinct and $H'\prec H$.  Hence $V'$ cannot contact $H$.  Since $N(H')\cap N(H)\neq \emptyset$, and $V'$ separates $H'$ from $U$, it therefore follows that $V'$ separates $H$ from $U$.  Since $f(H)$ contacts $H$ and $U$, we conclude that $V'$ crosses $f(H)$, which is a contradiction since $f(H')=V'$ but adjacency of $H,H'$ in $\Upsilon(U)$ requires that $f(H),f(H')$ do not contact.  This shows that $f(H')\neq S(H')$. 

Suppose $H\incoming H''$ for some $H''$. Then $f(H)\neq S(H)$ by the first assertion.  We claim that $S(H')$ crosses $H$.  If not, then since $H\coll H'$, the hyperplane $S(H')$ separates $H\cup H'$ from $U$. Now we use that $H'$ is not a father-osculator of $H$.  We distinguish  two cases.

First suppose that $H'$ does not contact $f(H)$.  Then some hyperplane $W$ separates $f(H)$ from $H'$ and therefore $W$ crosses $H$.  If $W$ does not cross $f(H')$, then it separates $H'\cup f(H')$ from $f(H)$, and hence crosses $U$ between $f(H')$ and $f(H)$.  Hence $W$ is a better potential father for $H$ than $f(H)$, a contradiction with the definition of $f(H)$.  Thus $W$ separates $H'$ from $U$ and crosses $f(H')$.  The hyperplanes $W,f(H'),S(H')$ are distinct since we are assuming $S(H')$ does not cross $H$.  Also, these three hyperplanes cannot all cross, so since $S(H')$ crosses $f(H')$ by \cite[Lem. 7]{ChepoiHagen}, $S(H')$ and $W$ must be disjoint.  Now, since $S(H')$ separates $H$ from $U$, it must cross $f(H)$.  Since $W$ separates $H'$ from $f(H)$, and does not cross $S(H')$, we see that $W$ separates $H'$ from $S(H')$, contradicting that $S(H')\coll H'$, by definition.  We conclude that $S(H')$ crosses $H$ as required.

Otherwise, if  $H'$ and $f(H)$ contact, then, since $H'$ is not a father-osculator of $H$ and $f(H)\neq S(H)$, necessarily $f(H)$ crosses $H'$. Hence $f(H)$ also crosses $S(H')$.  We first show that $H$ and $S(H')$ contact.  Indeed, otherwise some hyperplane $W$ separates $H$ from $S(H')$.  Since $S(H')$ and $H$ both contact $H'$, the hyperplanes $W$ and $H'$ must cross.  Since $f(H)$ crosses $S(H')$, and $W$ cannot separate $f(H)$ from $H$, the hyperplane $W$ must cross $f(H)$.  So, $W,H',f(H)$ pairwise cross, contradicting $2$--dimensionality.  Hence $S(H')$ must contact $H$, since 
they cannot coincide by Lemma \ref{lem:no-separation}.  If they cross, we are done, so assume that $S(H')$ and $H$ osculate.   Now, $S(H')$ separates $H$ and $H'$ from $U$. Since $H'H$ is an edge in $\Upsilon_0(U)$, we have $S(H)\neq S(H')$.  Hence, by \cite[Lem. 7]{ChepoiHagen}, $S(H)$ and $S(H')$ cross.  Now, $f(H)$ crosses $S(H')$, as noted above, and it crosses or coincides with any hyperplane separating $U$ from $H$.  Thus $f(H)=S(H)$, a contradiction, so $S(H')$ must cross $H$, as claimed.
\end{proof}

\begin{proof}[Proof of Proposition \ref{prop:bipartite}]
Suppose that $H',H,H''\in\mathcal R(U)$ satisfy $H'\incoming H\incoming H''$ (in  \cite[Sec. 6]{ChepoiHagen} such a hyperplane $H$ is called \emph{normal}).  Then \cite[Prop. 8]{ChepoiHagen} provides distinct hyperplanes $W'$ and $W$ that both cross $U$, such that $W'$ crosses $H'$, $W$ crosses $H$, and $W'$ does not contact $H$ and $W$ does not contact $H''$.  By Lemma \ref{lem:no-father-separation}, $S(H')$ crosses $H$, and since it separates $H'$ from $U$, it must cross $W'$.  The hyperplanes $W',S(H'),H,W,U$ are distinct and therefore form a $5$--cycle in $\Gamma_\#(\mathbf X)$, contradicting our hypothesis. Consequently, the graph $\Upsilon_0(U)$ has no directed path with three distinct vertices and no normal hyperplanes.

Let $C=(H_0,H_1,\ldots,H_n)$ be a cycle in $\Upsilon_0(U)$. If $n=2$, then, up to relabelling, $H_1\incoming H_0\incoming H_2$, so $H_0$ is normal, a contradiction.  If $n>3$, then \cite[Lem. 11]{ChepoiHagen} and the absence of normal hyperplanes imply that $C$ has even length.  Hence $\Upsilon_0(U)$ is bipartite.
\end{proof}

\end{document}